\documentclass[reqno]{amsart}

\def\RaymondContactDetails{
  \author[van Bommel]{Raymond van Bommel}
  \address{Raymond van Bommel, Department of Mathematics, Massachusetts
    Institute of Technology, 77 Massachusetts Avenue, Cambridge, MA 02139,
    USA}
  \email{bommel@mit.edu}
}

\def\JordanContactDetails{
  \author[Docking]{Jordan Docking}
  \address{Jordan Docking, Department of Mathematics, University College London, London WC1H 0AY, UK}
  \email{jordan.docking.18@ucl.ac.uk}
}

\def\VladimirContactDetails{
  \author[Dokchitser]{Vladimir Dokchitser}
  \address{Vladimir Dokchitser, University College London, London WC1H 0AY, UK}
  \email{v.dokchitser@ucl.ac.uk}
}

\def\ReynaldContactDetails{
  \author[Lercier]{Reynald Lercier}
  \address{%
    Reynald Lercier,
    DGA \& Univ Rennes, %
    CNRS, IRMAR - UMR 6625, F-35000
    Rennes, %
    France. %
  }
  \email{reynald.lercier@m4x.org}
}

\def\ElisaContactDetails{
  \author[Lorenzo Garc\'ia]{Elisa Lorenzo Garc\'ia}
  \address{
    Elisa Lorenzo Garc\'ia,
    Institut de Math\'ematiques,  Universit\'e de Neuch\^atel, rue Emile-Argand 11, 2000, Neuch\^atel,
    Switzerland \&
    Univ Rennes, CNRS, IRMAR - UMR 6625, F-35000
    Rennes, %
    France. %
  }
  \email{elisa.lorenzo@unine.ch}
}

\usepackage[utf8]{inputenc}
\usepackage[T1]{fontenc}

\usepackage{ae}
\usepackage{lmodern}
\usepackage{bm}

\advance\oddsidemargin -1.5cm
\advance\evensidemargin -1.5cm
\advance\textwidth 2cm
\advance\textheight 0.3cm %

\usepackage{color}
\usepackage[dvipsnames]{xcolor}

\usepackage{amsmath, amsthm, amssymb, amsfonts}

\def\gettexliveversion#1(#2 #3 #4#5#6#7#8)#9\relax{#4#5#6#7}
\edef\texliveversion{\expandafter\gettexliveversion\pdftexbanner\relax}
\ifnum\texliveversion<2022      %
\DeclareRobustCommand{\SkipTocEntry}[9]{}
\else %
\DeclareRobustCommand{\SkipTocEntry}[5]{}
\fi %

\usepackage{tikz-cd}

\usepackage{ifthen}

\usepackage{tkz-graph}
\usepackage{semtkzX}

\counterwithin{figure}{section}
\counterwithin{table}{section}
\newtheorem{theorem}{Theorem}[section]
\newtheorem*{theorem*}{Theorem}

\newtheorem*{corollary*}{Corollary}
\newtheorem{conjecture}[theorem]{Conjecture}
\newtheorem*{conjecture*}{Conjecture}
\newtheorem{lemma}[theorem]{Lemma}

\newtheorem{proposition}[theorem]{Proposition}
\newtheorem*{proposition*}{Proposition}

\theoremstyle{definition}
\newtheorem{definition}[theorem]{Definition}
\newtheorem*{definition*}{Definition}
\newtheorem{example}[theorem]{Example}
\newtheorem*{example*}{Example}
\newtheorem{remark}[theorem]{Remark}
\newtheorem{notation}[theorem]{Notation}
\newtheorem*{notation*}{Notation}

\usepackage{booktabs}
\usepackage{array}
\usepackage{collcell}
\usepackage{multirow}
\usepackage{colortbl}

\definecolor{mygray}{gray}{0.92}

\usepackage{float}
\usepackage{subcaption}
\usepackage[all]{xy}

\usepackage{graphicx}

\usepackage{xspace}

\usepackage[style=alphabetic,citestyle=alphabetic,sorting=nyt,sortcites=true,autopunct=true,hyperref=true,abbreviate=true,backref=true,useprefix=false,giveninits=true,maxbibnames=99,maxcitenames=99,maxalphanames=4]{biblatex}
\AtEveryBibitem{ \clearfield{url} \clearfield{urldate} \clearfield{review} \clearfield{series} \clearfield{doi} \clearfield{isbn} \clearfield{issn} }
\defbibheading{bibempty}{}
\DeclareNolabel{\nolabel{\regexp{[\p{Z}]+}}}
\AtBeginBibliography{\small}

\definecolor{mylinkcolor}{rgb}{0.5,0.0,0.0}
\definecolor{myurlcolor}{rgb}{0.0,0.0,0.75}
\usepackage[
	colorlinks, urlcolor=myurlcolor,citecolor=myurlcolor,linkcolor=mylinkcolor,
        pdfusetitle,
]{hyperref}

\usepackage[ruled,vlined,boxed,linesnumbered,norelsize]{algorithm2e}

\usepackage{tikz}
\usetikzlibrary{shapes.geometric}
\usetikzlibrary{hobby}
\usetikzlibrary{positioning}
\usetikzlibrary{arrows}

\usetikzlibrary{external}
\tikzsetfigurename{main-figure}

\tikzset{
  g3lattice/.style={inner sep=1pt,norm/.style={red!50!blue},char/.style={blue!50!black},
    lin/.style={black!50}},cnj/.style={black!50,yshift=-2.5pt,left=-1pt of #1,scale=0.5,fill=white},
  sml/.style={scale=0.55},
  typ/.style={scale=1.0,inner sep=0.2em},
  lrg/.style={scale=0.9,inner sep=0.2em},
  fname/.style={scale=0.55},
  lin/.style={-,shorten <=-0.07em,shorten >=-0.07em},
  rem/.style={black!20,thin},
  lname/.style={scale=0.55,sloped,red,above=-0.07em,near end},
  every loop/.style={}
}

\input cyracc.def
\font\tencyr=wncyr10
\def\sha{\text{\tencyr\cyracc{Sh}}}

\DeclareMathOperator{\Jac}{Jac}

\DeclareMathOperator{\diag}{diag}

\def\Q{\mathbb{Q}}
\def\Z{\mathbb{Z}}

\def\F{\mathbb{F}}
\def\P{\mathbb{P}}
\def\O{\mathcal{O}}

\def\Gm2{\mathbb{G}_m^2}

\def\PP{\mathbb{P}}

\newcommand{\oA}{\mathtt{A}}
\newcommand{\oB}{\mathtt{B}}
\newcommand{\oC}{\mathtt{C}}
\newcommand{\oD}{\mathtt{D}}
\newcommand{\oE}{\mathtt{E}}
\newcommand{\oF}{\mathtt{F}}
\newcommand{\oG}{\mathtt{G}}
\newcommand{\oH}{\mathtt{H}}

\newcommand{\pt}{\textrm{pt}}
\newcommand{\pl}{\textrm{pln}}
\renewcommand{\ln}{\textrm{ln}}
\newcommand{\tc}{{\mathrm{tc}}}

\newcommand{\CAApts}{\oA\oB|\oC\oD\|\oE\oF|\oG\oH}

\newcommand{\PGL}{\textup{PGL}_3(\overline{K})}
\newcommand*{\MAGMA}{\textsc{magma} \cite{magma}\xspace}
\renewcommand{\mod}[1]{\,\,\, (\textup{mod } #1)}

\newcommand{\stbtype}[2]{$\substack{\scalebox{0.7}{#1}\\\mathtt{(#2)}}$}
\newcommand{\stbtypehyp}[2]{$\substack{\scalebox{0.7}{#1}\\\mathtt{(#2)_{_{\mathtt{H}}}}}$}
\newcommand{\StbType}[2]{$\substack{\scalebox{0.9}{#1}\\\mathtt{(#2)}}$}
\newcommand{\StbTypeHyp}[2]{$\substack{\scalebox{0.7}{#1}\\\mathtt{(#2)_{_{\mathtt{H}}}}}$}

\newcommand{\subsnamechoice}{merotropic}

\newcommand{\AG}{\mathbf{A}}
\newcommand{\CG}{\mathbf{C}}

\newcommand{\Gg}{\mathbf{G}}

\newcommand{\SG}{\mathbf{S}}

\newcommand{\ifcolored}{false}

\newcommand\Aone{\ifthenelse{\equal{\ifcolored}{true}}
  {\fcolorbox{white}{SkyBlue}}{\fcolorbox{white}{white}}}             %
\newcommand\Aoneptwo{\ifthenelse{\equal{\ifcolored}{true}}
  {\fcolorbox{white}{Cyan}}{\fcolorbox{white}{white}}}            %
\newcommand\Aonepthree{\ifthenelse{\equal{\ifcolored}{true}}
  {\fcolorbox{white}{NavyBlue}}{\fcolorbox{white}{white}}}      %
\newcommand\RAonepthree{\ifthenelse{\equal{\ifcolored}{true}}
  {\fcolorbox{white}{MidnightBlue}}{\fcolorbox{white}{white}}} %
\newcommand\RAonepfoura{\ifthenelse{\equal{\ifcolored}{true}}
  {\fcolorbox{white}{Blue}}{\fcolorbox{white}{white}}}         %
\newcommand\RAonepfourb{\ifthenelse{\equal{\ifcolored}{true}}
  {\fcolorbox{white}{Periwinkle}}{\fcolorbox{white}{white}}}   %
\newcommand\RAonepfive{\ifthenelse{\equal{\ifcolored}{true}}
  {\fcolorbox{white}{Orchid}}{\fcolorbox{white}{white}}}        %
\newcommand\RAonepsix{\ifthenelse{\equal{\ifcolored}{true}}
  {\fcolorbox{white}{Fuchsia}}{\fcolorbox{white}{white}}}        %

\newcommand\Atwo{\ifthenelse{\equal{\ifcolored}{true}}
  {\fcolorbox{white}{pink}}{\fcolorbox{white}{white}}}    %
\newcommand\AoneAtwo{\ifthenelse{\equal{\ifcolored}{true}}
  {\fcolorbox{white}{Lavender}}{\fcolorbox{white}{white}}} %
\newcommand\AoneptwoAtwo{\ifthenelse{\equal{\ifcolored}{true}}
  {\fcolorbox{white}{Magenta}}{\fcolorbox{white}{white}}} %
\newcommand\AonepthreeAtwo{\ifthenelse{\equal{\ifcolored}{true}}
  {\fcolorbox{white}{Red}}{\fcolorbox{white}{white}}} %

\newcommand\Atwoptwo{\ifthenelse{\equal{\ifcolored}{true}}
  {\fcolorbox{white}{Goldenrod}}{\fcolorbox{white}{white}}} %
\newcommand\AoneAtwoptwo{\ifthenelse{\equal{\ifcolored}{true}}
  {\fcolorbox{white}{Apricot}}{\fcolorbox{white}{white}}} %

\newcommand\Athree{\ifthenelse{\equal{\ifcolored}{true}}
  {\fcolorbox{white}{Green}}{\fcolorbox{white}{white}}} %
\newcommand\RAoneAthree{\ifthenelse{\equal{\ifcolored}{true}}
  {\fcolorbox{white}{Green}}{\fcolorbox{white}{white}}} %

\newcommand\Atwopthree{\ifthenelse{\equal{\ifcolored}{true}}
  {\fcolorbox{white}{orange}}{\fcolorbox{white}{white}}} %

\colorlet{coltypeI} {black}
\colorlet{coltypeII}{black}

\iftrue
\newcommand{\comm}[1]{#1}
\newcommand{\rey}[1]{{\color{Green} \noindent Reynald: #1}}

\newcommand{\reynald}[1]{\rey{#1}}
\newcommand{\raymond}[1]{{\color{blue} \noindent Raymond: #1}}
\newcommand{\jordan}[1]{{\color{purple} \noindent Jordan: #1}}
\newcommand{\elisa}[1]{{\color{red} \noindent Elisa: #1}}
\else
\newcommand{\comm}[1]{}
\newcommand{\rey}[1]{}\newcommand{\reynald}[1]{}
\newcommand{\raymond}[1]{}
\newcommand{\jordan}[1]{}
\newcommand{\elisa}[1]{}
\fi

\subjclass[2020]{}

\def\pts{
    \foreach \x in {0,...,7}
        \draw ({3*cos(45*\x + 22.5)}, {3*sin(45*\x + 22.5)}) circle (0.1cm);}

\def\specialpts{
    \foreach \x in {-25,25,95,145,215,265}
        \draw ({3*cos(\x)}, {3*sin(\x)}) circle (0.1cm);
    \draw (0.4,1.2) circle (0.1cm);
    \draw (0.4,-1.2) circle (0.1cm);}

\def\ptslabel{
     \node[draw,circle,scale = 0.5] (CircleNode) at ({3*cos(45*0 + 22.5)}, {3*sin(45*0 + 22.5)})  {C};
     \node[draw,circle,scale=0.5] (CircleNode) at ({3*cos(45*1 + 22.5)}, {3*sin(45*1 + 22.5)})  {B};
     \node[draw,circle,scale=0.5] (CircleNode) at ({3*cos(45*2 + 22.5)}, {3*sin(45*2 + 22.5)})  {A};
     \node[draw,circle,scale=0.5] (CircleNode) at ({3*cos(45*3 + 22.5)}, {3*sin(45*3 + 22.5)})  {H};
     \node[draw,circle,scale=0.5] (CircleNode) at ({3*cos(45*4 + 22.5)}, {3*sin(45*4 + 22.5)})  {G};
     \node[draw,circle,scale=0.5] (CircleNode) at ({3*cos(45*5 + 22.5)}, {3*sin(45*5 + 22.5)})  {F};
     \node[draw,circle,scale=0.5] (CircleNode) at ({3*cos(45*6 + 22.5)}, {3*sin(45*6 + 22.5)})  {E};
     \node[draw,circle,scale=0.5] (CircleNode) at ({3*cos(45*7 + 22.5)}, {3*sin(45*7 + 22.5)})  {D};
}

\def\ptslabelcross{
     \node[draw,circle,scale = 0.5] (CircleNode) at ({3*cos(45*0 + 22.5)}, {3*sin(45*0 + 22.5)})  {F};
     \node[draw,circle,scale=0.5] (CircleNode) at ({3*cos(45*1 + 22.5)}, {3*sin(45*1 + 22.5)})  {E};
     \node[draw,circle,scale=0.5] (CircleNode) at ({3*cos(45*2 + 22.5)}, {3*sin(45*2 + 22.5)})  {B};
     \node[draw,circle,scale=0.5] (CircleNode) at ({3*cos(45*3 + 22.5)}, {3*sin(45*3 + 22.5)})  {A};
     \node[draw,circle,scale=0.5] (CircleNode) at ({3*cos(45*4 + 22.5)}, {3*sin(45*4 + 22.5)})  {H};
     \node[draw,circle,scale=0.5] (CircleNode) at ({3*cos(45*5 + 22.5)}, {3*sin(45*5 + 22.5)})  {G};
     \node[draw,circle,scale=0.5] (CircleNode) at ({3*cos(45*6 + 22.5)}, {3*sin(45*6 + 22.5)})  {D};
     \node[draw,circle,scale=0.5] (CircleNode) at ({3*cos(45*7 + 22.5)}, {3*sin(45*7 + 22.5)})  {C};
}

\newcommand{\ptscustomlabel}[8]{
     \node[draw,circle,scale = 0.5] (CircleNode) at ({3*cos(45*0 + 22.5)}, {3*sin(45*0 + 22.5)})  {#3};
     \node[draw,circle,scale=0.5] (CircleNode) at ({3*cos(45*1 + 22.5)}, {3*sin(45*1 + 22.5)})  {#2};
     \node[draw,circle,scale=0.5] (CircleNode) at ({3*cos(45*2 + 22.5)}, {3*sin(45*2 + 22.5)})  {#1};
     \node[draw,circle,scale=0.5] (CircleNode) at ({3*cos(45*3 + 22.5)}, {3*sin(45*3 + 22.5)})  {#8};
     \node[draw,circle,scale=0.5] (CircleNode) at ({3*cos(45*4 + 22.5)}, {3*sin(45*4 + 22.5)})  {#7};
     \node[draw,circle,scale=0.5] (CircleNode) at ({3*cos(45*5 + 22.5)}, {3*sin(45*5 + 22.5)})  {#6};
     \node[draw,circle,scale=0.5] (CircleNode) at ({3*cos(45*6 + 22.5)}, {3*sin(45*6 + 22.5)})  {#5};
     \node[draw,circle,scale=0.5] (CircleNode) at ({3*cos(45*7 + 22.5)}, {3*sin(45*7 + 22.5)})  {#4};}

\newcommand{\twin}[1]{\draw[thick,cyan,rotate={-45*#1}] (0,{3*sin(45+22.5)}) ellipse (1.5cm and 0.5cm);}

\newcommand{\twinbig}[1]{\draw[thick,cyan,rotate={-45*#1}] (0,{3*sin(45+22.5)}) ellipse (2cm and 0.8cm);}

\newcommand{\plane}[1]{\draw[thick,LimeGreen,rotate={-45*#1-45}] (-3.6,0.4) rectangle (3.6,3.6);}

\newcommand{\midplane}[1]{\draw[thick,LimeGreen,rotate={-45*#1}] (-1.7,-3.6) rectangle (1.7,3.6);}

\newcommand{\canplane}[1]{\draw[thick,LimeGreen,rotate={-45*#1}] (-1.7,3.3)--(1.7,3.3)--(3.3,-1.7)--(-3.3,-1.7)--(-1.7,3.3);}

\newcommand{\planedot}[4]{\node[#1,LimeGreen,fill,draw,scale=#4] (d) at (#2,#3) {};}

\newcommand{\TA}[1]{\draw[thick,red,rotate={-45*#1}] (0,{3*sin(45+22.5)}) ellipse (1.7cm and 0.7cm);\draw[thick,red,rotate={-45*#1}] (0,2.07) -- (0,-3.6);}

\newcommand{\TAL}[1]{\draw[thick,red,rotate={-45*#1}] (0,{3*sin(45+22.5)}) ellipse (1.7cm and 0.7cm);\draw[thick,red,rotate={-45*#1}] (0,2.07) .. controls (-5,0.18) and (4,0) .. (0,-3.6);}

\newcommand{\TAR}[1]{\draw[thick,red,rotate={-45*#1}] (0,{3*sin(45+22.5)}) ellipse (1.7cm and 0.7cm);\draw[thick,red,rotate={-45*#1}] (0,2.07) .. controls (5,0.18) and (-4,0) .. (0,-3.6);}

\newcommand{\TB}[1]{\draw[thick,YellowOrange,rotate=-22.5-45*#1] plot [smooth cycle] coordinates {(0,3.5) (2.5,2.5) (2.5,1.5) (0,2.2) (-2.5,1.5) (-2.5,2.5)};}

\newcommand{\CA}{\draw[very thick,Goldenrod] (-4,0) -- (4,0);\draw[very thick,Goldenrod] (0,-4) -- (0,4);}

\newcommand{\CB}[1]{\draw[thick,RoyalPurple,rotate={-45*#1}] (-2,2) -- (2,2);\draw[thick,RoyalPurple,rotate={-45*#1}] (0,-3) -- (0,2);}

\newcommand{\CC}[1]{\draw[thick,magenta,rotate={-45*#1-45},dashed] (-3.75,0.0) rectangle (3.75,0.0);}

\newcommand{\TCu}{\draw[thick,brown] (0,0) circle (4.5);}

\newcommand{\HE}[1]{\draw[thick,brown,rotate={-45*#1+225}] plot [smooth] coordinates {(0.3,4) (0,3.1) (-0.3,2.2) (0,1.3) (0.3,0.4) (0,-0.5) (-0.3,-1.4) (0,-2.3) (0.3,-3.2) (0,-4)};}

\newcommand{\skewHE}[1]{\draw[thick,brown,rotate={-45*#1}] (-2,3.5)--(1.6,3.5)--(3.2,1.2)--(1.06,-3.46)--(-2,3.5);}

\newcommand{\oindex}[1]{\node at (0,-6) {\scriptsize #1}; \node at (0,5) {};}

\def\T
{\begin{tikzpicture}[scale=1.5]
    \draw[line width = 6] (0,0) to (0.8,0);
  \end{tikzpicture}}

\def\Dn
{\begin{tikzpicture}[scale=.20]
		\draw[line width = 4] (0,0) to[quick curve through={(2,0.5) (3,1.5) (2,3) (1,1.5) (2,0.5)}]
		(4,0);
\end{tikzpicture}}

\def\Unn
{\begin{tikzpicture}[scale=.20,xscale=0.8]
		\draw[line width = 2] (0,0) to[quick curve through={(2,0.5) (3,1.5) (2,3) (1,1.5) (2,0.5) (4,0) (6,0.5) (7,1.5) (6,3) (5,1.5) (6,0.5)}]
		(8,0);
\end{tikzpicture}}

\def\Znnn
{\begin{tikzpicture}[scale=.20,xscale=0.7]
		\draw (0,0) to[quick curve through={(2,0.5) (3,1.5) (2,3) (1,1.5) (2,0.5) (4,0) (6,0.5) (7,1.5) (6,3) (5,1.5) (6,0.5) (8,0) (10,0.5) (11,1.5) (10,3) (9,1.5) (10,0.5)}]
		(12,0);
\end{tikzpicture}}

\def\DNA
{\begin{tikzpicture}[scale=.15]
		\draw (-0.5,0.5) to[quick curve through={(1,1) (2,2) (1,3) (0,4) (1,5) (2,6) (1,7)}]
		(-0.5,7.5);
		\draw (2.5,0.5) to[quick curve through={(1,1) (0,2) (1,3) (2,4) (1,5) (0,6) (1,7)}]
		(2.5,7.5);
		\clip (-2.5,0) rectangle (4.5,8);
\end{tikzpicture}}

\def\UeU
{\begin{tikzpicture}[scale=.15]
		\draw[line width = 2] (-0.5,0.5) to[quick curve through={(2,3)}]
		(-0.5,6.5);
		\draw[line width = 2] (2.5,0.5) to[quick curve through={(0,3)}]
		(2.5,6.5);
		\clip (-2.5,0) rectangle (4.5,8);
\end{tikzpicture}}

\def\CAVE{
\begin{tikzpicture}[scale= .15]
		\draw (0,0) to[quick curve through={(2,3)}]
		(6,4);
		\draw (-1,1) to[quick curve through={(2,0.5)}]
		(6,0);
		\draw (3,-1) to[out angle = 90, in angle = 90, curve through={(4,5)}]
		(5,-1);
\end{tikzpicture}}

\def\UeZn
{\begin{tikzpicture}[scale= .14]
		\draw (-0.5,0.5) to[quick curve through={(2,3) (-1.5,6.5) (-1.5, 5.5) (-1.5,6.5)}]
		(-3.5,6.5);
		\draw[line width = 2] (2.5,0.5) to[curve through={(0,3)}]
		(3.5,6.5);
\end{tikzpicture}}

\def\ZZeU
{\begin{tikzpicture}[scale=.15]
		\draw (-0.5,-0.5) to[quick curve through={(2,3)}]
		(-0.5,6.5);
		\draw (2.5,-0.5) to[quick curve through={(0,3)}]
		(2.5,6.5);
		\draw[line width = 2] (-2,0) to (4,0);
\end{tikzpicture}}

\def\ZZeZZ
{\begin{tikzpicture}[scale=.15]
		\draw (0,2) to[curve through={(3,1)}]
		(6,3);
		\draw (4,3) to[curve through={(7,1)}] (10,2);
		\draw (0,1) to[in angle = 120, curve through={(3,2)}]
		(6,0);
		\draw (4,0) to[out angle = 60, curve through={(7,2)}] (10,1);
\end{tikzpicture}}

\def\ZZeUn
{\begin{tikzpicture}[scale=.12]
		\draw (-0.5,-0.5) to[quick curve through={(2,3)}]
		(-0.5,6.5);
		\draw (2.5,-0.5) to[quick curve through={(0,3)}]
		(2.5,6.5);
		\draw (-2,0) to[curve through={(2,0)(4,0)(5,0.5) (5,2) (5,0.5)}] (8,0);
\end{tikzpicture}}

\def\BRAID
{\begin{tikzpicture}[scale=.13]
		\draw (0,5) to (5,0);
		\draw (0,4.5) to (7,1.5);
		\draw (1,2) to (9,5);
		\draw (3,-1) to (8,5);
\end{tikzpicture}}

\def\ZneZn
{\begin{tikzpicture}[scale =.13]
		\draw (-0.5,0.5) to[quick curve through={(2,3) (-1.5,6.5) (-1.5, 5.5) (-1.5,6.5)}]
		(-3.5,6.5);
		\draw (2.5,0.5) to[curve through={(0,3) (3.5,6.5) (3.5,5.5) (3.5,6.5)}]
		(5.5,6.5);
\end{tikzpicture}}

\def\DeU
{\begin{tikzpicture}[scale=.3]
		\draw[line width = 4] (0,0) to
		(3.5,0);
		\draw[line width = 2] (0.6,1) to (0.6,-0.5);
\end{tikzpicture}}

\def\DeZn
{\begin{tikzpicture}[scale=.7]
		\draw[line width = 4] (0,0) to (1.5,0);
		\draw (0.6,1) to[quick curve through={(0.7,0.4) (0.9,0.3) (1,0.4) (0.9,0.5) (0.7,0.4)}] (0.6,-0.5);
\end{tikzpicture}}

\def\UeUn
{\begin{tikzpicture}[scale=.5]
		\draw[line width = 2] (0,0) to[quick curve through={(1.2,0.2) (1.4, 0.4) (1.2, 0.8) (1,0.4) (1.2,0.2)}] (2,0);
		\draw[line width = 2] (0.4,1) to (0.4,-0.5);
\end{tikzpicture}}

\def\UeUeU
{\begin{tikzpicture}[scale=.5]
		\draw[line width = 2] (0,0) to (2,0);
		\draw[line width = 2] (0,1) to (0.5,-0.5);
		\draw[line width = 2] (2,1) to (1.5, -0.5);
\end{tikzpicture}}

\def\UeUeUeZ
{\begin{tikzpicture}[scale=1.2]
		\draw (0,0) to (1,0);
		\draw[line width = 2] (0,0.5) to (0.25,-0.25);
		\draw[line width = 2] (1,0.5) to (0.75, -0.25);
		\draw[line width = 2] (0.5,0.5) to (0.5,-0.25);
\end{tikzpicture}}

\def\UeeeZ
{\begin{tikzpicture}[scale=.20]
		\draw (0,0.5) to[quick curve through={(1,1) (2,2) (1,3) (0,4) (1,5)}] (2,5.5);
		\draw[line width = 2] (2,0.5) to[quick curve through={(1,1) (0,2) (1,3) (2,4) (1,5)}] (0,5.5);
\end{tikzpicture}}

\def\UeeeUeZ
{\begin{tikzpicture}[scale=.2]
		\draw (0,0.5) to[quick curve through={(1,1) (2,2) (1,3) (0,4) (1,5)}] (2,5.5);
		\draw (3,0.5) to[quick curve through={(1,1) (0,2) (1,3) (2,4) (1,5)}] (0,5.5);
		\draw[line width = 2] (2,0) to (3,3);
\end{tikzpicture}}

\def\UeeeUn
{\begin{tikzpicture}[scale=.17]
		\draw (0,0.5) to[quick curve through={(1,1) (2,2) (1,3) (0,4) (1,5)}] (2,5.5);
		\draw (2.3,0) to[quick curve through={(2.1, 0.5) (2.4,1) (2.7,0.5) (2,0.5) (1,1) (0,2) (1,3) (2,4) (1,5)}] (0,5.5);
\end{tikzpicture}}

\def\ZeUnn
{\begin{tikzpicture}[scale=.16]
		\draw (-1,0) to[quick curve through={(2,0.5) (3,1.5) (2,3) (1,1.5) (2,0.5) (4,0) (6,0.5) (7,1.5) (6,3) (5,1.5) (6,0.5)}]
		(8,0);
		\draw[line width = 2] (0,5) to (0,-1);
\end{tikzpicture}}

\def\ZneUn
{\begin{tikzpicture}[scale=.63]
		\draw (0.6,1) to[quick curve through={(0.7,0.4) (0.9,0.3) (1,0.4) (0.9,0.5) (0.7,0.4)}] (0.6,-0.6);
		\draw[line width = 2] (1.8,-0.4) to[quick curve through={(1.2,-0.3) (1.1,-0.1) (1.2,0) (1.3,-0.1) (1.2,-0.3)}] (0.3,-0.4);
\end{tikzpicture}}

\def\UeZeU
{\begin{tikzpicture}[scale=0.14]
		\draw[line width = 2] (-0.5,0.5) to[quick curve through={(2,3)}]
		(-0.5,6.5);
		\draw (2.5,0.5) to[quick curve through={(0,3)}]
		(4,6.5);
		\draw[line width = 2] (3,5) to (3,8);
\end{tikzpicture}}

\def\UneZeZ
{\begin{tikzpicture}[scale=.6]
		\draw[line width = 2] (0,0) to (2,0);
		\draw (0.6,1) to[quick curve through={(0.7,0.4) (0.9,0.3) (1,0.4) (0.9,0.5) (0.7,0.4)}] (0.6,-0.6);
		\draw[line width = 2] (1.5,1) to (1.5, -0.5);
\end{tikzpicture}}

\def\ZeUneZ
{\begin{tikzpicture}[scale=.45,xscale=0.8]
        \draw[] (-1,0) to[quick curve through={(1,0.2) (1.2,0.4) (1,0.8) (0.8,0.4) (1,0.2)}]
		(3,0);
		\draw[line width = 2] (-0.5,1.5) to (0,-0.5);
		\draw[line width = 2] (2.5,1.5) to (2, -0.5);
\end{tikzpicture}}

\def\UnneUn
{\begin{tikzpicture}[scale=.25]
		\draw (-0.8,3) to[quick curve through={(-0.6,1.8) (-0.2,1.6) (0,1.8) (-0.2,2) (-0.6,1.8)}] (-0.8,-0.7);
		\draw (-2,0.5) to[quick curve through={(1,0.25) (1.5,0.75) (1,1.5) (0.5,0.75) (1,0.25) (2,0) (3,0.25) (3.5,0.75) (3,1.5) (2.5,0.75) (3,0.25)}]
		(4,0);
\end{tikzpicture}}

\def\UeZeZn
{\begin{tikzpicture}[scale=0.15]
		\draw[line width = 2] (-0.5,0.5) to[quick curve through={(2,3)}]
		(-0.5,6.5);
		\draw (2.5,0.5) to[quick curve through={(0,3)}]
		(4,6.5);
		\draw (3.2,7.5) to[quick curve through={(3.4,4.8) (3.8,4.6) (4,4.8) (3.8,5) (3.4,4.8)}] (3.2,3);
\end{tikzpicture}}

\def\ZneZeU
{\begin{tikzpicture}[scale= .15]
		\draw (-0.5,0.5) to[quick curve through={(2,3) (-1.5,6.5) (-1.5, 5.5) (-1.5,6.5)}]
		(-3.5,6.5);
		\draw (2.5,0.5) to[curve through={(0,3)}]
		(3.5,6.5);
		\draw[line width = 2] (3,5) to (3,8);
\end{tikzpicture}}

\def\ZneUeZn
{\begin{tikzpicture}[scale=0.55]
	    \draw[line width = 2] (0,0) to (2.2,0);
		\draw (0.6,1) to[quick curve through={(0.7,0.4) (0.9,0.3) (1,0.4) (0.9,0.5) (0.7,0.4)}] (0.6,-0.6);
		\draw (1.6,1) to[quick curve through={(1.7,0.4) (1.9,0.3) (2,0.4) (1.9,0.5) (1.7,0.4)}] (1.6,-0.6);
\end{tikzpicture}}

\def\ZneZneU
{\begin{tikzpicture}[scale=0.4]
       \draw[] (-1,0) to[quick curve through={(1,0.2) (1.2,0.4) (1,0.8) (0.8,0.4) (1,0.2)}]
		(3,0);
		\draw (-0.6,1.5) to[quick curve through={(-0.4,0.75) (-0.15,0.65) (-0,0.75) (-0.25,1) (-0.4,0.75)}] (0,-0.5);
		\draw[line width = 2] (2.5,1.5) to (2, -0.5);
\end{tikzpicture}}

\def\UeZZeU
{\begin{tikzpicture}[scale=0.15]
		\draw (-0.5,0.5) to[quick curve through={(2,3)}]
		(-2,6.5);
		\draw (2.5,0.5) to[quick curve through={(0,3)}]
		(4,6.5);
		\draw[line width = 2] (3,5) to (3,8);
		\draw[line width = 2] (-1,5) to (-1,8);
\end{tikzpicture}}

\def\ZneUeUeZ
{\begin{tikzpicture}[scale=1.2]
		\draw (0,0) to (1,0);
		\draw (0,0.5) to[quick curve through={(0.1,0.125) (0.225,0.075) (0.3,0.125) (0.175,0.25) (0.1,0.125)}] (0.3,-0.25);;
		\draw[line width = 2] (1,0.5) to (0.75, -0.25);
		\draw[line width = 2] (0.5,0.5) to (0.5,-0.25);
\end{tikzpicture}}

\def\ZeeeZeZn
{\begin{tikzpicture}[scale=0.18]
		\draw (0,0.5) to[quick curve through={(1,1) (2,2) (1,3) (0,4) (1,5)}] (2,5.5);
		\draw (3,0.5) to[quick curve through={(1,1) (0,2) (1,3) (2,4) (1,5)}] (0,5.5);
		\draw (2,0) to[quick curve through={(2.5,1.5)(2.8,1.75)(3.5,1.5)(2.7,1.25)(2.5,1.5)}] (3,3);
\end{tikzpicture}}

\def\ZneZeZn
{\begin{tikzpicture}[scale= .13]
		\draw (-0.5,0.5) to[quick curve through={(2,3) (-1.5,6.5) (-1.5, 5.5) (-1.5,6.5)}]
		(-3.5,6.5);
		\draw (2.5,0.5) to[curve through={(0,3)}]
		(3.5,6.5);
		\draw (3.2,7.5) to[quick curve through={(3.4,4.8) (3.8,4.6) (4,4.8) (3.8,5) (3.4,4.8)}] (3.2,3);
\end{tikzpicture}}

\def\ZneZneZn
{\begin{tikzpicture}[scale=.4,xscale=0.8]
        \draw[] (-1,0) to[quick curve through={(1,0.2) (1.2,0.4) (1,0.8) (0.8,0.4) (1,0.2)}]
		(3,0);
		\draw (-0.6,1.5) to[quick curve through={(-0.4,0.75) (-0.15,0.65) (-0,0.75) (-0.25,1) (-0.4,0.75)}] (0,-0.5);
		\draw (2.6,1.5) to[quick curve through={(2.4,0.75) (2.15,0.65) (2,0.75) (2.25,1) (2.4,0.75)}] (2,-0.5);
\end{tikzpicture}}

\def\ZZeZeU
{\begin{tikzpicture}[scale=.12]
		\draw (-0.5,-0.5) to[quick curve through={(2,3)}]
		(-0.5,6.5);
		\draw (2.5,-0.5) to[quick curve through={(0,3)}]
		(2.5,6.5);
		\draw (-2,0) to (7,0);
		\draw[line width = 2] (5,-0.5) to (5,5.5);
\end{tikzpicture}}

\def\UeZZeZn
{\begin{tikzpicture}[scale=0.12]
		\draw (-0.5,0.5) to[quick curve through={(2,3)}]
		(-2,6.5);
		\draw (2.5,0.5) to[quick curve through={(0,3)}]
		(4,6.5);
		\draw (3.2,7.5) to[quick curve through={(3.4,4.8) (3.8,4.6) (4,4.8) (3.8,5) (3.4,4.8)}] (3.2,3);
		\draw[line width = 2] (-1.2,7.5) to (-1.2,3);
\end{tikzpicture}}

\def\ZneZneUeZ
{\begin{tikzpicture}[scale=1]
		\draw (-0.1,0) to (1,0);
		\draw (0.1,0.5) to[quick curve through={(0.1,0.2)(0.2,0.13)(0.3,0.2)(0.2,0.3)(0.1,0.2)}] (0.1,-0.25);
		\draw[line width = 2] (1,0.5) to (0.75, -0.25);
		\draw (0.5,0.5) to[quick curve through={(0.5,0.2)(0.6,0.13)(0.7,0.2)(0.6,0.3)(0.5,0.2)}] (0.5,-0.25);
\end{tikzpicture}}

\def\ZZeZeZn
{\begin{tikzpicture}[scale=0.12]
		\draw (-0.5,-0.5) to[quick curve through={(2,3)}]
		(-0.5,6.5);
		\draw (2.5,-0.5) to[quick curve through={(0,3)}]
		(2.5,6.5);
		\draw (-2,0) to (7,0);
		\draw (5,-0.5) to[quick curve through={(5.2,2)(5.7,2.5)(6.7,2)(5.7,1.5)(5.2,2)}] (5,4);
\end{tikzpicture}}

\def\ZneZZeZn
{\begin{tikzpicture}[scale=0.12]
		\draw (-0.5,0.5) to[quick curve through={(2,3)}]
		(-2,6.5);
		\draw (2.5,0.5) to[quick curve through={(0,3)}]
		(4,6.5);
		\draw (3.2,7.5) to[quick curve through={(3.4,4.8) (3.8,4.6) (4,4.8) (3.8,5) (3.4,4.8)}] (3.2,3);
		\draw (-2.2,7.5) to[quick curve through={(-1.4,4.8) (-1.8,4.6) (-2,4.8) (-1.8,5) (-1.4,4.8)}] (-1.2,3);
\end{tikzpicture}}

\def\ZneZneZneZ
{\begin{tikzpicture}[scale=1]
		\draw (-0.1,0) to (1.1,0);
		\draw (0.1,0.5) to[quick curve through={(0.1,0.2)(0.2,0.13)(0.3,0.2)(0.2,0.3)(0.1,0.2)}] (0.1,-0.25);
		\draw (0.5,0.5) to[quick curve through={(0.5,0.2)(0.6,0.13)(0.7,0.2)(0.6,0.3)(0.5,0.2)}] (0.5,-0.25);
		\draw (0.9,0.5) to[quick curve through={(0.9,0.2)(1,0.13)(1.1,0.2)(1,0.3)(0.9,0.2)}] (0.9,-0.25);
\end{tikzpicture}}

\RaymondContactDetails
\JordanContactDetails
\VladimirContactDetails
\ReynaldContactDetails
\ElisaContactDetails

\begin{document}

\title{Reduction of Plane Quartics and Cayley Octads}

\begin{abstract}
  We give a conjectural
  characterisation of the stable reduction of plane quartics over local fields
  in terms of their Cayley
  octads.
  This results in $p$-adic criteria that efficiently give the stable reduction type
  amongst the 42 possible types, and whether the reduction is hyperelliptic or not.
  These criteria are in the vein of the machinery of ``cluster pictures'' for hyperelliptic
  curves.
  We also construct explicit families of quartic curves that
  realise all possible stable types, against which we test these criteria. We give numerical examples that illustrate how to use these criteria in practice.
\end{abstract}

\maketitle

\ \vspace*{-1.5cm}
\setcounter{tocdepth}{2}
\tableofcontents

\section{Introduction}

Deligne and Mumford's proof of the irreducibility of the moduli space of
smooth projective curves of genus $g \geq 2$ over an algebraically closed field
consists of compacting this space by adding the curves with mild
singularities, the \emph{stable curves}~\cite{DM69}. Singularities of a stable
curve are ordinary double points, and its irreducible components of geometric
genus 0 have at least three such double points, counted with multiplicity.
A consequence of the Deligne-Mumford construction is the stable reduction
theorem: any curve over a local field acquires stable reduction after a finite
extension of the base field.
Stable models of a curve give access to much information of arithmetic
nature about the curve and its Jacobian (genus, conductor, \textit{etc.}),
and their actual calculation is a motivating question.
A commonly used way to determine it is to repeatedly blow up the singular
points and components of the special fibre and take
normalizations~\cite{HM98}.  However, this method can also be a difficult task
from a computational point of view, even for genus 3 curves that are the focus
of this work.

One of the first results that goes in the direction of greater effectivity is
due to Liu, for the case of curves of genus 2~\cite{Liu93}. Liu gives, in terms
of the Igusa invariants of a curve, $C$, not only the stable type,
\textit{i.e.}\ the graph of the irreducible components of $C$ as well as the
genera of their normalisations, but also the $j$-invariant of the irreducible
components of genus 1 when the special fibre is not smooth.

While there are only 7 possibilities for the type of stable reduction in genus
2, the situation is considerably more involved for curves of genus 3.  %
The multiplicity of cases and also the complexity of invariant
algebras complicate the generalisation of Liu's approach to genus 3. Partial
results exist, however, characterising potentially good
reduction of a quartic, or determining when a plane quartic reduces to a
hyperelliptic curve~\cite{lllr21}.

One way to extend the locus of curves for which the stable reduction can be
determined efficiently is to consider families given by Galois covers of
curves of smaller genus.
Modulo certain moderate conditions on the genus and degree of the cover, this
approach has been successfully developed by Bouw and Wewers for the case of
cyclic covers of $\P^1$, \textit{i.e.}\ superelliptic curves~\cite{bouwwew} and
the important sub-loci of hyperelliptic curves (covers of degree 2) and genus~3
Picard curves (covers of degree 3)~\cite{BKSW20}.
The method can sometimes be generalised to curves whose automorphism group is
non-trivial. In genus 3, for example, we have in~\cite{BCKLS20} a complete
answer for Ciani curves, which admit a faithful action of the Klein group
$\CG_2 \times \CG_2$ and which can therefore be realised as a degree 2 cover of an
elliptic curve.

One can also consider models of curves that lend themselves to computational
analysis.
There is a remarkable result in this direction due to
Tim Dokchitser~\cite{Dokchitser21}. By associating a fan to the Newton
polytope of a plane curve, Dokchitser makes explicit how to deduce, among many
other arithmetic invariants, a minimal regular model with normal crossings,
and thus a stable model, so long as the curve satisfies a
condition known as ``$\Delta_v$-regularity''.
This approach is fairly well adapted to the case of plane quartics, %
however complications arise when trying to construct a $\Delta_v$-regular quartic from one given in an arbitrary way. We note that this method is particularly unlikely to succeed when the reduction type of the quartic is more complicated.

The aim of the present paper is to describe a new approach to understanding the stable reduction of curves of genus 3, which most closely resembles \cite{m2d2}. The latter characterises the stable reduction of hyperelliptic curves $y^2=f(x)$ using ``cluster pictures'', combinatorial objects that encode $p$-adic distances between the roots of the polynomial $f(x)$. An important advantage of this description is that it is well-suited for the study of curves over global fields, where one often needs to control the local arithmetic of curves over all the completions of the field simultaneously. We believe to have identified the correct replacement for the roots of the polynomial $f(x)$ in the context of non-hyperelliptic genus 3 curves.

Let us briefly recall some of the cluster picture machinery. %
Consider the genus 3 hyperelliptic curve $C/\Q_p$ given by
$$
y^2=f(x) = (x-p^2 \alpha_1)(x-p^2 \alpha_2)(x - p^2 \alpha_3)(x-\beta_1)(x-\beta_2)(x-\beta_3)(x-\beta_4)(x-\beta_5),
$$
where $p$ is an odd prime, and $\alpha_i, \beta_i \in \Z_p \backslash \{ 0 \}$ have distinct images in the residue field. Over $\Q_p$ the roots of $f$ are not equidistant, rather they contain a ``cluster'' of size $3$, \textit{i.e.}\ there are 3 roots that coincide modulo $p$ (and even modulo $p^2$) while the others remain distinct. This is reflected in the stable reduction as follows. The naive reduction of this equation is the curve $y^2 = x(x-\beta_1)(x-\beta_2)(x-\beta_3)(x-\beta_4)(x-\beta_5)$, a genus $2$ curve with a cusp. The change of variable $x\mapsto p^2x'$ and $y\mapsto p^{3}y'$, provides a model whose reduction is the elliptic curve $y'^2= -\beta_1 \beta_2 \beta_3 \beta_4 \beta_5 (x' - \alpha_1)(x'- \alpha_2)(x'- \alpha_3)$, with a cusp at infinity. The special fibre of the stable model consists of this genus $2$ curve and elliptic curve intersecting at one point. The key point is that the purely combinatorial information about how roots cluster together (the ``cluster picture'') completely determines the stable reduction type. Indeed, any genus 3 hyperelliptic curve with a cluster of size 3 and no other clusters will have a potential stable model whose special fibre consists of a genus 2 curve and an elliptic curve meeting at one point.
More generally, clusters of size 2 or 6 give rise to nodes. Clusters of size 3
or 5 correspond, as we have just seen, to a decomposition of the special fibre
into two connected parts, one of arithmetic genus 1 and one of arithmetic
genus 2, intersecting in one point. Clusters of size 4 correspond to a
decomposition of the special fibre into two connected parts, both of
arithmetic genus 1, intersecting at two points. All cluster pictures are made
by combining these elementary blocks. For the complete dictionary, see
\cite[Tab.~9.1]{m2d2}.

\begin{figure}[htbp]
  \centering
  \tikzsetnextfilename{IntroCluster}
  \begin{tikzpicture}
    \node at (-4,0) (26) { {\clusterpicture\Root[A]{2}{first}{r1};\Root[A]{2}{r1}{r2};\ClusterLD c1[][] = (r1)(r2);\Root[A]{2}{c1}{r3};\Root[A]{2}{r3}{r4};\Root[A]{2}{r4}{r5};\Root[A]{2}{r5}{r6};\Root[A]{2}{r6}{r7};\Root[A]{2}{r7}{r8};\ClusterLD c3[][] = (c1)(r3)(r4)(r5)(r6)(r7)(r8);\endclusterpicture} / {\clusterpicture\Root[A]{2}{first}{r1};\Root[A]{2}{r1}{r2};\Root[A]{2}{r2}{r3};\Root[A]{2}{r3}{r4};\Root[A]{2}{r4}{r5};\Root[A]{2}{r5}{r6};\ClusterLD c1[][] = (r1)(r2)(r3)(r4)(r5)(r6);\Root[A]{2}{c1}{r7};\Root[A]{2}{r7}{r8};\ClusterLD c3[][] = (c1)(r7)(r8);\endclusterpicture}  };
    \node[scale=0.7] at (2,0) (2n) { \Dn };

    \draw[->, shorten <= 1em, shorten >= 1.7em] (26) to (2n);

    \node at (-4,-0.75) (35) { {\clusterpicture\Root[A]{}{first}{r1};\Root[A]{2}{r1}{r2};\Root[A]{2}{r2}{r3};\ClusterLD c1[][] = (r1)(r2)(r3);\Root[A]{2}{c1}{r4};\Root[A]{2}{r4}{r5};\Root[A]{2}{r5}{r6};\Root[A]{2}{r6}{r7};\Root[A]{2}{r7}{r8};\ClusterLD c3[][] = (c1)(r4)(r5)(r6)(r7)(r8);\endclusterpicture} / {\clusterpicture\Root[A]{2}{first}{r1};\Root[A]{2}{r1}{r2};\Root[A]{2}{r2}{r3};\Root[A]{2}{r3}{r4};\Root[A]{2}{r4}{r5};;\ClusterLD c1[][] = (r1)(r2)(r3)(r4)(r5);\Root[A]{2}{c1}{r6};\Root[A]{2}{r6}{r7};\Root[A]{2}{r7}{r8}\ClusterLD c3[][] = (c1)(r6)(r7)(r8);\endclusterpicture} };
    \node[scale=0.9] at (2,-0.75) (2e) { \DeU };

    \draw[->, shorten <= 1em, shorten >= 1em] (35) to (2e);

    \node at (-4,-1.5) (4) { \clusterpicture\Root[A]{2}{first}{r1};\Root[A]{2}{r1}{r2};\Root[A]{2}{r2}{r3};\Root[A]{2}{r3}{r4};\ClusterLD c1[][] = (r1)(r2)(r3)(r4);\Root[A]{2}{c1}{r5};\Root[A]{2}{r5}{r6};\Root[A]{2}{r6}{r7};\Root[A]{2}{r7}{r8};\ClusterLD c3[][] = (c1)(r5)(r6)(r7)(r8);\endclusterpicture };
    \node[scale=0.9,rotate=90] at (2,-1.5) (1=1) { \UeU };

    \draw[->, shorten <= 3em, shorten >= 1em] (4) to (1=1);

  \end{tikzpicture}%
\end{figure}

\noindent

We now turn to our setting of plane quartics (equivalently, non-hyperelliptic genus $3$ curves). Here, we propose to replace the eight Weierstrass points in the hyperelliptic case by a {\em Cayley octad}.
Fixing one of the 36 even theta characteristics $\theta$ (\textit{i.e.}\ a divisor on the curve such that $2\theta$ lies in the canonical divisor class and the Riemann-Roch space of $\theta$ has even dimension) gives rise to both an embedding of the plane quartic into $\P^3$, and $8$ points in $\P^3$ which form the Cayley octad. %
Analogous to the Weierstra\ss\ points of a hyperelliptic curve, these 8 points determine the curve. We conjecture that combinatorial data about the configuration of these eight points fully determines the stable reduction type of the curve, analogously to the cluster picture in the hyperelliptic case.

There are two natural complications compared to the hyperelliptic case. First, our eight points live in $\P^3$ rather than in $\P^1$. In particular, this means that there are new possible degenerations to consider when looking at the points over the residue field. Specifically, there are the following four basic degenerations: (i) several points  coinciding (the direct analogue of a cluster), (ii) four points lying on a plane, (iii) three points lying on a line, (iv) seven points lying on a twisted cubic curve (as with cluster pictures, the case of no degenerations occurring corresponds to good reduction).

Second, unlike the case of the Weierstra\ss\ points of a hyperelliptic curve, the 8 points of a Cayley octad are not independent and satisfy an algebraic relation, under which any point is determined by the other seven, see Rem.~\ref{rmk:7determine8}.
Because of this relation, the degenerations that can actually occur are harder to describe.
For example, under this relation, a degeneration containing four coplanar points forces the other four points of the octad to also be coplanar as depicted in Fig.~\ref{fig:introexamplealpha2singularity}. Similarly, if seven points lie on a twisted cubic, then so does the eighth.

To see how the degenerations of the Cayley octad affect the reduction type in practice, consider the following example.

\begin{example}
\label{ex:introalpha2singularity}
    Consider the plane quartic $C/\Q_{17}$ given by
    \footnotesize
\begin{align*}
C \colon \, &32626944 x^4 + 47032608 x^3 y - 57820391 x^2 y^2 - 53891530 x y^3 +
 42837025 y^4 \\ + &23956128 x^3 z + 117108918 x^2 y z -
 103647330 x y^2 z + 6832980 y^3 z - 3176703 x^2 z^2 \\ +
 &46151898 x y z^2 - 8406186 y^2 z^2 - 2780622 x z^3 - 692172 y z^3 +
 439569 z^4 = 0.
\end{align*}
\normalsize
Then $C$ has a Cayley octad, $O$, whose 8 points are given by the columns of the following matrix:
\begin{displaymath}
      \arraycolsep=0.4\arraycolsep%
      \mbox{\scriptsize
        $\begin{bmatrix}
            u \\
            v \\
            r \\
            s \\
        \end{bmatrix}
        =
        \begin{bmatrix}
          1 & 0 & 0 &  -475783999 & 0 & 1 &  13 & 4\\
          0 & 1 & 0 & 2635587848 & 0 & 1 & -6 & -7 \\
          0 & 0 & 1 & -5915152420 & 0 & 1 & 5 & 11 \\
          0 & 0 & 0 &  3698228144 & 1 & 1 & -8 & -10
        \end{bmatrix}$}
    \end{displaymath}
    Over $\F_{17}$, the first four points lie on the plane $s = 0$, while the final four lie on the plane $2r + 6u + 9v = 0$. If one considers $C$ under the embedding into $\P^3$ given by $O$  (as in Rem.~\ref{rmk:construction-cayley}), $C$ now contains a line, which is the intersection of these two planes (see Figure \ref{fig:introexamplealpha2singularity}). This line contracts when considering the stable model of $C$, which we thus observe to have a singular point (and is in fact a genus $2$ curve with a node).

    \begin{figure}[htbp]
        \centering
        \includegraphics[scale=0.2]{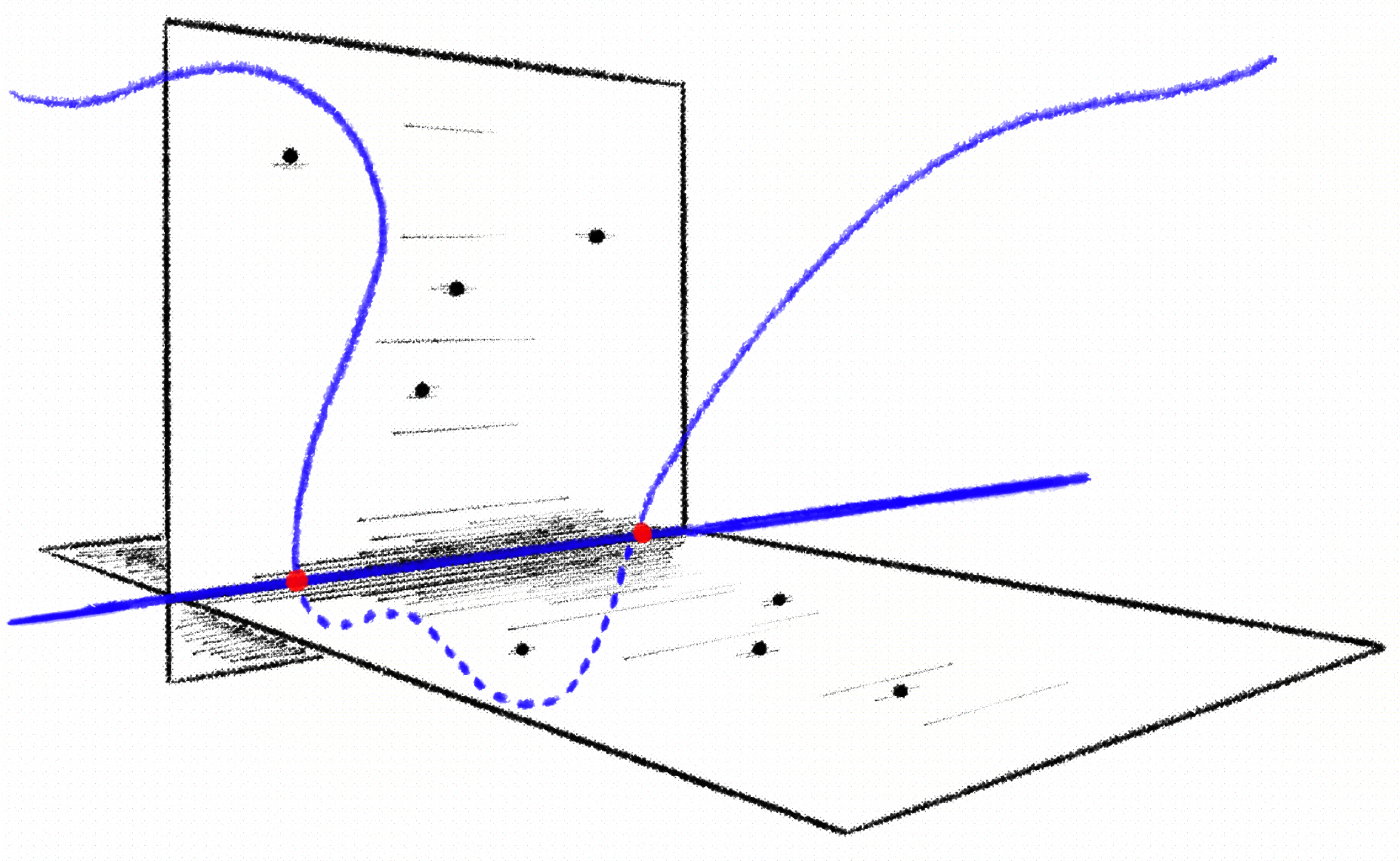}
        \caption{$C$ (blue) over $\F_{17}$, with line component, which is simultaneously the intersection of the two planes $s = 0$ and $2r + 6u + 9v = 0$ upon which the 8 points of the Cayley octad (indicated by dots) lie.}
        \label{fig:introexamplealpha2singularity}
    \end{figure}
\end{example}

The central point of this work is to understand all possible degenerations (leading us to the notion of \textit{octad} pictures, see Def.~\ref{def:blockdecompositiondiagram}) and relate them to stable reduction types. We will not, however, attempt to produce explicit models like in the above example.  Indeed, unlike the hyperelliptic case where one can easily see how clusters give rise to components of the special fibre, this process appears to be rather mysterious for plane quartics. We will, however, give a conjectural recipe for the reduction type. Just as cluster pictures are made up of a collection of clusters, so our octad pictures are made up of certain {\em building blocks}. The so-called ${\bm \alpha}$-, ${\bm \chi}$- and ${\bm \phi}$-blocks, which will be defined later in Tab.~\ref{tab:introblocksdefinition}, conjecturally correspond to clusters of size 2 or 6, 3 or 5, and 4 respectively.
The degeneration in which two points of the Cayley octad collide, and the degeneration in Ex.~\ref{ex:introalpha2singularity}, are examples of ${\bm \alpha}$-blocks.
The last type of degeneration, the so-called hyperelliptic blocks, includes the degeneration in which the 8 points lie on a twisted cubic, and correspond to plane quartics that become hyperelliptic after reduction.

We expect these degenerations to determine the stable reduction of the plane quartic curve, akin to the situation of cluster pictures for hyperelliptic curves. In fact, we produce an explicit map
\begin{equation}\label{eq:explicit-map}
\textup{Octad Pictures} \to \textup{Stable Reduction Types}
\end{equation}
\begin{conjecture}[see Conj.~\ref{conj:SpecialFibreOfTheStableModel} and~\ref{conj:SpecialFibreOfTheStableModelDetailed}]
    For any plane quartic $C$ over a non-archimedean local field $K$ of characteristic $p \neq 2$, applying this explicit map to any of its Cayley octads gives the stable reduction type of $C/K$.
\end{conjecture}

There is substantial evidence supporting this conjecture. Along with extensive numerical data (see Sec.~\ref{subsec:evidence}), we have a perfect combinatorial matching between the degenerations of Cayley octads and stable reduction types, which will occupy the bulk of this paper. We also have the following theoretical checks, demonstrating that our conjecture is compatible with natural structures on Cayley octads.

Consider, first, that we may perform a change of coordinates on $\P^3$. Naturally this does not change the stable reduction type of a non-hyperelliptic genus $3$ curve, but can change the degenerations in the Cayley octad. For instance, if $O$ is a Cayley octad over a local field $K$ with two points coinciding over the residue field $k$ of $K$ (and no further degenerations), then there is a $\PGL$-transformation of $\P^3$ which, over $k$, separates the two coincident points, but renders the remaining $6$ points coplanar.

\begin{theorem}[see Thm.~\ref{thm:pgl-action-on-block-decomposition}]
    The octad picture produced by degenerations of a Cayley octad is independent of the choice of coordinates on $\P^3$.
\end{theorem}

Second, non-hyperelliptic genus $3$ curves have $36$ Cayley octads; our conjecture should be independent of whichever we choose. We have proved that this is indeed the case when the degeneration of a Cayley octad consists of a single building block. Of course we anticipate this to hold no matter the number of building blocks, see Conj.~\ref{conj:symplecticgroupactioncommutespictures}.

\begin{theorem}[see Prop.~\ref{CreTransTwins},~\ref{CreTransTypes},~\ref{prop:CreTransCans} and~\ref{prop:CreTransHE}]
    The image of an octad picture associated to a Cayley octad of a non-hyperelliptic genus $3$ curve under the explicit map \eqref{eq:explicit-map} is independent of choice of Cayley octad, in the case that its degenerations are given by a single building block.
\end{theorem}

Again we caution the reader that the octad picture \textit{does} depend on the choice of Cayley octad; it is only the image under the explicit map which does not. For example, if $C$ is a plane quartic with a Cayley octad whose degenerations consist of two coplanar quadruples in the residue field (as in Ex. \ref{ex:introalpha2singularity}), then, up to $\PGL$, $20$ of the octad pictures of $C$ will come from two coplanar quadruples in the residue field, and $16$ will come from two coincident points in the residue field.

Our explicit map obeys the expected structure as well. Modulo questions of labelling (see the discussion immediately preceding Conjecture~\ref{conj:SpecialFibreOfTheStableModel}) and a natural notion of multiplicity (see Remark~\ref{rmk:octadpictureindices}), our explicit map is a 36 to 1 covering. That is, all stable reduction types are achieved by some collection of degenerations, and the only other degenerations achieving the same stable reduction type are accounted for precisely by the $36$ choices of Cayley octad.

\begin{remark}\label{rmk:hyperell-comparison}
The Cayley octad is a natural geometric analogue of the Weierstra\ss\ points of a hyperelliptic curve as follows.
There exists a well known connection between the Weierstra\ss\ points and the theta characteristics of a hyperelliptic curve.
If $P$ is a Weierstra\ss\ point of a hyperelliptic curve, the divisor $2\,[P]$ is for instance one of them, and the others can be obtained by adding a 2-torsion class.
In fact, it even turns out that in the case of the Cayley octad degenerating into 8 points on a twisted cubic, the 8 points on this twisted cubic are the Weierstra\ss\ points of the hyperelliptic curve obtained after reduction, see~\cite[Sec.~IX.3]{DolgachevOrtland}.
\end{remark}

\begin{remark}\label{rmk:construction-cayley}
Given a plane quartic curve $C$ with even theta characteristic $\theta$, the accompanying Cayley octad can be constructed as follows.
Explicitly, $\theta$ allows $C$ to be written in the form $\det( x L + y M + z N) = 0$, where $L, M, N$ are symmetric $4 \times 4$-matrices (such a
representation has been a subject of study since the 19th century~\cite{Hesse1844,Hesse,Dixon1902}).
The collection of matrices $ x L + y M + z N$, for $(x : y : z) \in \P^2$, gives a net, $\mathcal{N}$, of quadrics in $\P^3$, where, generically, the points of $C$ (embedded in $\P^2$) come from quadrics that are degenerate, \textit{i.e.}\ of rank 3 or lower.
The eight points of the Cayley octad form the base locus of $\mathcal{N}$, \textit{i.e.}\ they are precisely the points common to all quadrics in $\mathcal{N}$.
The totality of singular points of quadrics of $\mathcal{N}$ gives the image of $C$ under the embedding induced by $3\theta$.
Moreover, the 28 lines that these eight points define are bi-secants to the curve $C$. The pairs of points where these lines meet $C$ are exactly the 28 odd theta characteristics of $C$, see \cite[Sec.~6]{GrossHarris}.
\end{remark}

\begin{example}
   Consider the Cayley octad $O/\Q_{7}$, whose points are given by the columns of the following matrix:
   $$\begin{bmatrix}u\\ v\\ r\\ s\end{bmatrix} =
   \begin{bmatrix}
   1  &0  &0  &0  &1  &-1 &-1 &2 \\
   0  &1  &0  &0  &1  &6  &2  &12 \\
   0  &0  &1  &0  &1  &13 &1  &10 \\
   0  &0  &0  &1  &1  &1  &7  &7
   \end{bmatrix}$$
   The underlying plane quartic here is
   \begin{align*}
   C \, \colon \, &676x^4 + 468x^3y + 36504x^3z - 179x^2y^2 + 32452x^2yz + 607044x^2z^2 - 90xy^3 + 3444xy^2z\\ &+ 464712xyz^2 + 3916080xz^3 + 25y^4 - 1480y^3z + 68104y^2z^2 + 1716960yz^3 + 8643600z^4.\end{align*}
   Observe that, modulo $7$, the fourth, fifth, and sixth point lie on the line $u = v = r$. Meanwhile, the other five points lie on the plane $s = 0$, and hence on a conic, which in this case is $3uv + 3ur + vr$. Note that this conic and the line intersect.

   This degeneration is an example of a ${\bm \chi}_{2\textrm{c}}$-block. The stable reduction type of $C$ is a genus $2$ curve and an elliptic curve meeting at a point. Moreover, the genus $2$ curve is ramified over the final five points of the octad, with the conic forming the quotient under the hyperelliptic involution.
    \begin{figure}[htbp]
        \centering
        \includegraphics[scale=0.2]{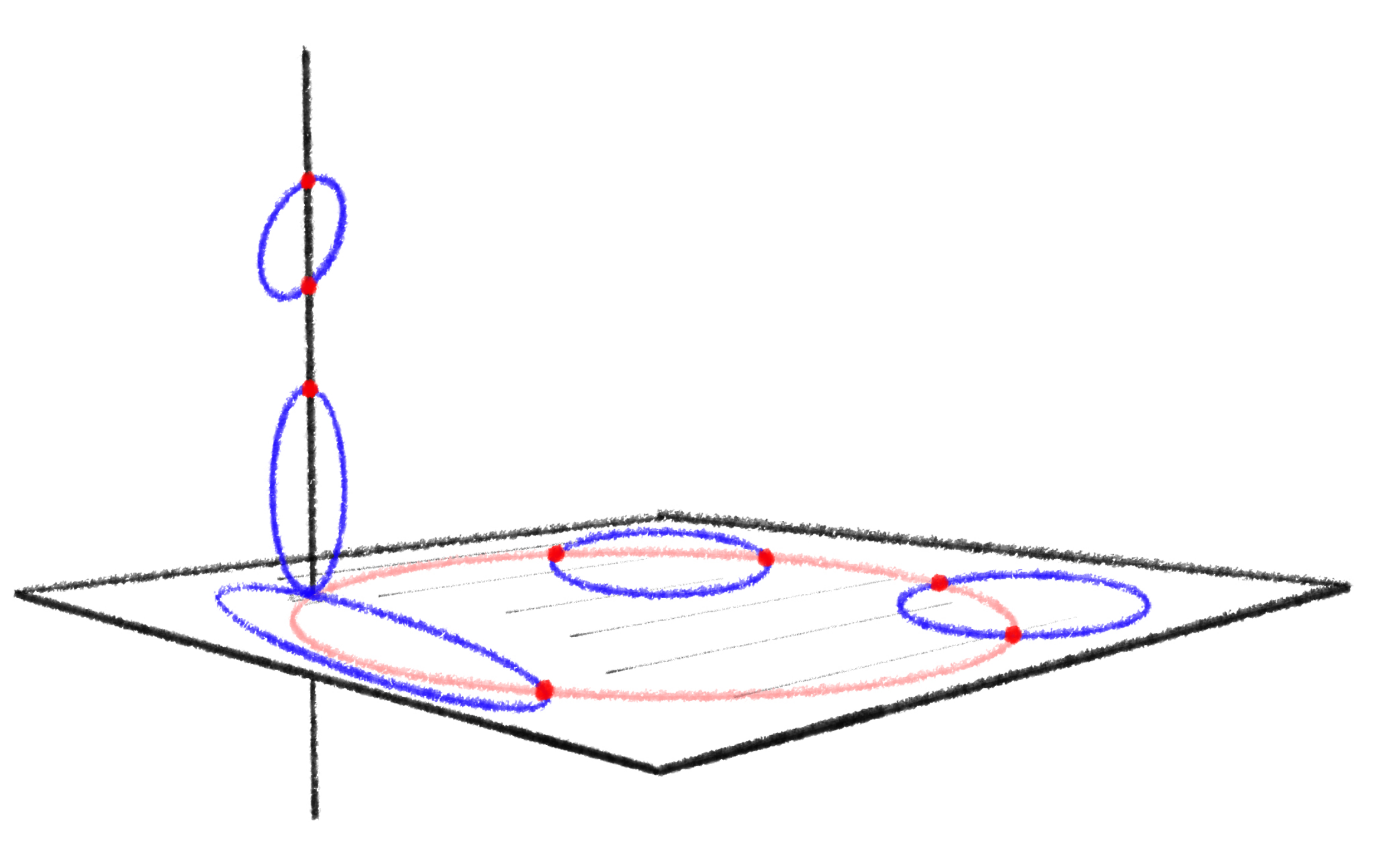}
        \caption{$C$ (blue), with the $8$ points of $O$ (red). The 5 coplanar points lie on a conic (light red).}
    \end{figure}
\end{example}

\begin{remark}
    Just as we have given a general relationship between clusters appearing in a cluster picture and the corresponding stable reduction type, so too can we give a general relationship between our building blocks and the corresponding stable reduction type. The naming scheme is intended to reinforce this relationship.

    \begin{itemize}
        \item  $\bm{\alpha}$-blocks give rise to nodal intersections in the stable reduction type (graphically akin to the glyph $\bm{\alpha}$).
        \item $\bm{\chi}$-blocks correspond to a decomposition of the stable reduction type into two connected parts of arithmetic genus $1$ and $2$, respectively, intersecting once.
        \item $\bm{\phi}$-blocks correspond to a decomposition of the stable reduction type into two connected parts both of arithmetic genus $1$, intersecting twice.
        \item The hyperelliptic blocks correspond to a hyperelliptic stable reduction type. In the case of a {\bf TCu}-block, the $8$ points of the Cayley octad become the $8$ Weierstrass points of the reduced curve. This follows essentially from \cite[Sec.~IX.3]{DolgachevOrtland}.
    \end{itemize}
\end{remark}

The remainder of this introduction will make the definitions and conjectures above precise, with the aim of being largely self-contained. Subsec.~\ref{subsec:mainresults} gives a detailed description of the degenerations in play, and how they decompose into building blocks (see Tab.~\ref{tab:introblocksdefinition}). Subsec.~\ref{subsec:modelsandoctads} discusses how changing coordinates on $\P^3$ and changing choice of Cayley octad impacts both the resulting degenerations and octad picture. Subsec.~\ref{subsec:introcombinatorialanalysis} gives substance to the explicit map~\eqref{eq:explicit-map}, with a detailed recipe that produces a stable reduction type from a given octad picture. This subsection also discusses the perfect combinatorial matching between octad pictures and stable reduction types.

\subsection{Octad Pictures}
\label{subsec:mainresults}
For the purposes of the introduction, let $K$ be a non-archimedean local field of residue characteristic $p \neq 2$. Suppose that $C/K$ is a smooth plane quartic and $O$ is
one of its Cayley octads. We note that the points of $O$ need not be defined over $K$, and our valuation is on $\bar{K}$ (normalised with respect to $K$).

To analyse the $p$-adic configuration of the points of $O$ in $\P^3$, we consider
valuations of the Pl\"ucker coordinates they define. Recall in particular that the Pl\"ucker coordinate of four points vanishes if and only if the four points are themselves coplanar, and so this measurement exhibits the extent to which each of the 70 possible subsets of 4 points are $p$-adically `nearly' coplanar (and thus lie on a plane over the residue field). The analogous construction for the cluster picture of a hyperelliptic curve $y^2 = f(x)$ is measuring the pairwise distances between the roots of $f$.

There are, however, other special arrangements of points in the Cayley octad we would like to measure; two or more points colliding, three or more points becoming collinear, and all eight points lying on a twisted cubic curve. The first two of these can in fact be measured from the Pl\"ucker coordinates alone (see Ex.~\ref{ex:introexample}, for instance). The third of these requires another $p$-adic measurement, the \textit{twisted cubic index}.

\begin{definition*}[see Def.~\ref{def:valuation-data}]
    Let $O = ( O_\oA, O_\oB, \ldots, O_\oH )$ be an $8$-tuple (an
    \textit{octad}) of points in $\P^3$ over $\overline{K}$.
    Affine representatives for the four coordinates of each of the 8 points are
    chosen so that their minimum valuation is 0 (recall that the
    valuation is normalised with respect to $K$).
    Let $P = \{ S \subset \{ \oA, \ldots, \oH \} \, \colon \, |S| = 4 \}$, and
    introduce the formal symbol $\dagger$.

    The \textit{valuation data} of $O$ over $K$ is a function $v \colon P \cup \{ \dagger \} \to \Q \colon S \mapsto v_S$, where:
    \begin{itemize}
    \item For each $S = \{i,j,k,l\} \in P$, the value $v_S$ is the valuation
      of the Pl\"ucker coordinate $\det(O_i | O_j | O_k | O_l)$.
    \item The value $v_\dagger$ is the \textit{twisted cubic index}, a non-negative number indicating with what multiplicity the points in $O$ lie on a twisted cubic curve.
\end{itemize}
\end{definition*}

Where convenient we also consider $v$ as a pair $(w,t)$, where $w \in \Q^P$ and $t = v_\dagger$.

\begin{example}
\label{ex:introexample}
Let $K = \Q_7$, and let $O = ( O_\oA, O_\oB, \ldots, O_\oH )$ be the Cayley octad with coordinates
\begin{displaymath}
      \arraycolsep=0.4\arraycolsep%
      \mbox{\scriptsize
        $\begin{bmatrix}
          1 & 0 & 0 & 0 & 1 & 1 & 7 & 5 \\
          0 & 1 & 0 & 0 & 1 & 7 & 10 & -34  \\
          0 & 0 & 1 & 0 & 1 & 14 & 8 & 24  \\
          0 & 0 & 0 & 1 & 1 & -7 & 9 & 11
        \end{bmatrix}$}
    \end{displaymath}
The corresponding plane quartic is given by
\footnotesize
\begin{align*}
C \colon \, 115600 & x^4 + 544000 x^3y + 1031680 x^2y^2 + 921600 xy^3 + 331776y^4 + 151640x^3z + 526816x^2yz + 1122432xy^2z \\
&+ 870912y^3z - 152231x^2z^2 + 118680xyz^2 + 913680y^2z^2 - 132462xz^3 + 449064yz^3 + 88209z^4.
\end{align*}
\normalsize
 There are sixteen quadruples of points of the Cayley octad that become coplanar over $\F_7$. Fifteen of these quadruples are of the form $O_\oA, O_\oF, O_\star, O_\star$; this is explained by $O_\oA$ and $O_\oF$ coinciding over $\F_7$. The sixteenth of these coplanar quadruples over $\F_7$ consists of the points with indices $\oB,\oC,\oD,$ and $\oG$; the valuation of its Pl\"ucker coordinate is $1$. Its complementary quadruple, the points with indices $\oA,\oE,\oF,$ and $\oH$, is coplanar with valuation $2$. The points do not lie on a twisted cubic over $\F_7$. Thus this valuation data is explained by the points $O_\oA$ and $O_\oF$ coinciding over $\F_7$ (to order $1$), and the complementary quadruples of points indexed by $\oA\oE\oF\oH$ and $\oB\oC\oD\oG$ lying on planes (also to order $1$).
\end{example}

\begin{conjecture*}[see Conj.~\ref{conj:VDUnieuqOctadDiagram} and~\ref{conj:SpecialFibreOfTheStableModel}]
    Let $C/K$ be a smooth plane quartic over a non-archimedean local field of residue characteristic $p \neq 2$, with Cayley octad $O$. Then the valuation data of $O$ determines the stable reduction type of $C/K$.
\end{conjecture*}

\begin{example}
    Let $C/\Q_7$ be the curve of Ex.~\ref{ex:introexample}. The stable
    reduction type of $C$ is a genus $1$ curve with two-self intersections. If
    the above conjecture is true, then any other octad over any other non-archimedean local field of residue characteristic other than $2$ with same valuation data gives rise to the same stable reduction type of the underlying curve.
\end{example}

Observe that in Ex.~\ref{ex:introexample}, the valuation data admits a natural, combinatorial description, being the result of two geometric degenerations; two points coinciding over $\F_7$, and two complementary quadruples of points becoming coplanar over $\F_7$. Below, we generalise this phenomenon. For simplicity, however, we restrict our attention to Cayley octads having five points in general position assumed to be $(1:0:0:0)$, $\ldots$ $(0:0:0:1)$ and $(1:1:1:1)$. This is without loss of generality, in the sense that all regular Cayley octads can be normalised by a $\PGL$-transformation (\textit{normalised} octads).

The valuation data of normalised Cayley octads can be written as a sum of certain \textit{building blocks}. These blocks can be interpreted as geometric degenerations in the Cayley octad over the residue field of $K$, with the simplest ones (as seen in Ex.~\ref{ex:introexample}) corresponding to two points coinciding ($\bm{\alpha}_{1a}$-blocks) or to four
points becoming coplanar ($\bm{\alpha}_{2a}$-blocks). We introduce $19$ such building blocks (up to choice of labelling) in Tab.~\ref{tab:introblocksdefinition}. In order to define these blocks, we introduce the nomenclature of Def.~\ref{def:standard-valuation-data}. We emphasise that building blocks are to octad diagrams what clusters are to cluster pictures.

\begin{definition}\label{def:standard-valuation-data}
Let $T \subset \{ \oA, \ldots, \oH \}$, and let $P = \{ S \subset \{ \oA, \ldots, \oH \} \, \colon \, |S| = 4 \}$. We define the following vectors:
\begin{itemize}
\item The point coincidence vector $v_{\pt}^T$ with values $v_{\pt, S}^T= \max(0, |S \cap T|-1)$ for $S \in P$, and twisted cubic index $v_{\pt,\dagger}^T = 0$. This is the valuation data for an octad in which the points of $T$ coincide up to order 1, and no other coincidences occur.
\item The line coincidence vector $v_{\ln}^T$ with values $v_{\ln, S}^T = \max(0, |S \cap T| - 2)$ for $S \in P$, and twisted cubic index $v_{\pt,\dagger}^T = 0$. This is the valuation data for an octad in which the points of $T$ lie on a line with order 1, and no other coincidences occur.
\item The plane coincidence vector $v_{\pl}^T$ with values $v_{\pl, S}^T = \max(0, |S \cap T| - 3)$ for $S \in P$, and twisted cubic index $v_{\pt,\dagger}^T = 0$. This is the valuation data for an octad in which the points of $T$ lie on a plane with order 1, and no other coincidences occur.
\item The twisted cubic coincidence vector $v_{\tc}$ with values $v_{\tc,S} = 0$ for $S \in P$ and twisted cubic index $v_{\tc,\dagger} = 1$. This is the valuation data for a Cayley octad in which the 8 points lie on a twisted cubic with order 1, and no other coincidences occur.
\end{itemize}
The building blocks in Tab.~\ref{tab:introblocksdefinition} are formed by combining some of these vectors. In particular, given two valuation data $v$ and $v'$, we define
\begin{itemize}
\item their \textit{sum}, which we denote by $v+v'$, by taking the componentwise sum, and
\item their {\em maximum}, which we denote by $\max(v,v')$, by taking the componentwise maximum.
\end{itemize}
\end{definition}

\newcolumntype{P}[1]{>{\centering\arraybackslash}p{#1}}
\newcommand\BBlockLine[1]{
  \multicolumn{5}{c}{
    \begin{tabular}{p{0.3\hsize}P{0.24\hsize}p{0.45\hsize}}
      \hrule width \hsize \kern 2pt \hrule width \hsize
      & {\centering #1}
      &\hrule width \hsize \kern 2pt \hrule width \hsize
    \end{tabular}
  }\\[-1cm]
}

\begin{table}[htbp]
  \centering
  \resizebox{!}{10cm}{%
  {\renewcommand{\arraystretch}{1.8}
    \begin{tabular}{lcccc}
      Block
      & Valuation Data
      & Subspace/Subset
      & Picture\\[-0.5cm]
      \BBlockLine{$\bm{\alpha}$-blocks}
      ${\bm{\alpha}}_{1{\mathrm{a}}}^{\oA\oB}$
      &   $v_{\pt}^{\oA\oB}$
      & \multirow{2}{*}{$\langle \oA\oB \rangle $}
      & \multirow{2}{*}{\resizebox{1.8cm}{!}{
        \tikzsetnextfilename{TypeTWintro}
        \begin{tikzpicture}[scale=0.25]
          \ptslabel \twinbig{0}
        \end{tikzpicture}
        }}\\
      ${\bm{\alpha}}_{1{\mathrm{b}}}^{\oA\oB}$
      &    $v_\pl^{\oC\oD\oE\oF\oG\oH}$
      &
      &                   \\ \hline
      ${\bm{\alpha}}_{2{\mathrm{a}}}^{\oA\oB\oC\oD}$
      &   $v_\pl^{\oA\oB\oC\oD} + v_\pl^{\oE\oF\oG\oH}$
      & \multirow{2}{*}{$\langle \oA\oB\oC\oD \rangle$}
      & \multirow{2}{*}{\resizebox{2cm}{!}{
        \tikzsetnextfilename{TypePLintro}
        \begin{tikzpicture}[scale=0.25]
          \ptslabel \plane{0}
        \end{tikzpicture}
        }}\\
      ${\bm{\alpha}}_{2{\mathrm{b}}}^{\oA\oB\oC\oD}$
      &     $v_\pl^{\oA\oB\oC\oD} + v_{\pt}^{\oA\oB\oC\oD}$
      &
      & \\[-0.3cm]
      \BBlockLine{$\bm{\chi}$-blocks}

      ${\bm{\chi}}_{1\mathrm{a}}^{\oA\oB|\oC\oD\oE|\oF\oG\oH}$
      & $\max(v_\ln^{\oC\oD\oE} + v_\ln^{\oF\oG\oH},
        v_\pl^{\oC\oD\oE\oF\oG\oH})$
      & \multirow{3}{*}{$\langle \oA\oB, \oA\oC\oD\oE \rangle $}
      & \multirow{3}{*}{\tikzsetnextfilename{TAintro} \begin{tikzpicture}[scale=0.25] \ptslabel \TA{0} \end{tikzpicture}} \\
      ${\bm{\chi}}_{1\mathrm{b}}^{\oA\oB|\oC\oD\oE|\oF\oG\oH}$
      &    $\max(v_\pt^{\oA\oB}, v_\pl^{\oA\oB\oC\oD\oE} +
        v_\pl^{\oA\oB\oF\oG\oH})$
      &
      &                   \\
      ${\bm{\chi}}_{1\mathrm{c}}^{\oA\oB|\oC\oD\oE|\oF\oG\oH}$
      &    $\max(v_\ln^{\oA\oB\oC\oD\oE} + v_\ln^{\oC\oD\oE},
        v_\pt^{\oC\oD\oE})$
      &
      &                   \\\hline
      ${\bm{\chi}}_{2\mathrm{a}}^{\oA\oB\oC}$
      &   $v_\ln^{\oD\oE\oF\oG\oH} + v_\pl^{\oD\oE\oF\oG\oH}$
      & \multirow{3}{*}{$\langle \oA\oB, \oA\oC \rangle$}
      & \multirow{3}{*}{\tikzsetnextfilename{TBintro} \begin{tikzpicture}[scale=0.25] \ptslabel \TB{0} \end{tikzpicture}} \\
      ${\bm{\chi}}_{2\mathrm{b}}^{\oA\oB\oC}$
      &     $v_\ln^{\oA\oB\oC} + v_\pt^{\oA\oB\oC}$
      &
      &                  \\
      ${\bm{\chi}}_{2\mathrm{c}}^{\oA\oB\oC}$
      &     $v_\ln^{\oA\oB\oC} + v_\pl^{\oD\oE\oF\oG\oH}$
      &
      & \\[-0.3cm]
      \BBlockLine{$\bm{\phi}$-blocks}
       ${\bm{\phi}}_{1\mathrm{a}}^{\CAApts}$
      &   $v_{\ln}^{\oA\oB\oE\oF} + v_{\ln}^{\oA\oB\oG\oH} +
        v_{\pl}^{\oA\oB\oC\oD} + v_{\pl}^{\oE\oF\oG\oH}$
      & \multirow{2}{*}{
\begin{tabular}{l}
$\langle \oA\oB, \oC\oD, \oA\oC\oE\oF \rangle $ or\\[-0.5cm]\hfill $\langle \oE\oF, \oG\oH, \oA\oB\oE\oG \rangle$
\end{tabular}}

      & \multirow{2}{*}{\tikzsetnextfilename{CandyAintro} \begin{tikzpicture}[scale=0.23] \ptslabelcross \CA \end{tikzpicture}} \\
      ${\bm{\phi}}_{1\mathrm{b}}^{\CAApts}$
      &    $\max(v_{\pt}^{\oA\oB} + v_{\pt}^{\oC\oD},
        v_{\pl}^{\oA\oB\oC\oD\oE\oF} + v_{\pl}^{\oA\oB\oC\oD\oG\oH})$
      &
      & \\\hline
      $\bm{\phi}_{2\mathrm{a}}^{\oA\oB\oC|\oF\oG\oH}$
      &   $\max(v_\ln^{\oA\oB\oC} + v_\ln^{\oF\oG\oH}, v_\pl^{\oA\oB\oC\oD\oE}
        + v_\pl^{\oD\oE\oF\oG\oH})$
      & \multirow{3}{*}{\begin{tabular}{l}
$\langle \oA\oB, \oB\oC, \oD\oE \rangle$ or\\[-0.5cm]\hfill $\langle \oF\oG, \oF\oH, \oD\oE \rangle$\end{tabular}
      }
      & \multirow{3}{*}{ \tikzsetnextfilename{CandyBintro} \begin{tikzpicture}[scale=0.23] \ptslabel \CB{3} \end{tikzpicture}}  \\
      $\bm{\phi}_{2\mathrm{b}}^{\oA\oB\oC|\oF\oG\oH}$
      &     $v_\pt^{\oD\oE} + v_\ln^{\oA\oB\oC\oD} + v_\ln^{\oA\oB\oC\oE}$
      &
      & \\
      $\bm{\phi}_{2\mathrm{c}}^{\oA\oB\oC|\oF\oG\oH}$
      &     $v_\ln^{\oA\oB\oC\oD} + v_\ln^{\oA\oB\oC\oE} +
        v_\pl^{\oA\oB\oC\oF\oG\oH}$
      &
      & \\ \hline
      $\bm{\phi}_{3\mathrm{a}}^{\oA\oB\oC\oD}$
      &   $v_\ln^{\oA\oB\oC\oD} + v_\pl^{\oA\oB\oC\oD} + v_\pl^{\oE\oF\oG\oH}$
      & \multirow{2}{*}{\begin{tabular}{l}
$\langle \oA\oB, \oA\oC, \oA\oD \rangle$ or\\[-0.5cm]\hfill  $\langle \oE\oF, \oE\oG, \oE\oH \rangle$\end{tabular}
      }
      & \multirow{2}{*}{\tikzsetnextfilename{CandyCintro} \begin{tikzpicture}[scale=0.23] \ptslabel \CC{0} \end{tikzpicture}} \\
      $\bm{\phi}_{3\mathrm{b}}^{\oA\oB\oC\oD}$
      &   $v_\pt^{\oA\oB\oC\oD} + v_\ln^{\oA\oB\oC\oD} + v_\pl^{\oA\oB\oC\oD}$
      & \\[-0.3cm]
      \BBlockLine{hyperelliptic blocks}
      \begin{tabular}{l}
      $\mathrm{{\bf TCu}}$\\[2cm]
      \end{tabular}
      &
        \begin{tabular}{l}
          $v_{\tc}$\\[2cm]
        \end{tabular}
      &
        \begin{tabular}{l}
          $\{\,0\,\} \,\cup\, \{\, S \subset \{ \oA, \ldots \oH \} \,:\, |S| = 2\}$\\[2cm]
        \end{tabular}

      & \tikzsetnextfilename{TCuintro}
        \begin{tikzpicture}[scale=0.18]  \node at (0,5) {}; \ptslabel \TCu \end{tikzpicture} \\[-1cm] \hline
      \begin{tabular}{l}
      $\mathrm{{\bf Line}}^{\oA\oB\oC\oD}$\\[1.5cm]
      \end{tabular}

      &
        \begin{tabular}{l}
          $v_\ln^{\oA\oB\oC\oD}$\\[1.5cm]
        \end{tabular}
      & $\begin{array}{l}\{\,0\,\} \ \cup
          \{\ S \cup T \, \colon \, S \subset \{\oA, \oB, \oC, \oD\},
           \,\\[-0.5cm] \hspace*{0.5cm} T \subset \{ \oE, \oF, \oG, \oH \}, \, |S| - |T| = \pm 2 \}\\[1.5cm]
          \end{array}$
      & \tikzsetnextfilename{Lineintro}
        \begin{tikzpicture}[scale=0.2] \node at (0,5) {}; \ptslabel
          \HE{0} \end{tikzpicture}\\[-1cm]

    \end{tabular}}%
    }
  \caption{Building blocks, their associated subspaces and pictures (see Sec.~\ref{sec:formalanalysisoctadpictures})}
  \label{tab:introblocksdefinition}
\end{table}

\begin{example}
\label{ex:introexampleblockdecomp}
Continuing Ex.~\ref{ex:introexample}, the valuation data of $O/\Q_7$ decomposes as $(v_{\pt}^{\oA\oF}) + (v_{\pl}^{\oA\oE\oF\oH} + v_{\pl}^{\oB\oC\oD\oG}) = \bm{\alpha}_{1a}^{\oA\oF} + \bm{\alpha}_{2a}^{\oA\oE\oF\oH}$. We use the pictures of Tab.~\ref{tab:introblocksdefinition} (in this case the two pictures given in Fig.~\ref{fig:introexamplepictures}), to associate a picture to the entire valuation data, simply by overlaying the two (Fig.~\ref{fig:introexampleoctadpicture}).

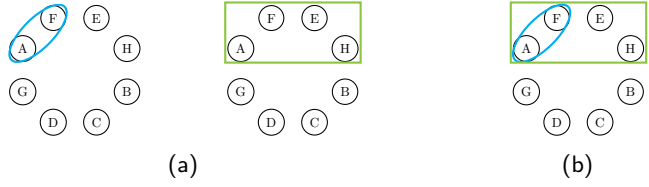
\begin{figure}[htbp]
    \centering
    \begin{subfigure}[b]{0.35\linewidth}
    \centering
         \tikzsetnextfilename{IntroOctPicA1v2}
        \begin{tikzpicture}[scale=0.25]
            \ptscustomlabel{F}{E}{H}{B}{C}{D}{G}{A}
            \twinbig{7}
        \end{tikzpicture}
        \qquad\,
        \tikzsetnextfilename{IntroOctPicA2}
        \begin{tikzpicture}[scale=0.25]
             \ptscustomlabel{F}{E}{H}{B}{C}{D}{G}{A}
             \plane{7}
        \end{tikzpicture}
        \caption{}
        \label{fig:introexamplepictures}
    \end{subfigure}
    \begin{subfigure}[b]{0.35\linewidth}
    \centering
        \tikzsetnextfilename{IntroOctPicB}
        \begin{tikzpicture}[scale=0.25]
             \ptscustomlabel{F}{E}{H}{B}{C}{D}{G}{A}
             \twinbig{7}
             \plane{7}
        \end{tikzpicture}
        \caption{}
        \label{fig:introexampleoctadpicture}
      \end{subfigure}
      \caption{Valuation data decomposition}
\end{figure}

\end{example}

Each building block is represented by a picture on the $8$ points $\oA, \ldots, \oH$, and this is typically an easy way of rendering valuation data. Additionally, each building block is associated with a subspace or subset of the vector space $E_3$ (see Rem.~\ref{rmk:cremonatransformations}). These subspaces are significant ingredients in the explicit formation of the stable reduction type, see Sec.~\ref{subsec:introcombinatorialanalysis}.

\begin{figure}[htbp]
\begin{subfigure}[c]{0.35\linewidth}
  \centering
  \resizebox{0.9\linewidth}{!}{%
  {\renewcommand{\arraystretch}{2}
    \begin{tabular}{lclcc}
      Block
      & Auxiliary Blocks\\
      \hline\hline
       ${\bm{\phi}}_{1\mathrm{a}}^{\CAApts}$
      & \multirow{2}{*}{ \begin{tabular}{c}${\bm{\alpha}}_{2 \star}^{\oA\oB\oC\oD}, \textbf{Line}^{\oA\oB\oE\oF}, \textbf{Line}^{\oA\oB\oG\oH}$,\\ [-0.5cm]$\textbf{Line}^{\oC\oD\oE\oF}, \textbf{Line}^{\oC\oD\oG\oH}$\end{tabular}}
      \\
      ${\bm{\phi}}_{1\mathrm{b}}^{\CAApts}$
      & \\\hline
      $\bm{\phi}_{2\mathrm{a}}^{\oA\oB\oC|\oF\oG\oH}$
      & \multirow{3}{*}{ \begin{tabular}{c}${\bm{\alpha}}_{1 \star }^{\oD\oE}, \textbf{Line}^{\oA\oB\oC\oD}, \textbf{Line}^{\oA\oB\oC\oE}$,\\[-0.5cm] $\textbf{Line}^{\oD\oF\oG\oH}, \textbf{Line}^{\oE\oF\oG\oH}$\end{tabular}}
 \\
      $\bm{\phi}_{2\mathrm{b}}^{\oA\oB\oC|\oF\oG\oH}$
      & \\
      $\bm{\phi}_{2\mathrm{c}}^{\oA\oB\oC|\oF\oG\oH}$
      & \\ \hline
      $\bm{\phi}_{3\mathrm{a}}^{\oA\oB\oC\oD}$
      & \multirow{2}{*}{\hspace{0.3cm}\begin{tabular}{c}${\bm{\alpha}}_{2 \star}^{\oA\oB\oC\oD}$, $\textbf{TCu}$,\\[-0.5cm] $\textbf{Line}^{\oA\oB\oC\oD}, \textbf{Line}^{\oE\oF\oG\oH}$\end{tabular}}
      \\
      $\bm{\phi}_{3\mathrm{b}}^{\oA\oB\oC\oD}$
      & \\[-0.3cm]
    \end{tabular}}%
    }
  \caption{Definitions}
  \label{fig:auxdefinition}
\end{subfigure}
\begin{subfigure}[c]{0.6\linewidth}
    \centering
    \resizebox{1.0\linewidth}{!}{\tikzsetnextfilename{AuxiliaryBlocks}
    \begin{tikzpicture}
      \node at (-7.5,0) (a10) {\begin{tikzpicture}[scale=0.15]
          \pts\HE{1}\CA
        \end{tikzpicture}};
      `   \node at (-5.5,0) (a11) {\begin{tikzpicture}[scale=0.15]
          \pts\midplane{7}\CA
        \end{tikzpicture}};
      \node at (-2.6,0) (a12) {\begin{tikzpicture}[scale=0.15]
          \pts\CA
        \end{tikzpicture}};
      \draw[->,shorten <=0.2cm,shorten >=0.2cm,line width=0.4pt]
      (a11) to (a12);
      \draw[line width=0.4pt] (-6.7,-0.5) to (-6.3,0.5);
      \node at (0,0) (a20) {\begin{tikzpicture}[scale=0.15]
          \pts\twin{0}\CB{0}
        \end{tikzpicture}};
      \node at (2,0) (a21) {\begin{tikzpicture}[scale=0.15]
          \pts\HE{1}\CB{0}
        \end{tikzpicture}};
      \node at (4.5,0) (a22) {\begin{tikzpicture}[scale=0.15]
          \pts\CB{0}
        \end{tikzpicture}};
      \draw[->,shorten <=0.2cm,shorten >=0.2cm,line width=0.4pt]
      (a21) to (a22);
      \draw[line width=0.4pt] (0.8,-0.5) to (1.2,0.5);
      \node at (-5.4,-2) (a3) {\begin{tikzpicture}[scale=0.15]
          \pts\plane{0}\CC{0}
        \end{tikzpicture}};
      \node at (-3.0,-2) (a30) {\begin{tikzpicture}[scale=0.15]
          \pts\TCu\CC{0}
        \end{tikzpicture}};
      \node at (-0.6,-2) (a31) {\begin{tikzpicture}[scale=0.15]
          \pts\HE{0}\CC{0}
        \end{tikzpicture}};
      \node at (2.2,-2) (a32) {\begin{tikzpicture}[scale=0.15]
          \pts\CC{0}
        \end{tikzpicture}};
      \draw[->,shorten <=0.2cm,shorten >=0.2cm,line width=0.4pt]
      (a31) to (a32);
      \draw[line width=0.4pt] (-4.4,-2.5) to (-4.0,-1.5);
      \draw[line width=0.4pt] (-2.0,-2.5) to (-1.6,-1.5);
    \end{tikzpicture}
  }
    \caption{Equivalences (modulo labelling)}
    \label{fig:introauxiliarysuppression}
\end{subfigure}
\caption{Auxiliary blocks}
\label{fig:introaux}
\end{figure}

The building blocks carry a pairwise compatibility condition that one may think of as equivalent to the condition, on the level of cluster pictures, that two clusters must either be contained one inside the other, or be mutually disjoint (see Tab.~\ref{tab:twoblockssubspaces} for the condition given pairwise in terms of pictures, but note that, on the level of the corresponding subspaces, this amounts to simple inclusion conditions, see Def.~\ref{def:admissible}).

\begin{remark}
  For technical reasons, $\bm{\phi}$-blocks are associated in
  Section~\ref{sec:phi-blocks} with a class of
  \textit{auxiliary} $\bm{\alpha}$-blocks and two classes of
  \textit{auxiliary} hyperelliptic blocks (see Fig.~\ref{fig:introaux}).
  Valuation data containing a $\bm{\phi}$-block may also contain these
  auxiliary blocks, however these do not have any bearing on the stable
  reduction type, and we omit them from the induced octad picture
  (Fig.~\ref{fig:introauxiliarysuppression}).

  By convention we say that auxiliary blocks are compatible with their
  associated $\bm{\phi}$-blocks, but in line with this remark we omit them
  from the table. Note also that two blocks inducing the same picture are
  considered compatible.

  We expect that in the semistable case, these auxiliary blocks contain
  information regarding chains of $\P^1$s between components, but this exceeds
  the scope of the present work. In a way, these technical details can be
  overlooked on first readings.
\end{remark}

A set $\mathcal{B} = \{ B_1, \ldots, B_n \}$ of building blocks is said to be compatible if all pairs $B_i, B_j$ meet this condition. If {the} valuation data
is written as a sum of compatible building blocks, then this is said to be a \textit{compatible block decomposition} (see Def.~\ref{def:compatibleblocks}). We view this as the `correct' way to decompose valuation data, as in Fig.~\ref{fig:introexampleoctadpicture}. As a (non-)example, valuation data cannot decompose compatibly to give the picture in Fig.~\ref{fig:intrononcompatibleexample}.

\begin{table}[htbp]
  \centering
  \begin{tabular}{l>{\centering\arraybackslash}p{0.6\textwidth}}
    Blocks                               & Compatible configurations \\
    \hline\hline
    \rule{0pt}{0.8cm}Two $\bm{\alpha}$-blocks &
    \parbox[m]{0.6\textwidth}{\centering
    \tikzsetnextfilename{AdmisTwoObjA}
    \begin{tikzpicture}[scale=0.15] \pts \twin{7} \twin{1} \end{tikzpicture}
    \quad
    \tikzsetnextfilename{AdmisTwoObjB}
    \begin{tikzpicture}[scale=0.15] \pts \twin{7} \plane{7} \end{tikzpicture}
    \quad
    \tikzsetnextfilename{AdmisTwoObjC}
    \begin{tikzpicture}[scale=0.15] \pts \plane{5} \plane{7} \end{tikzpicture}
    } \\[0.6cm] \hline
    \rule{0pt}{1.4cm} $\bm{\alpha}$-block and $\bm{\chi}$-block  &
    \parbox[m]{0.4\textwidth}{\centering
      \tikzsetnextfilename{AdmisObjTypA}
      \begin{tikzpicture}[scale=0.15] \pts \TA{0} \twin{0} \end{tikzpicture}
      \quad
      \tikzsetnextfilename{AdmisObjTypB}
      \begin{tikzpicture}[scale=0.15] \pts \twin{1} \TB{1} \end{tikzpicture}
      \quad
      \tikzsetnextfilename{AdmisObjTypC}
      \begin{tikzpicture}[scale=0.15] \pts \TA{0} \plane{1} \end{tikzpicture}
      \quad
      \tikzsetnextfilename{AdmisObjTypD}
      \begin{tikzpicture}[scale=0.15] \pts \TA{0} \twin{2} \end{tikzpicture}
      \quad
      \tikzsetnextfilename{AdmisObjTypE}
      \begin{tikzpicture}[scale=0.15] \pts \TA{0} \plane{7} \end{tikzpicture}
      \quad
      \tikzsetnextfilename{AdmisObjTypF}
      \begin{tikzpicture}[scale=0.15] \pts \TB{1} \twin{7} \end{tikzpicture}
      \quad
      \tikzsetnextfilename{AdmisObjTypG}
      \begin{tikzpicture}[scale=0.15] \pts \TB{1} \plane{1} \end{tikzpicture}
    } \\[1.1cm] \hline
    \rule{0pt}{0.8cm}Two $\bm{\chi}$-blocks        &
    \parbox[m]{0.6\textwidth}{\centering
    \tikzsetnextfilename{AdmisTwoTypA}
    \begin{tikzpicture}[scale=0.15] \pts \TA{7} \TA{1} \end{tikzpicture}
    \quad
    \tikzsetnextfilename{AdmisTwoTypB}
    \begin{tikzpicture}[scale=0.15] \pts \TA{7} \TB{1} \end{tikzpicture}
    \quad
    \tikzsetnextfilename{AdmisTwoTypC}
    \begin{tikzpicture}[scale=0.15] \pts \TB{2} \TB{5} \end{tikzpicture}
    }
    \\[0.6cm] \hline
    \rule{0pt}{0.9cm} $\bm{\phi}$-block and $\bm{\alpha}$-block   &
    \parbox[m]{0.6\textwidth}{\centering
    \tikzsetnextfilename{AdmisCandyObjectA}
    \begin{tikzpicture}[scale=0.15] \pts \CA \twin{7} \end{tikzpicture}
    \quad
    \tikzsetnextfilename{AdmisCandyObjectB}
    \begin{tikzpicture}[scale=0.15] \pts \CA \plane{6} \end{tikzpicture}
    \quad
    \tikzsetnextfilename{AdmisCandyObjectC}
    \begin{tikzpicture}[scale=0.15] \pts \CB{0} \twin{2} \end{tikzpicture}
    \quad
    \tikzsetnextfilename{AdmisCandyObjectD}
    \begin{tikzpicture}[scale=0.15] \pts \CB{0} \plane{4} \end{tikzpicture}
    \quad
    \tikzsetnextfilename{AdmisCandyObjectE}
    \begin{tikzpicture}[scale=0.15] \pts \CC{7} \twin{7} \end{tikzpicture}
    } \\[0.6cm] \hline
    \rule{0pt}{0.8cm} $\bm{\phi}$-block and $\bm{\chi}$-block &
    \parbox[m]{0.6\textwidth}{\centering
    \tikzsetnextfilename{AdmisCandyTypA}
    \begin{tikzpicture}[scale=0.15] \pts \CA \TA{7} \end{tikzpicture}
    \quad
    \tikzsetnextfilename{AdmisCandyTypB}
    \begin{tikzpicture}[scale=0.15] \pts \CB{4} \TA{7} \end{tikzpicture}
    \quad
    \tikzsetnextfilename{AdmisCandyTypC}
    \begin{tikzpicture}[scale=0.15] \pts \CB{0} \TB{5} \end{tikzpicture}
    \quad
    \tikzsetnextfilename{AdmisCandyTypD}
    \begin{tikzpicture}[scale=0.15] \pts \CC{1} \TB{1} \end{tikzpicture}
    }
    \\[0.6cm] \hline
    \rule{0pt}{0.8cm} HE block and $\bm{\alpha}$-block &
    \parbox[m]{0.6\textwidth}{\centering
    \tikzsetnextfilename{AdmisHyperAlpha1}\begin{tikzpicture}[scale=0.15] \pts \twin{0} \TCu \end{tikzpicture}
    \quad
    \tikzsetnextfilename{AdmisHyperAlpha2}\begin{tikzpicture}[scale=0.15] \pts \twin{0} \HE{0} \end{tikzpicture}
    \quad
    \tikzsetnextfilename{AdmisHyperAlpha3}\begin{tikzpicture}[scale=0.15] \pts \plane{1} \HE{0} \end{tikzpicture}
    }
    \\[0.6cm] \hline
    \rule{0pt}{0.8cm} HE block and $\bm{\chi}$-block &
    \parbox[m]{0.6\textwidth}{\centering
    \tikzsetnextfilename{AdmisHyperChi1}\begin{tikzpicture}[scale=0.15] \pts \TB{0} \TCu \end{tikzpicture}
    \quad
    \tikzsetnextfilename{AdmisHyperChi2}\begin{tikzpicture}[scale=0.15] \pts \TA{0} \HE{0} \end{tikzpicture}
    \quad
    \tikzsetnextfilename{AdmisHyperChi3}\begin{tikzpicture}[scale=0.15] \pts \TB{0} \HE{0} \end{tikzpicture}
    }
  \end{tabular}
  \caption{Pictures of any two compatible blocks (modulo labelling)}
  \label{tab:twoblockssubspaces}
\end{table}

\begin{figure}[htbp]
    \centering
    \tikzsetnextfilename{IntroNonCompatiblePicture}
    \begin{tikzpicture}[scale=0.25]
        \ptslabel
        \twinbig{0}\plane{1}
    \end{tikzpicture}
    \caption{Example of two incompatible blocks}
    \label{fig:intrononcompatibleexample}
\end{figure}
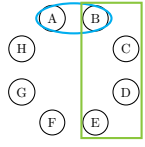

\begin{conjecture*}[see Conj.~\ref{conj:VDUnieuqOctadDiagram}]
    Let $O/K$ be a normalised Cayley octad over a non-archimedean local field of residue characteristic $p \neq 2$. Then the valuation data of $O/K$ admits a unique compatible block decomposition.
\end{conjecture*}

    We can prove that a compatible block decomposition is essentially unique, provided that there exists such a decomposition, see Thm.~\ref{thm:UniqueDecomposition}. Thus there is a canonical collection of building blocks attached to a Cayley octad over $K$. Just as a cluster picture of a curve $y^2 = f(x)$ is the totality of the clusters of the roots of $f$, the octad picture of a Cayley octad $O$ is the totality of the building blocks in the compatible block decomposition of $O$.

\begin{definition*}[see Def.~\ref{def:blockdecompositiondiagram}~and~\ref{def:octadpictures}]
    Let $\mathcal{B}$ be a set of compatible blocks. Then the \textit{octad picture} associated to $\mathcal{B}$ is the picture achieved by simultaneously overlaying the pictures attached to each $B \in \mathcal{B}$.
    The octad picture, $d_K(O)$, of a normalised Cayley octad $O/K$ over a non-archimedean local field of residue characteristic $p \neq 2$, is the octad picture of its set of compatible blocks, given by Conj.~\ref{conj:VDUnieuqOctadDiagram}.

    More generally an \textit{octad picture} is any picture associated to a set of compatible building blocks in this way.
\end{definition*}

\begin{remark}
    In the octad picture, we do not keep track of the multiplicity of the different building blocks in the valuation data of the octad $O$, which would be the analogue of the (relative) depth of a cluster. We conjecture that this multiplicity does not affect the stable reduction, but we expect that it does affect arithmetic invariants, as in the case of cluster pictures of hyperelliptic curves.
\end{remark}

This leads us to the central conjecture of the paper.

\begin{conjecture*}[see Conj.~\ref{conj:SpecialFibreOfTheStableModel}]
  The stable reduction type, including whether or not it is hyperelliptic, of any smooth plane quartic over any non-archimedean local field of residue characteristic $p \neq 2$ is determined by the octad picture of any of its Cayley octads.
\end{conjecture*}

\begin{example} \label{ex:110}
  Let $C/\Q_7$ and $O$ be as in Ex.~\ref{ex:introexample}. Then
  $d_{\Q_7}(O)$ is given in Fig.~\ref{fig:introexampleoctadpicture}.
  By Tab.~\ref{tab:smoctadsc02} we predict that the stable reduction type of
  $C/\Q_7$ is a genus one curve with two self-intersections. Indeed, we have
  verified computationally that this is the case.
\end{example}

\begin{remark}
   We remind the reader that a given plane quartic has $36$ Cayley octads (up to ordering of points, and the action of $\PGL$), and so any associated octad picture may not be unique. This is discussed more fully in Sec.~\ref{subsec:modelsandoctads}.
\end{remark}

\begin{remark}
    We have been able to generate examples of Cayley octads for all octad pictures, and so if Conj.~\ref{conj:SpecialFibreOfTheStableModel} is true, the full correspondence is given (up to labelling) in Tab.~\ref{tab:smoctadsc02},~\ref{tab:smoctadsc34}, and~\ref{tab:smoctadsc56} in App.~\ref{sec:special-fibres-octad-1}, and Tab.~\ref{tab:smhoctadsc02},~\ref{tab:smhoctadsc3},~\ref{tab:smhoctadsc4}, and~\ref{tab:smhoctadsc56} in App.~\ref{sec:special-fibres-octad-2}.
\end{remark}

In fact, we expect that one can equip $d_K(O)$ with more information in order to recover finer arithmetic data than the stable model. For example, when $C$ is semistable we anticipate that $d_K(O)$ (along with the multiplicities of the corresponding building blocks) determines the minimal regular model, as is the case with cluster pictures.%

\begin{remark}
  The correspondence of Tab.~\ref{tab:smoctadsc02}--\ref{tab:smhoctadsc56} is by decree, and we reserve discussion of an explicit correspondence to Sec.~\ref{subsec:introcombinatorialanalysis}.
Typically, however, one can determine the stable reduction type from the following rules. %
\begin{itemize}
    \item There is a $\bm{\phi}$-block if and only if the stable reduction type admits a partition into two unions of irreducible components, both connected with arithmetic genus $1$, and meeting at two points.
    \item The number of $\bm{\chi}$-blocks is the number of ways the stable reduction type may be partitioned into two unions of irreducible components, one connected with arithmetic genus $1$, the other connected with arithmetic genus $2$, and meeting at one point.
    \item The number of $\bm{\alpha}$-blocks is the number of remaining intersections.
    \item The stable reduction type is hyperelliptic if and only if there is a $\bm{\phi}$-block or a hyperelliptic block.
\end{itemize}

For instance, in Ex.~\ref{ex:introexampleblockdecomp} the valuation
data for $O$ decomposes into two $\bm{\alpha}$-blocks. The stable reduction type is thus non-hyperelliptic and contains two intersections. There are $5$ types matching this description (see Fig.~\ref{fig:stablemodelgraphquartic}), but only the type which consists of a genus one curve with two self-intersections corresponds to valuation data containing neither a $\bm{\chi}$-block nor a  $\bm{\phi}$-block. Therefore, the stable reduction type of $C/K$ is a genus one curve with two self-intersections.

Note that the rules above do not provide enough information to distinguish all cases. For instance, valuation data of the form $\bm{\chi} + \bm{\alpha}$ (where the indices are to be determined) does not discriminate between the stable reduction types \texttt{(2m)} and \texttt{(1me)} (see Fig.~\ref{fig:stablemodelgraphquartic}). One consists of a genus 2 curve intersecting a genus 0 curve with a self-intersection. In the other, a genus 1 curve intersects another genus 1 curve with a self-intersection. These ambiguities are resolved in Sec.~\ref{sec:formalanalysisoctadpictures}.
\end{remark}

\subsection{Octad Pictures as Invariants}
\label{subsec:modelsandoctads}

Our discussion so far has fixed a choice of Cayley octad of a plane quartic $C$. The stable reduction type of $C$ is independent of this choice and the action of $\PGL$ possibly changing the valuation data of the Cayley octad, and thus a natural question is to consider if the constructions and conjectures of Sec.~\ref{subsec:mainresults} respect these choices.

\begin{theorem*}[see Thm.~\ref{thm:pgl-action-on-block-decomposition}]
    Suppose $O, O'/K$ are two normalised Cayley octads over a non-archimedean local field of residue characteristic $p \neq 2$ which coincide under the action of $\PGL$. Then $d_K(O) = d_K(O')$.
\end{theorem*}

In particular, the octad picture is independent of the choice of which 5 points are normalised, and the octad picture acts as an invariant for $\PGL$-orbits of normalised Cayley octads.

Classically, one considers the action of the group $\textup{Sp}(6,2)$, described by Cremona transformations, on $\PGL$-equivalence
classes of Cayley octads of a given plane quartic $C$ (see Thm.~\ref{thm:regularoctad}). Pertinently, there is a  subgroup $S_8 \subset \textup{Sp}(6,2)$ acting by permuting the points of the octad, and the full action is transitive on the Cayley octads of $C$.

There is also an $\textup{Sp}(6,2)$ action on octad pictures. To define this, we introduce the vector space $E_3$.

\begin{definition}
\label{def:introE3}
Let $\Sigma_3 = \{ \oA, \oB, \oC, \oD, \oE, \oF, \oG, \oH \}$, and let $E_3$ be the set of subsets of $\Sigma_3$ of even cardinality, modulo the equivalence relation $I \sim \Sigma_3 \backslash I$. Then $E_3$ is a $6$-dimensional $\F_2$-vector space where addition is given by the symmetric sum $I + J = ( I \cup J ) \backslash ( I \cap J)$. Moreover, $E_3$ also carries the symplectic pairing $\langle I, J \rangle = | I \cap J | \mod{2}$. For ease we write the elements of $E_3$ as strings instead of sets.
\end{definition}

If $C$ is a plane quartic, then
$\Jac C[2] \simeq E_3$ as symplectic vector spaces. The symplectic group $\textup{Sp}(6,2)$ acts on $E_3$
as the linear transformations preserving the symplectic pairing.

Each octad picture $P$ induces a collection of subspaces of $E_3$ via the
association in Tab.~\ref{tab:introblocksdefinition} %
(see Def.~\ref{def:picturesubspacecorrespondence} and~\ref{def:picturesubspacescorrespondenceHE}). As
$\textup{Sp}(6,2)$ acts on subspaces of $E_3$, there is an induced action on
octad pictures. %

\begin{conjecture*}[see Conj.~\ref{conj:symplecticgroupactioncommutespictures}]
    Let $K$ be a non-archimedean local field of residue characteristic $p \neq 2$. Then the action of $\textup{Sp}(6,2)$ on $\PGL$-equivalence classes of Cayley octads commutes with the map $d_K$.
\end{conjecture*}

\begin{theorem*}[see Prop.~\ref{CreTransTwins},~\ref{CreTransTypes},~\ref{prop:CreTransCans} and~\ref{prop:CreTransHE}]
    Conj.~\ref{conj:symplecticgroupactioncommutespictures} is true in the case that an octad picture comes from a single building block.
\end{theorem*}

\begin{remark}
\label{rmk:intropictuestabletypebijection}
    Assuming both Conj.~\ref{conj:SpecialFibreOfTheStableModel} and Conj.~\ref{conj:symplecticgroupactioncommutespictures}, there is then a well-defined bijection
    $$
    \textup{Octad Pictures} / \textup{Sp}(6,2)  \,\, - \,\,  \textup{Exceptional Orbit} \, \longrightarrow \,  \textup{Stable Reduction Types} .
    $$
    Both sides of the bijection are given in Tab.~\ref{tab:smoctadsc02}--\ref{tab:smhoctadsc56}.

    The exceptional orbit of octad pictures is given by Cayley octads which can only be achieved
    over a field whose residue characteristic is $2$ (such a picture is obtained, for example, by the eight vertices of a cube). The corresponding octad picture consists of seven $\bm{\alpha}$-blocks, but there is no stable genus 3 curve with 7 nodes. See also Rem.~\ref{rmk:exceptionalpictures}.
\end{remark}

\begin{example}
Let $C/\Q_5$ be the curve
$$
x^3y + x^2y^2 - 4xy^3 + 2y^4 + x^3z + 3x^2yz - 3xy^2z +  x^2z^2 - 3xyz^2 + 3y^2z^2 - 4xz^3 + 2z^4 = 0\,.
$$
Then, of the $36$ Cayley octads of $C$, $16$ have valuation data of the form $\bm{\alpha}_{1a}$ or $\bm{\alpha}_{1b}$ and so have octad picture given in Fig.~\ref{fig:IntroOctadDiagrams} (left). Note that the valuation data $\bm{\alpha}_{1a}$ and $\bm{\alpha}_{1b}$ are $\textup{PGL}_3(\bar{\Q}_5)$-equivalent, and so cannot be distinguished from the octad picture. The valuation data of the remaining $20$ Cayley octads is either $\bm{\alpha}_{2a}$ or $\bm{\alpha}_{2b}$ (similarly $\textup{PGL}_3(\bar{\Q}_5)$-equivalent), and give the octad picture in Fig.~\ref{fig:IntroOctadDiagrams} (right).

Assuming our conjectures, any one of these
$36$ pictures is sufficient to determine the special fibre of the stable model
of $C$, whose dual graph is given in Fig.~\ref{fig:IntroDualGraph};
it is an irreducible curve of genus 2 with a self-intersection point.

\begin{figure}[htbp]
  \begin{subfigure}[b]{0.35\linewidth}
    \centering
    \tikzsetnextfilename{DualGraph3}
    \begin{tikzpicture}[scale=0.8]
        \node[shape=circle,draw=black,scale=0.6] (A) at (0,0) {2};
        \path  (A) edge [loop above,min distance = 2cm,in=135,out=45] node {} (A);
      \end{tikzpicture}
    \caption{Stable model (special fibre)}
    \label{fig:IntroDualGraph}
  \end{subfigure}
  \vspace{\floatsep}
  \begin{subfigure}[b]{0.35\linewidth}
    \centering
    \tikzsetnextfilename{TABOctadDnA}
      \begin{tikzpicture}[scale=0.12]
        \pts
        \twin{0}
        \oindex{16}
      \end{tikzpicture}
      \qquad
      \tikzsetnextfilename{TABOctadDnB}
      \begin{tikzpicture}[scale=0.12]
        \pts
        \plane{7}
        \oindex{20}
      \end{tikzpicture}
    \caption{Octad pictures}
    \label{fig:IntroOctadDiagrams}
  \end{subfigure}
  \vspace*{-0.5cm}
  \caption{Stable reduction of $C/\Q_5$}
\end{figure}
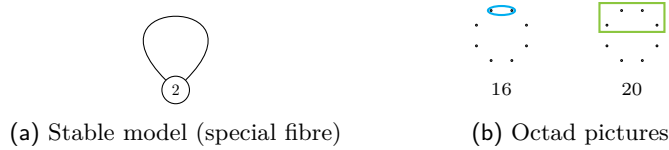

Note that, under our conjectures, any other plane quartic (over a non-archimedean local field with residue characteristic $p \neq 2$) which has the same stable type must also have the exact same set of $36$ octad pictures.
\end{example}

\subsection{Combinatorial Analysis}
\label{subsec:introcombinatorialanalysis}

The purpose of this section is to make the correspondence of Tab.~\ref{tab:smoctadsc02}--\ref{tab:smhoctadsc56} and the compatibility conditions in Tab.~\ref{tab:twoblockssubspaces} less ad hoc. The content of these tables is in fact much more natural in terms of the vector space $E_3$ (see Def.~\ref{def:introE3}), and a classification of octad pictures is readily achieved through this paradigm. We are lead to a framework that accounts explicitly and completely for the correspondence between octad pictures and stable reduction types.

\begin{definition*}[see Def.~\ref{def:admissible}]
Suppose $\mathcal{L}$ is a set whose elements are pairs $(X, X^\bot)$ of orthogonally complement subspaces of $E_3$, ordered so that $\dim(X) \leqslant \dim(X^\bot)$. Then, $\mathcal{L}$ is said to be \textit{compatible} when
\begin{itemize}
    \item $(\mathbf{0}, E_3) \in \mathcal{L}$,
    \item for all $(X, X^\bot) \in \mathcal{L}$, we have either $\dim X \leq 1$, or $X$ is not isotropic, and
    \item for all distinct pairs $(X, X^\bot), (Y, Y^\bot) \in \mathcal{L}$, we have $X \subset Y, \, Y \subset X$ or $X \subset Y^\bot$, and
    \item for all distinct pairs $(X, X^\bot), (Y, Y^\bot) \in \mathcal{L} \backslash \{ (\mathbf{0}, E_3) \}$, neither $X \subset Y \cap Y^\bot$, nor $Y \subset X \cap X^\bot$.
\end{itemize}
\end{definition*}

\begin{example}
\label{ex:introsubspaceexample}
    Let $X_1$ be the $1$-dimensional subspace of $E_3$ generated by $\oA\oF$, and let $X_2$ be the subspace generated by $\oA\oE\oF\oH$. Then $\mathcal{L} = \{ (\mathbf{0}, E_3), (X_1, X_1^\bot), (X_2, X_2^\bot) \}$ is compatible.

    For a non-example, let $X'_1$ be generated by $\oA\oB$ and $X'_2$ generated by $\oB\oC\oD\oE$, and observe that $\langle \oA\oB, \oB\oC\oD\oE \rangle = 1$. Then $\mathcal{L}' = \{ (\mathbf{0}, E_3), (X'_1, {X'_1}^\bot), (X'_2, {X'_2}^\bot) \}$ is  not compatible, as $X_1'$ and $X_2'$ do not meet the third bullet point.
\end{example}

\begin{proposition*}[see Prop.~\ref{prop:picturesubspacecorrespondence}]
    There are $43$ $\textup{Sp}(6,2)$-orbits of compatible collections of subspaces of $E_3$.
\end{proposition*}

These $43$ orbits count precisely the $42$ stable reduction types, not taking hyperelliptic reduction into account, along with the exceptional orbit of Rem.~\ref{rmk:intropictuestabletypebijection}. One can extend Def.~\ref{def:admissible} to include \textit{hyperelliptic}-compatible subspaces collections (Def.~\ref{def:hyperellipticcompatiblesubspaces}), and their $\textup{Sp}(6,2)$-orbits correspond exactly to the hyperelliptic versions of those stable reduction types that can be both non-hyperelliptic and hyperelliptic (see Prop.~\ref{prop:picturesubspacecorrespondenceHE}) The compatibility conditions amongst subspaces is how Tab.~\ref{tab:twoblockssubspaces} is derived (see Thm.~\ref{thm:introsubspaceoctadpictureagreement}).

The correspondence between a compatible collection of subspaces and stable reduction types can be made explicit with a universal recipe. We now explain this process in detail.

\begin{definition*}[see Def.~\ref{def:inclusiongraph}]
    Let $\mathcal{L}$ be a compatible collection of subspaces. The \textit{inclusion graph} of $\mathcal{L}$ is the following directed graph:
\begin{itemize}
    \item For each $(X, X^\bot) \in \mathcal{L}$, a vertex labelled with the subspace $X$, with the exception of the pair $(\mathbf{0},E_3)$ for which we use the subspace $E_3$.
    \item An edge $X \to Y$ when $Y$ is the smallest subspace strictly containing $X$.
\end{itemize}
Though edges are directed, by convention we do not include the directions when drawing the inclusion graph, rather they are understood to be directed upwards.
\end{definition*}

\begin{example}
    Let $\mathcal{L}$ be as in~\ref{ex:introsubspaceexample}. The inclusion graph for $\mathcal{L}$ is given in Fig.~\ref{fig:introinclusiongraphexample}.
    \begin{figure}[htbp]
        \tikzsetnextfilename{InclusionGraphExample}
        \begin{tikzpicture}[scale=1]
        \node at (0,0.2) {};
        \node[shape=circle,draw=black,scale=0.5] (6) at (0,0) {$E_3$};
        \node[shape=circle,draw=black,scale=0.5] (1A) at (-0.3,-1) {$X_1$};
        \node[shape=circle,draw=black,scale=0.5] (1B) at (0.3,-1) {$X_2$};
        \draw (6) -- (1A);
        \draw (6) -- (1B);
      \end{tikzpicture}
      \caption{Inclusion graph for $\mathcal{L}$}
      \label{fig:introinclusiongraphexample}
    \end{figure}
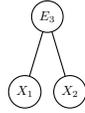
\end{example}

Each vertex of the inclusion graph is assigned a \textit{relation index} (see Def.~\ref{def:orthogonallydependent}), generalising the notion of \"ubereven clusters in a cluster picture. Typically, this relation index is $0$ for all vertices, and the exceptions can be found in Tab.~\ref{tab:smoctadsc02} -~\ref{tab:smhoctadsc56}; these are precisely the \textit{subspace graphs} with blue boxes (see Def.~\ref{def:subspacegraph}).

\begin{definition*}[see Def.~\ref{def:stablegraph},~\ref{def:stable-graph}]
    For the purposes of this introductory summary, suppose that the vertices of the inclusion graph of $\mathcal{L}$ all have relation index $0$. Then the \textit{stable graph}, $\mathcal{G}$, of $\mathcal{L}$ is the following cover of the inclusion graph, $\mathcal{I}$, of $\mathcal{L}$:
    \begin{itemize}
        \item Vertices: The vertices are the vertices of $\mathcal{I}$.
        \item Edges: Each edge $X \to Y$ is covered by one edge if $\dim(X)$ is even, and two edges if $\dim(X)$ is odd.
        \item Labels: The vertex of $\mathcal{I}$ labelled $X$ is replaced with an integer $\gamma(X)$ so that
        $$
        \tfrac{1}{2} \dim(X/X \cap X^\bot) = \sum_{Z \leqslant X} \gamma(Z) + \# \textup{loops occuring below X,}
        $$
        where $Z \leqslant X$ ranges over all edges below and including $X$ in the inclusion graph.
    \end{itemize}
    Then one contracts all vertices which are labelled $0$ and have exactly two edges, to a single edge.
\end{definition*}

\begin{example}
\label{ex:introstablegraph}
    Let $\mathcal{L}$ be as in Ex.~\ref{ex:introsubspaceexample}. As $\dim(X_1) = \dim(X_2) = 1$, both edges of the inclusion graph become two edges in the stable graph. There are no vertices below either $X_1$ or $X_2$, and so the conditions on the labelling reduce to $\gamma(X_1) = 0$ and $\gamma(X_2)= 0$. Both these vertices have exactly two edges, and are thus contracted. The label of the vertex $E_3$ satisfies $3 = \gamma(E_3) + 2$, and so is labelled $1$. The stable graph for $\mathcal{L}$ is thus given in Fig.~\ref{fig:introstablegraphexample}.
    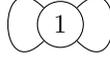
\begin{figure}[htbp]
    \tikzsetnextfilename{CandyGraph}
  \begin{tikzpicture}[scale=0.6]
        \node[shape=circle,draw=black,scale=1] (A) at (0,0) {1};
        \path  (A) edge [loop,min distance = 1.5cm,in=225,out=135] node {} (A);
        \path  (A) edge [loop,min distance = 1.5cm,in=45,out=-45] node {} (A);
      \end{tikzpicture}
      \caption{Stable graph of $\mathcal{L}$}
      \label{fig:introstablegraphexample}
      \end{figure}
\end{example}

\begin{notation}
Octad pictures are called hyperelliptic when they contain a $\mathrm{{\bf TCu}}$-, $\mathrm{{\bf Line}}$-,  or $\bm{\phi}$-block.
We write $\mathcal{P}^\textup{NH}$ for the set of non-hyperelliptic octad pictures which do not belong in the exceptional orbit, and $\mathcal{P}^{\textup{HE}}$ for the hyperelliptic octad pictures not in the exceptional orbit (see Rem.~\ref{rmk:intropictuestabletypebijection}).
\end{notation}

Note that the exceptional orbit of octad pictures themselves induce an exceptional orbit of compatible subspaces via the correspondence of Tab.~\ref{tab:introblocksdefinition}; we again consider this orbit of collections exceptional.

\begin{theorem}
\label{thm:introsubspaceoctadpictureagreement}
    Each octad picture $P$ induces a compatible collection of subspaces, which we denote $\varsigma(P)$ (see Def.~\ref{def:picturesubspacecorrespondence}). Moreover:
    \begin{itemize}
        \item $\varsigma$ gives a one-to-one correspondence
        $$
        \textup{Octad Pictures} \longleftrightarrow \textup{Compatible Sets of Subspaces}
        $$
        \item If $P$ is an octad picture not in the exceptional orbit, then the dual graph of the stable reduction type corresponding to $P$ as given in Tab.~\ref{tab:smoctadsc02}--\ref{tab:smhoctadsc56} is the stable graph of $\varsigma(P)$, where the genus of each component is given by the label of the corresponding vertex. That is, the following diagram commutes.
\begin{displaymath}
\xymatrix{
&\textup{$\mathcal{P}^\textup{NH} \cup \mathcal{P}^\textup{HE}$}
\ar@{<->}[rr]^-{\varsigma}\ar[rrd]_(.4){\textup{Tab.~\ref{tab:smoctadsc02}--\ref{tab:smhoctadsc56}}\ \hspace*{1.5cm}} & & \textup{Compatible Sets of Subspaces} \backslash \textup{Exceptional Orbit} \ar[d]^(.5){\ \textup{Stable Graph}}  \\
& & &\textup{Stable Types}
}
\end{displaymath}
    \end{itemize}
\end{theorem}

\begin{corollary*}[see Conj.~\ref{conj:SpecialFibreOfTheStableModelDetailed}]
    Assume Conj.~\ref{conj:SpecialFibreOfTheStableModel}. Let $C/K$ be a plane quartic over a non-archimedean local field with residue characteristic $p \neq 2$, and let $O$ be a Cayley octad of $C$. Then the stable graph of $\varsigma(d_K(O))$ is the dual graph of the stable reduction type of $C/K$.
\end{corollary*}

\begin{example}
Let $\mathcal{L}$ be as in Ex.~\ref{ex:introsubspaceexample}, and let $C/\Q_5$ and $O$ be as in Ex.~\ref{ex:introexample}. The octad picture $P \in \varsigma^{-1}(\mathcal{L})$ is given in Fig.~\ref{fig:intropicturesubspacecorrespondenceexample}.
    \begin{figure}[htbp]
        \centering
         \tikzsetnextfilename{IntroOctPicD}
         \begin{tikzpicture}[scale=0.25]
             \ptscustomlabel{F}{E}{H}{B}{C}{D}{G}{A}
             \twinbig{7}
             \plane{7}
        \end{tikzpicture}
        \caption{Octad picture, $P$, with $\varsigma(P) = \mathcal{L}$}
        \label{fig:intropicturesubspacecorrespondenceexample}
    \end{figure}
    Observe that $P = d_{\Q_5}(O)$. By Conj.~\ref{conj:SpecialFibreOfTheStableModelDetailed} and Ex.~\ref{ex:introstablegraph}, we predict that the dual graph of the stable reduction type of $C/\Q_5$ is given by the stable graph of $\mathcal{L}$, as in Fig.~\ref{fig:introstablegraphexample}, and so its stable reduction type is a genus 1 curve with two self-intersections.
\end{example}

Per Thm.~\ref{thm:introsubspaceoctadpictureagreement} and the correspondence given in Tab.~\ref{tab:smoctadsc02}--\ref{tab:smhoctadsc56}, the map from $\mathcal{P}^\textup{NH} \cup \mathcal{P}^{\textup{HE}}$ to the set of stable types in genus $3$, given by sending an octad picture $P$ to the stable graph of $\varsigma(P)$, factors through $\textup{Sp}(6,2)$; in fact $(\mathcal{P}^\textup{NH} \cup \mathcal{P}^{\textup{HE}}) / \textup{Sp}(6,2)$ is in bijection with the set of stable types.

For simplicity, consider the 42 stable types without considering whether the reduction is hyperelliptic or not. The bijection above admits even further structure when one considers that these 42 stable types in genus $3$ can be generated by a process of degenerating a smooth genus $3$ curve (see Sec.~\ref{sec:quart-stable-reduct}). The result is a directed graph, which we denote $\textup{SM}_3$ (see Fig.~\ref{fig:stablemodelgraphquartic}).

The set of octad pictures can also be generated by degenerating a trivial octad picture (equally, degenerating the trivial collection of compatible subspaces by either adding a pair, or by swapping $(X,X^\bot)$ for $(Y,Y^\bot)$ with $X = Y \cap Y^\bot$, see Sec.~\ref{subsec:structureoctadpictures}), giving a directed graph $\mathcal{G}_3$ whose vertex set is the set of octad pictures.

\begin{theorem*}[see Thm.~\ref{thm:OctadDiagramGraphIsomorphisms}]
The map sending an octad picture $P$ to the stable graph of $\varsigma(P)$ induces an isomorphism of directed graphs
$$
\mathcal{G}_3\,/\, \textup{Sp}(6,2) \to \textup{SM}_3.
$$
\end{theorem*}

The above theorem holds \textit{mutatis mutandis} with
hyperelliptic octad pictures and hyperelliptic stable reduction types (again, see Thm.~\ref{thm:OctadDiagramGraphIsomorphisms}).

\subsection{Evidence}
\label{subsec:evidence}

There is significant evidence in support of Conj.~\ref{conj:SpecialFibreOfTheStableModelDetailed}.
First, we emphasise that we have implemented and tested these reduction
criteria on a large number of integer quartics, in particular amongst all reduction types (based
on the explicit families constructed in Sec.~\ref{sec:expl-constr}) and for about 170,000 of the pairs $(C,p)$ where $C$ is a plane quartic and $p$ an odd prime of bad reduction, out of the 186,502 such pairs that occur in Andrew Sutherland's database of plane quartics \cite{Database3}. For the missing pairs, we did not have a way to independently verify the stable reduction type.
Assuming our conjectures we are now able, in a few seconds on a standard laptop, to determine the stable reduction type of a quartic, and in particular if it has hyperelliptic reduction or not (see~\cite{BGit23}).

The point on which the conjecture seems the most fragile is the residue
characteristic. We have excluded $2$ for several reasons, among which, the
existence of octad pictures of larger size than expected. As we did not find counterexamples in our experiments in
characteristic 3, 5, \textit{etc.}, we have settled for $p \neq 2$.

Beyond that, we have good theoretical reasons to believe in the validity of our conjecture. If $C$ is a plane quartic for which any of its 36 octad pictures are trivial (that is, coming from valuation data consisting of no blocks), then the conjecture holds for $C$, and its reduction type is a smooth (non-hyperelliptic) genus $3$ curve, see Thm.~\ref{thm:regularoctad}.
In fact we have obtained more partial results for our conjectures. Given the multiplicity
of situations however (42 stable types, of which only 32 can occur when the
reduction is hyperelliptic), and our desire to keep this article to a
reasonable length, we prefer to defer this question to future
work~\cite{BDDLL23}. Here we give a Thomae-like formula for the theta constants of a plane quartic in terms of the Pl\"ucker coordinates of the Cayley octad: this might allow via the theory of spin theta constants to prove our conjecture in the remaining cases.

Finally, we stress the compatibility of all known examples with the subspace framework of Sec.~\ref{subsec:introcombinatorialanalysis}. In particular, the complete combinatorial agreement, which includes the exact matching of the action of the symplectic group $\textup{Sp}(6,2)$ and the multiplicities of every octad picture, gives us faith in our conjectures.

\subsection{Structure of the Paper}

In Sec.~\ref{sec:quart-stable-reduct}, we introduce the different stable types that can occur for genus 3 curves (in both the hyperelliptic and non-hyperelliptic cases).

In Sec.~\ref{sec:CayleyOctadsSmoothCurves}, we recall the definition of a Cayley octad associated to a plane quartic, and we define the notion of valuation data of such octads.

In Sec.~\ref{sec:ocdiagrams}, we define building blocks for the different types of degenerations of Cayley octads that can occur, \textit{e.g.} two points colliding, or two complementary sets of four points on a plane. Prop.~\ref{thm:equiv-objects},~\ref{PGLTypes},~\ref{prop:PGLphiblocks} and~\ref{prop:TCuequiv} give the invariance of octad pictures consisting of one block under $\PGL$. Prop.~\ref{CreTransTwins},~\ref{CreTransTypes},~\ref{prop:CreTransCans} and~\ref{prop:CreTransHE} show that the action of $\textup{Sp}(6,2)$ commutes with taking octad pictures in the case of a single block.

In Sec.~\ref{sec:CombiningOctads}, we explain how these degenerations of the Cayley octad lead to octad pictures.

In Sec.~\ref{sec:formalanalysisoctadpictures}, we give a precise conjecture, Conj.~\ref{conj:SpecialFibreOfTheStableModelDetailed}, that predicts how the stable reduction type of a curve can be obtained from the octad picture via the framework of associated subspaces.
We also give a detailed worked example to illustrate how our constructions can be used in practice.

Sec.~\ref{sec:exper-supp-conj} contains examples which have been used to conceive and verify our conjecture. An application of our conjectures gives predictions not known in the literature.

App.~\ref{sec:special-fibres-octad-1} and App.~\ref{sec:special-fibres-octad-2} contain tables giving the full correspondence between octad pictures and stable reduction types.

\subsection{General Notation}

Throughout this paper, we make use of the following notation and terminology:
\begin{itemize}
\item $K$, $R$, $\pi$: a local field, its ring of integers and a uniformising element;
\item $C$: a curve over $K$;
\item $M_g$: the moduli space of smooth curves of genus $g$;
\item $\overline{M_g}$: the Deligne-Mumford compactification of $M_g$;
\item $O = ( O_\oA, O_\oB, \ldots, O_\oH )$: an 8-tuple of points on $\P^3$, also known as an octad;
\item $p_{ijkl}$: the Pl\"ucker coordinate $\det(O_i | O_j | O_k | O_l)$;
\item $v_{\pt}, v_{\ln}, v_\pl, v_{\tc}$: elementary valuation data corresponding to a set of points colliding, lying on a line, plane, or twisted cubic, respectively, see Def.~\ref{def:standard-valuation-data};
\item ${\bm{\alpha}}$-block: an octad building block responsible for a node in the special fibre of the stable model, other than those nodes explained by ${\bm \chi}$- and ${\bm \phi}$-blocks;
\item ${\bm{\chi}}$-block: an octad building block indicating the existence of partition of the special fibre of the stable model into two connected parts of arithmetic genera 1 and 2, intersecting each other in 1 node;
\item ${\bm{\phi}}$-block: an octad building block indicating the existence of a partition of the special fibre of the stable model into two connected parts of arithmetic genus 1, intersecting each other in 2 nodes;
\item $d_K(O)$: an octad picture as defined in Sec.~\ref{sec:CombiningOctads};
\item $E_g$: the $2g$-dimensional symplectic space as defined in Sec.~\ref{sec:sympl-f_2-vect};
\item $X^{\perp}$: the orthogonal complement of the subspace $X$ with respect to the symplectic pairing;
\item $\mathcal{P}^\textup{NH}$, $\mathcal{P}^\textup{HE}$ and $\mathcal{P}^\phi$: the set of non-hyperelliptic octad pictures not in the exceptional orbit, hyperelliptic octad pictures, and octad pictures containing a ${\bm{\phi}}$-block, respectively;
\item $\mathcal{G}_3$ (resp. $\mathcal{G}_3^{\mathrm{HE}}$): the directed graph of non-hyperelliptic (resp. hyperelliptic) octad pictures with $\mathcal{P}^\textup{NH}$ (resp. $\mathcal{P}^\textup{HE}$) as vertex set, and directed arrows for the degenerations of pictures;
\item $\Delta$, $\Delta_v$: polytopes as in \cite{Dokchitser21};
\item $\mathrm{SM}_3$: the graph of stable types of genus 3, see Section
\ref{sec:quart-stable-reduct}.
\end{itemize}\bigskip

\addtocontents{toc}{\SkipTocEntry}
\subsection*{Acknowledgements}

The first author was supported by the Simons Foundation grant 550033.
The second author undertook this research whilst supported by the Engineering and Physical Sciences Research Council [EP/L015234/1], the EPSRC Centre for Doctoral Training in Geometry and Number Theory (The London School of Geometry and Number Theory) at University College London.
The third author was supported by a Royal Society University Research Fellowship.
The fourth author benefits from the France 2030 framework program ``Centre
Henri Lebesgue'', ANR-11-LABX-0020-01.
The research of the fifth author is partially funded by the Melodia
ANR-20-CE40-0013 project and the 2023-08 Germaine de Sta\"el project.

For the purpose of open access, the authors have applied a Creative Commons Attribution (CC BY) licence to any Author Accepted Manuscript version arising.

\section{Background: Quartic Stable Curves}
\label{sec:quart-stable-reduct}

A {\em stable curve} is a fibrewise reduced and connected curve $C$ over some base scheme $S$, such that $C$ is of arithmetic genus at least 2, its geometric fibres have
only nodal singularities (ordinary double points), and the irreducible components of genus 0 of these fibres have at least 3 singular points, counted with multiplicity.

A theorem of Deligne–Mumford~\cite[Corollary 2.7]{DM69} states that given any
curve $X$ over $K$ of genus at least 2, there exists a finite algebraic extension $L$ of
$K$ and a stable curve ${C}$ over the integral closure $R_L$ of $R$
in $L$ with generic fibre ${C}_{\eta} \simeq X\times_K L$.
The \textit{stable reduction} of a curve is the special fibre of such a stable
model, and the \textit{(stable) reduction type} is the stable type of the stable reduction, including whether the reduction is hyperelliptic or not, \textit{i.e.}\ whether it admits a degree two map to a semi-stable curve of arithmetic genus 0.\medskip

We now consider the so-called {\em stable type} that a stable curve $C$ %
can have.
For this purpose, we study the Deligne-Mumford compactification $\overline{M_g}$ over $\mathrm{Spec}(\Z)$ of the moduli stack
$M_g$  of smooth curves of genus $g$, which is a geometric space of dimension $3\,g-3$, classifying stable curves of arithmetic genus $g$.
There is a natural stratification on $\overline{M_g}$ given by the stable
type, \textit{i.e.}\ the intersection multigraph of irreducible components together
with their geometric genera. Strata of codimension $k$ consist of stable curves
having exactly $k$ nodes. In particular, the only stratum of codimension 0 is the
open set $M_g \subset \overline{M_g}$ of smooth curves.\smallskip

It is possible to obtain all stable types recursively by degeneration,
using the following two topological transformations: given an irreducible
component $\mathcal{X}_0$ of genus $g_0$,
\begin{itemize}
\item if $g_0 \geqslant 1$, add a node to $\mathcal{X}_0$ and let its
  geometric genus drop by 1;
\item or given two non-negative integers $g_1$ and $g_2$ such that
  $g_1 + g_2 = g_0$, replace $\mathcal{X}_0$ by two irreducible components
  $\mathcal{X}_1$ and $\mathcal{X}_2$, of genus $g_1$ and $g_2$, which intersect
  at a node, and redistribute the nodes of $\mathcal{X}_0$ between
  $\mathcal{X}_1$ and $\mathcal{X}_2$ (see Fig.~\ref{fig:degeneration}).
\end{itemize}
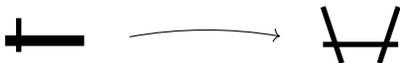
\begin{figure}[htbp]
  \centering
  \tikzsetnextfilename{degeneration}
  \begin{tikzpicture}[scale=1.4,g3lattice]
    \node at (0,0) (2e) {\DeU};
    \node at (3,-0) (1ee) {\UeUeU};
    \draw[->,shorten <=0.5cm,shorten >=0.5cm,line width=0.4pt]
    (2e.-10) to[bend left=10 ] (1ee.190);
  \end{tikzpicture}
  \caption{Degeneration of a genus 2 component into two genus 1 components}
  \label{fig:degeneration}
\end{figure}

So, for a fixed genus $g \geqslant 2$, there are only finitely many
possibilities. They have been enumerated for $g \leqslant 8$,
see~\cite{maggiolo11,oeisA174224}.
\smallskip

\begin{theorem}[{\cite[Prop.~4.1]{maggiolo11} or \cite[Th.~2.2.12]{Chan2012}}]\label{thm:M3nonH}
  There are 42 stable types in genus 3 that make up the graph $\textup{SM}_3$
  defined on Fig.~\ref{fig:stablemodelgraphquartic}. They are ordered by
  codimension from top to bottom according to the stratification of
  $\overline{M_3}$ they induce. There is an edge between two types if one is
  obtained from the other by one degeneration.
\end{theorem}

In addition to small descriptive drawings that
represent irreducible components and their intersections, it is
convenient to use a compact naming convention where
\begin{itemize}
\item \texttt{0}, \texttt{1}, \texttt{2} or \texttt{3} refer to components of
  that genus,
\item \texttt{n} means that the component has a self-intersection point (``n''ode),
\item \texttt{e} means a genus 1 component (``e''lliptic curve) intersecting the preceding component,
\item \texttt{m} means a genus 0 component with a self-intersection  (``m''ultiplicative
  reduction of elliptic curves) intersecting the preceding component,
\item \texttt{X=Y} means that \texttt{X} and \texttt{Y} intersect at two
  points, \texttt{Z} is shorthand for \texttt{0=0},
\item \texttt{X\text{-}\text{-}\text{-}Y}, or
  \texttt{X\text{-}\text{-}\text{-}\text{-}Y}, means that \texttt{X} and
  \texttt{Y} intersect at 3, or 4 points.
\end{itemize}
For instance, \texttt{(0nnn)} is a genus 0 curve with three self-intersections and
\texttt{(0eee)} looks like the letter $\sha$ where the bottom horizontal component is of genus 0
and the other ones are of genus 1.
The above conventions cover 40 of the 42 types unambiguously. The remaining
two, and borrowing the notation from \cite{BCKLS20}, are \texttt{(CAVE)} (3 lines, one of which crosses each of the other two
twice, the other two crossing each other once) and \texttt{(BRAID)} (4 lines
that intersect two by two).

Note that some of the stable types are automatically hyperelliptic. Indeed, a curve of type \texttt{(1=1)} is always hyperelliptic. Both genus 1 components have an involution that swap the two intersection points, and the quotient under this involution is a stable curve of arithmetic genus 0. This also applies to any type that is degenerated from \texttt{(1=1)} (and no other types).
\medskip

The corresponding graph for genus 3 hyperelliptic curves is in fact not a subgraph of the graph in
Fig.~\ref{fig:stablemodelgraphquartic}. To illustrate this, see for example that the degeneration \texttt{(1=1)} in the hyperelliptic case comes directly from the \texttt{(3)} without passing through \texttt{(2n)} as in the non-hyperelliptic case. This is explained in \cite[sect.~4b]{CornalbaHarris}. The degenerations that occur within
the hyperelliptic locus are best understood on the level of cluster pictures~\cite{m2d2,otherm2d2};
there is an edge $v_1 \to v_2$ precisely when the cluster picture for $v_2$ is
the result of adding one cluster to the cluster picture for $v_1$.

\begin{theorem}{\cite[Table~9.1]{m2d2}}\label{thm:M3H}
  There are 32 hyperelliptic stable types in genus 3 that make up the graph
  defined on Fig.~\ref{fig:stablemodelgraphhyperelliptic}. They are ordered
  by codimension from top to bottom according to the stratification of the
  moduli space $\overline{H_3} \subset \overline{M_3}$ of stable hyperelliptic
  curves they induce. There is an edge between two types if one is obtained
  from the other by one degeneration.
\end{theorem}

\begin{remark}
The edges in Fig.~\ref{fig:stablemodelgraphhyperelliptic} can be verified through the
correspondence with hyperelliptic reduction types given in Tab.~\ref{tab:Cluster2Type}.
\end{remark}

\def\fs{\kern 0.2em}
\def\bpt{\fs$\cdot$}

\begin{table}[htbp]
  \renewcommand{\arraystretch}{1.0}
  \setlength{\tabcolsep}{0.5mm}
  \centering

  \caption{Cluster picture and stable reduction type
    correspondence.} %
  \label{tab:Cluster2Type}
\end{table}

\begin{figure}[htbp]
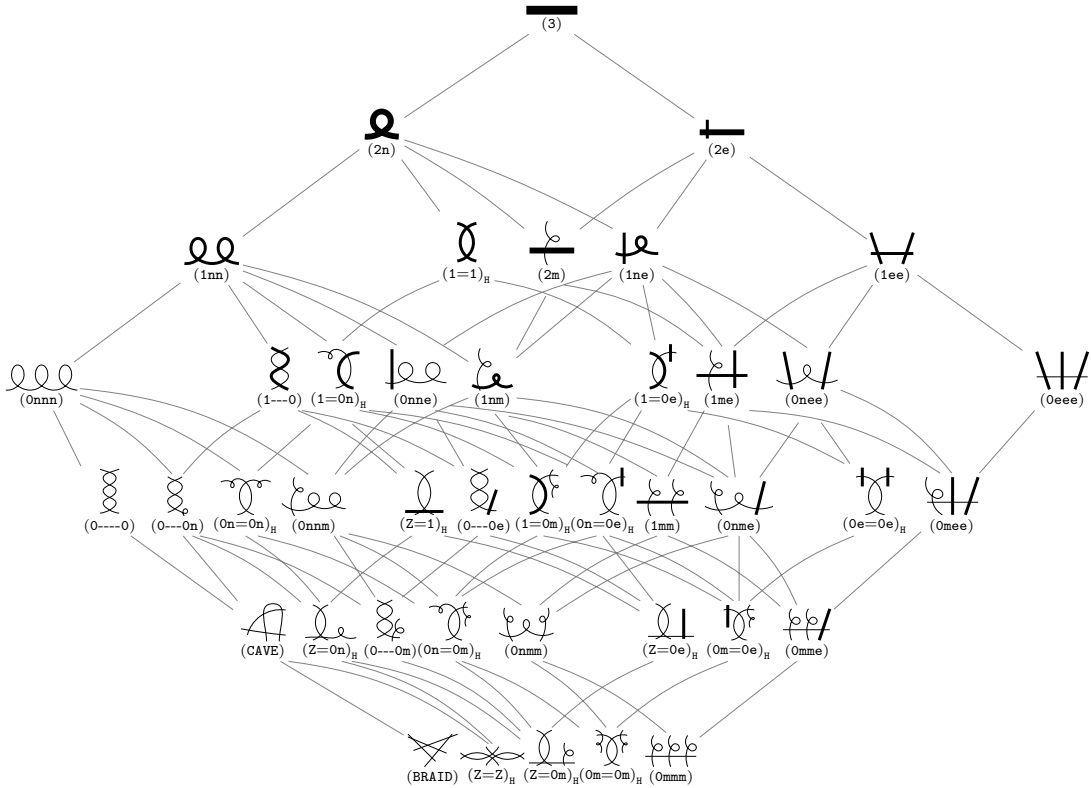
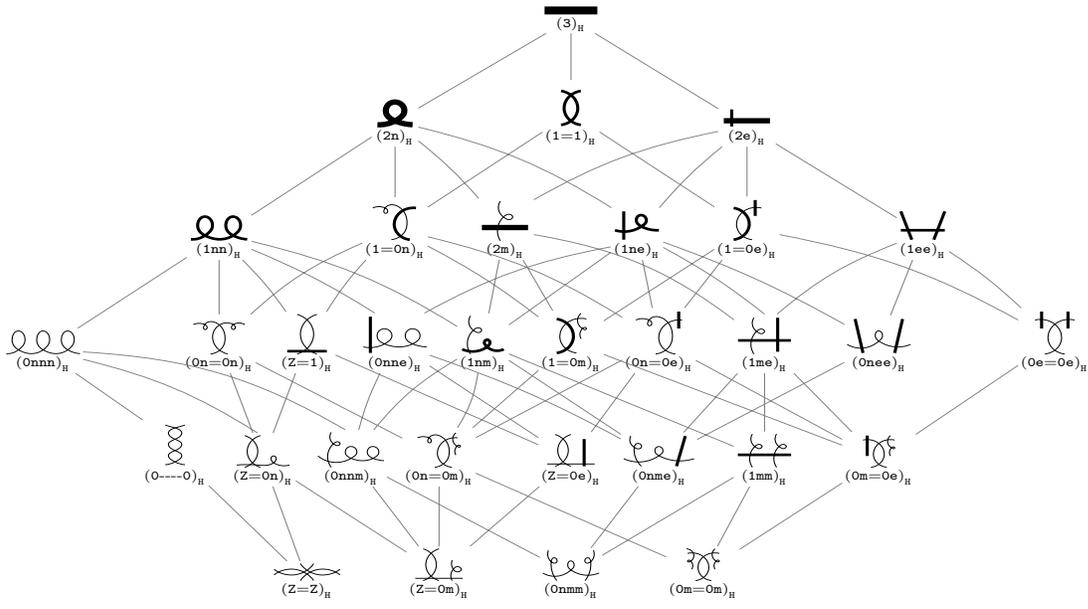


  \begin{subfigure}[b]{1.0\linewidth}

    \resizebox{1.0\linewidth}{0.5\textheight}{%
      \tikzsetnextfilename{stablemodelgraph}
      %
    }
    \caption{Non-hyperelliptic (42 stable models)}
    \label{fig:stablemodelgraphquartic}
  \end{subfigure}

  \vspace{\floatsep}

  \begin{subfigure}[b]{1.0\linewidth}
    \resizebox{1.0\linewidth}{0.38\textheight}{%
      \tikzsetnextfilename{stablemodelgraphhyperelliptic}
      %
    }
    \caption{Hyperelliptic (32 stable models)}
    \label{fig:stablemodelgraphhyperelliptic}
  \end{subfigure}

  \caption{
    \protect\begin{tabular}[t]{l}
      Stable model types of genus 3 curves\\[-0.2cm]
      {\footnotesize (the genus of a component is indicated by its thickness)}
    \protect\end{tabular}
    \label{fig:stablemodelgraph}
  }
\end{figure}

\section{Cayley Octads and Valuation Data}
\label{sec:CayleyOctadsSmoothCurves}

In this section, we introduce the objects we need next to characterise the
stable reduction type of quartics: Cayley octads and their valuations. Cayley octads themselves have been studied extensively, and we refer to \cite{DolgachevOrtland} and \cite{psv11} for a more complete treatment.

\subsection{Cayley Octads}
\label{sec:cayley-octads}

Let $\mathbb{K}$ be a field of characteristic $\ne$ 2, supposed to be
algebraically closed for the sake of simplicity.  Curves of genus $3$ defined over
${\mathbb{K}}$ have $64$ theta characteristics, $28$ of these are odd, and
$36$ are even.

Geometrically, the $28$ odd theta characteristics are realised as the
bitangents to a plane quartic, the canonical model of a non-hyperelliptic
curves of genus $3$ \cite{MR2352717}. The $36$ even theta characteristics give $36$ different
ways in which a plane quartic can be written as
\begin{displaymath}
  \det(A x + B y + C z) = 0,
\end{displaymath}
where $A, B, C$ are symmetric $4 \times 4$ matrices, and two such representations are equivalent (that is, come from the same theta characteristic) if they are in the same orbit under simultaneous $\textup{GL}_4$-conjugation, see~\cite[Thm. 3.1]{psv11}. %

Given such a representation, there are three quadrics in $\P^3_{\mathbb{K}}$ given by the zero sets
\begin{displaymath}
  \boldsymbol{u}^{T} A \boldsymbol{u} = 0, \quad %
  \boldsymbol{u}^{T} B \boldsymbol{u} = 0, \quad %
  \boldsymbol{u}^{T} C \boldsymbol{u} = 0.
\end{displaymath}
These three quadrics form the basis of a net, whose base locus consists
(typically) of $8$ points. These $8$ points are a geometric realisation of the
even theta characteristics of a non-hyperelliptic genus $3$
curve. $\textup{GL}_4$-conjugation of a determinantal representation
corresponds to a $\textup{PGL}_3({\mathbb{K}})$  transformation of the resulting 8
points. Any $2$ of these $8$ points determine a bitangent of the original plane quartic. For, the theta-characteristic giving rise to the determinantal representation embeds the quartic in $\P^3_{\mathbb{K}}$, and the chord joining the two points of the Cayley octad intersects this embedded curve twice. The two corresponding points of the quartic (in $\P^2_{\mathbb{K}}$) are the points of intersection of a bitangent.
It turns out that these 8 points form a regular octad in the following sense.

\begin{definition}
  An ordered set $O = ( O_\oA, O_\oB, \ldots, O_\oH )$ of $8$ points in $\P^3_{\mathbb{K}}$ is called an \textit{octad}. An octad $O$ is said to be \textit{regular} if
  \begin{itemize}
  \item no two points of $O$ coincide,
  \item no four points of $O$ are coplanar, and
  \item the 8 points of $O$ do not all lie on a twisted cubic, \textit{i.e.}\ the image of a degree 3 map of $\P^1_{\mathbb{K}}$ into $\P^3_{\mathbb{K}}$.
  \end{itemize}
  An octad is called a \textit{Cayley octad} if its points are the common
  intersection of three quadrics in $\mathbb{P}^3$.
  The Cayley octad of a plane quartic curve with a given even theta characteristic is the 8 intersects points of the quadrics $\boldsymbol{u}^{T} A \boldsymbol{u}$, $\boldsymbol{u}^{T} B \boldsymbol{u}$, and $\boldsymbol{u}^{T} C \boldsymbol{u}$, as defined above, in some order. Note this only defines the Cayley octad up to $\PGL$, and permutation of the $8$ points.
\end{definition}

\begin{definition}
Two octads are said to be \textit{equivalent} if they coincide under the action of $\textup{PGL}_3 \times S_8$.
\end{definition}

Given 4 points $O_i,O_j,O_k,O_l$ in $\mathbb{A}^4_{\mathbb{K}}$, we define the Pl\"ucker coordinate: $p_{ijkl}=\operatorname{det}(O_i\mid O_j\mid O_k\mid O_l)$. By making a choice of affine representatives of points in an octad we extend this definition to quadruples of points in an octad.

\begin{theorem}\label{thm:Cayley-def}
  Let $\mathbb{K}$ be an algebraically closed field with $\mathop\mathrm{char}\mathbb{K} \ne 2$.
  An octad $O= ( O_\oA$, $O_\oB$, $\ldots\, O_\oH )$ defined over $\mathbb{K}$ is a
  Cayley octad if and only if any of the following equivalent conditions hold:
  \begin{itemize}
  \item[i)] ${O}$ is self-dual in the Gale sense, \textit{i.e.}\ ${O}$ and ${O}^*$ are $\textup{PGL}_3(\mathbb{K})$-equivalent where ${O}^*$ is any matrix whose row space equals the kernel of ${O}$ viewed as a $4\times 8$ matrix.
  \item[ii)] $O$ satisfies the 630 equations obtained by taking all permutation of the $8$ points in the equation
    \begin{equation}\label{eq:cayley}
      p_{\oA\oB\oC\oD}\,p_{\oA\oB\oE\oF}\,p_{\oC\oE\oG\oH}\,p_{\oD\oF\oG\oH} = p_{\oE\oF\oG\oH}\,p_{\oC\oD\oG\oH}\,p_{\oA\oB\oD\oF}\,p_{\oA\oB\oC\oE}\,.
    \end{equation}
  \end{itemize}
\end{theorem}
\begin{proof}
see~\cite[Thm. 3.1]{Coble} or~\cite{ep00} for the first characterisation and~\cite[Prop.~7.2]{psv11} for the second one.
\end{proof}

\begin{remark}\label{rmk:7determine8} Any $7$ points, with no four of them coplanar, lie on a unique regular Cayley octad, \textit{i.e.}\ they determine the eighth point. see~\cite[Prop.~7.1]{psv11} for an explicit description.
\end{remark}

\subsubsection*{Action of the Symplectic Group}

Every plane quartic has 36 inequivalent Cayley octads (up to $\textup{PGL}_3\times S_8$), one for each even theta characteristic.

\begin{theorem}[Thm. IX.2.2 in~\cite{DolgachevOrtland}]
\label{thm:regularoctad}
    Let $\mathbb{K}$ be an algebraically closed field with $\mathop\mathrm{char}\mathbb{K} \ne 2$.
    There is a one-to-one correspondence between equivalence classes of regular Cayley octads, and smooth plane quartics with full level 2 structure, \textit{i.e.}\ plane quartic curves together with an isomorphism $\mathrm{Jac}(C)({{\mathbb{K}}})[2] \to E_3$, up to isomorphism.
\end{theorem}

There is a natural action of the group $\textup{Sp}(6,2)$ on plane quartics with full level 2 structure.  In particular, there is an action of $\textup{Sp}(6,2)$ on the set of Cayley octads of $C$.

The stabiliser of any given Cayley octad as an unordered set is a subgroup of $\textup{Sp}(6,2)$ isomorphic to $S_8$, which acts by permuting the points (note, though, that if one considers a Cayley octad as an \textit{ordered} set then its stabiliser is trivial).

\begin{remark}
\label{rmk:cremonatransformations}

Recall the definition of $E_3$ given in Def.~\ref{def:introE3}, and that the group $\textup{Sp}(6,2)$ acts on $E_3$ as the group of linear transformation preserving the symplectic pairing.
Certainly, $S_8$ acts by permutation on the set $\{ \oA, \ldots, \oH \}$, and so by extension there is then an action of $S_8$ on $E_3$. Note that this action is well-defined and preserves the symplectic pairing, and thus fixes $S_8 \subset \textup{Sp}(6,2)$. Given that the action of $S_8$ on Cayley octads is plain (simply permuting the points), we elucidate the action of the cosets $\textup{Sp}(6,2)/S_8$ (though this is not a group action, as $S_8$ is not a normal subgroup of $\textup{Sp}(6,2)$).

Each such coset corresponds to a \textit{Cremona transformation}. Classically, Cremona transformations (of a given dimension $n$ and field $\mathbb{K}$) are the birational transformations of $\P^n_{\mathbb{K}}$. When $n=3$ they act on (equivalence classes of) Cayley octads, and their action can be described as follows.

A Cremona transformation is either the identity transformation or a transformation associated to a partition of an octad $O = ( O_\oA, O_\oB, \ldots, O_\oH )$ into two sets, each containing four points. Without loss
of generality, applying a permutation in $S_8$
if necessary, such a partition is indexed by $\oA\oB\oC\oD \mid \oE\oF\oG\oH$, and, after a suitable $\PGL$-transformation, the corresponding Cremona transformation in the basis $\{\oA,\oB,\oC,\oD\}$ send the coordinates
of the points of $O$ to
\begin{displaymath}
  \arraycolsep=0.4\arraycolsep%
  \mbox{\scriptsize$\begin{bmatrix}
      1 & 0 & 0 & 0 & e_1 & f_1 & g_1 & h_1 \\
      0 & 1 & 0 & 0 & e_2 & f_2 & g_2 & h_2 \\
      0 & 0 & 1 & 0 & e_3 & f_3 & g_3 & h_3 \\
      0 & 0 & 0 & 1 & e_4 & f_4 & g_4 & h_4
    \end{bmatrix}$}
  \ \longrightarrow\ %
  \mbox{\scriptsize$\begin{bmatrix}
      1 & 0 & 0 & 0 & e_1^{-1} & f_1^{-1} & g_1^{-1} & h_1^{-1} \\
      0 & 1 & 0 & 0 & e_2^{-1} & f_2^{-1} & g_2^{-1} & h_2^{-1} \\
      0 & 0 & 1 & 0 & e_3^{-1} & f_3^{-1} & g_3^{-1} & h_3^{-1} \\
      0 & 0 & 0 & 1 & e_4^{-1} & f_4^{-1} & g_4^{-1} & h_4^{-1}
    \end{bmatrix}$}\,.
\end{displaymath}
This is well-defined by Thm.~\ref{thm:regularoctad}.

As written above (with ordering coming from the labelling of points), this Cremona transformation is in fact an element of $\textup{Sp}(6,2)$. To aid the discussion we take the liberty of fixing for each Cremona transformation a representative element $\sigma_H \in H$, where $H$ is the corresponding coset. We achieve this with the following scheme. If $H$ is the coset containing the identity element of $\textup{Sp}(6,2)$ we fix $\sigma_H = \textup{id}$. Now suppose $H$ is the coset corresponding to the Cremona transformation $ijkl \mid pqrs$, where $\{ i,j, \ldots, s\} = \{ \oA, \ldots, \oH \}$. Fix the following symplectic basis of $E_3$,
\begin{displaymath}
  \{\,ij,\ pq,\ kr,\ il,\ ps,\ ijlr\,\}\,.
\end{displaymath}
Then the element $\sigma_H \in \textup{Sp}(6,2)$ that we attach to this coset is, with respect to this basis, given by the matrix
\begin{displaymath}
  \arraycolsep=0.4\arraycolsep%
  \renewcommand{\arraystretch}{0.9}%
  \mbox{\scriptsize$\begin{bmatrix}
    1 & 0 & 0 & 0 & 0 & 0 \\
    0 & 1 & 0 & 0 & 0 & 0 \\
    0 & 0 & 0 & 0 & 0 & 1 \\
    0 & 0 & 0 & 1 & 0 & 0 \\
    0 & 0 & 0 & 0 & 1 & 0 \\
    0 & 0 & 1 & 0 & 0 & 0 \\
  \end{bmatrix}$}\,.
\end{displaymath}
Every element $\tau \in \textup{Sp}(6,2)$ can be realised as a product $\tau = \sigma_H \circ \rho$, where $\rho$ is an $S_8$-permutation and $\sigma_H$ is one of the $36$ elements introduced above. In particular, as the action of $S_8$ upon Cayley octads is transparent, understanding the action of the full symplectic group $\textup{Sp}(6,2)$ on the Cayley octads reduces to considering the action given by these $36$ elements.

\end{remark}

\subsection{Valuation Data}
\label{sec:valuation-data}

Our main motivation being the stable reduction of quartics,
we now let $K$ be a local field and  $R$ its ring of integers. All our methods should still work for a discrete valuation field $K$, as long as one specifies an extension of the valuation to $\overline{K}$ in advance, but for simplicity we restrict to local fields.
Given Thm.~\ref{thm:regularoctad}, one might wish to keep track of the
extent to which a set of $8$ points in $\P^3$ fails to be regular after reduction. To this
end, we introduce the following data.

\begin{definition}\label{def:valuation-data}
Let $O = (O_\oA, \ldots, O_\oH)$ be an ordered $8$-tuple of points in $\P^3$ over $\overline{K}$ with valuation $\nu_K \colon \overline{K} \to \Q$. Fix lifts $(O_i^{(0)}, O_i^{(1)}, O_i^{(2)}, O_i^{(3)}) \in \mathbb{A}^4$ of $O_i$, for $i \in \{ \oA, \ldots, \oH\}$, normalised so that $\min_j(\nu_K(O_i^{(j)})) = 0$. Let $P = \{S \subset \{oA, \ldots, \oH : |S| = 4\}$ and introduce a formal symbol $\dagger$. The \textit{valuation data} associated to $O$ is a function $v \colon P \cup \{\dagger\} \to \Q \colon S \mapsto v_S$ as follows.
\begin{itemize}
    \item For each $S = \{i,j,k,l\} \in P$, the value $v_S$ is the valuation of the Pl\"ucker coordinate $p_{ijkl}$.
    \item The value $v_\dagger$ is a non-negative integer indicating with
      what multiplicity the points in $O$ are lying on a twisted cubic,
      \textit{i.e.}\ the minimum of the gaps between
      \begin{itemize}
      \item the valuations of the 840 polynomials obtained by taking all $S_8$
        permutations of the polynomials described in~\cite[Eq.~(7.4)]{psv11}
        evaluated in the points of $O$, and
      \item the valuations of the same polynomials evaluated at a generic
        octad for which the valuation data agrees on $P$.
      \end{itemize}
\end{itemize}
\end{definition}

\begin{example}
Let $K = \Q_{10007}$, and let $O$ be the ordered 8-tuple (points in columns),
\begin{displaymath}\arraycolsep=0.4\arraycolsep%
  \mbox{\scriptsize$\begin{bmatrix}
  1 &0 &0 &0 &1 &3 \cdot 10007 &3 \cdot 10007 & -2\cdot 10007\\
  0 &1 &0 &0 &-2\cdot 10007 &1 &5\cdot 10007 &10007\\
  0 &0 &1 &0 &-2\cdot 10007 &-2\cdot 10007 &1 &4\cdot 10007\\
  0 &0 &0 &1 &-5\cdot 10007 &3\cdot 10007 &5\cdot 10007 &1
\end{bmatrix}$}
\end{displaymath}
There are four pairs of points colliding modulo $10007$: the pair $O_{\oA}$ and $O_{\oE}$, the pair $O_{\oB}$ and $O_{\oF}$, the pair $O_{\oC}$ and $O_{\oG}$, and the pair $O_{\oD}$ and $O_{\oH}$.
Computing all the 70 Pl\"ucker coordinates and taking their valuations, we get 16 coordinates with valuation 0, 48 coordinates with valuation 1, and 6 coordinates with valuation 2.
This corresponds exactly to the vector $v_{\pt}^{\oA\oE} + v_{\pt}^{\oB\oF} + v_{\pt}^{\oC\oG} + v_{\pt}^{\oD\oH}$.
Out of the 840 relations for that determine whether the 8 points lie on a twisted cubic, 120 of them attain valuation 3, and 720 of them attain valuation 4.
Note that this does not mean that there is a $v_{\tc}$ component in the valuation data of $O$, as any (generic) octad for which the Pl\"ucker part of the valuation data is $v_{\pt}^{\oA\oE} + v_{\pt}^{\oB\oF} + v_{\pt}^{\oC\oG} + v_{\pt}^{\oD\oH}$, has 120 twisted cubic relations of valuation 3 and 720 such relations of valuation 4.
Therefore, the twisted cubic index of this octad is 0, and the valuation data of $O$ is $v_{\pt}^{\oA\oE} + v_{\pt}^{\oB\oF} + v_{\pt}^{\oC\oG} + v_{\pt}^{\oD\oH}$.
\end{example}

Recall that the valuation data $v_{\pt}^T, v_{\ln}^T, v_{\pl}^T, v_{\tc}$,  the sum, and the maximum of two valuation data have all been given in Def.~\ref{def:standard-valuation-data}. Moreover, we define a notion of equivalence on valuation data as follows.

\begin{definition}
Two valuation data $v$ and $v'$ are said to be $\PGL$-equivalent if for any ordered 8-tuple $O$ of points with valuation data $v$ or $v'$, there exists a $\PGL$-equivalent octad $O'$ with valuation data $v'$ or $v$, respectively.
\end{definition}

\begin{remark}\label{rem:valuation-lattice}
  When 8 points form a Cayley octad, the
  70 Pl\"ucker valuations in their valuation data must lie in the subspace defined by the equations in the second part of
  Thm.~\ref{thm:Cayley-def}, considered as an additive equation in the valuations.
  Note that if we drop the non-negativity condition on valuation vectors, we get
  in this way a lattice $\textup{L}$ inside $\Z^P$, of rank 42. There are
  \begin{itemize}
  \item 35 valuation vectors given by having two complementary planes vanishing at
    order 1, and these define a sub-lattice $\textup{L}_\pl$ of rank
    35,
  \item 28 valuation vectors given by having two points colliding at order 1, and
    these define a sub-lattice $\textup{L}_{\pt}$ of rank 28.
  \end{itemize}
  It turns out that $\textup{L}$ is equal to the sum of these two
  sub-lattices, \textit{i.e.}\ %
  $\textup{L}=\textup{L}_\pl + \textup{L}_{\pt}$.
\end{remark}

\smallskip

\subsection{Action of \texorpdfstring{PGL$_3$}{PGL} on Octads}

We study the action of $\PGL$ on octads. We start with an example.

\begin{example}
Let $K = \Q_{13}$, and let $O$ be the ordered octad (points are in columns),
\begin{displaymath}
  \arraycolsep=0.4\arraycolsep%
  \mbox{\scriptsize$\begin{bmatrix}
    1 & 0 & 0 & 0 & 1 &  1 &   1 & 1 \\
    0 & 1 & 0 & 0 & 1 &  9 &  78 & 41831/441 \\
    0 & 0 & 1 & 0 & 1 & 12 & 156 & 79532/891 \\
    0 & 0 & 0 & 1 & 1 & 11 &  13 & -7729/549
  \end{bmatrix}$}\,.
\end{displaymath}
Then the first and seventh points coincide modulo $13$, and there are no other
coincidences modulo $13$. Therefore, the valuation data of this octad is
$v_{\pt}^{\oA\oG}$. %
Now we see what happens with the valuation when we let $\PGL$ act on this octad.
We can force the first and seventh points to not collide anymore by applying
the $\PGL$ transformation given by
the $4 \times 4$ diagonal matrix $M = \diag(13,1,1,1)$\,.
The change would cause the other points to all lie on a plane with valuation 1, so the valuation data of this equivalent octad is $v_\pl^{\oB\oC\oD\oE\oF\oH}$.
We could also do the opposite and instead apply $M^{-1}$\,.
The valuation data for this (non-normalised) equivalent octad is
$v_{\pt}^{\oA\oE\oF\oG\oH} + v_{\pt}^{\oA\oG}$,
as the five points
$O_\oA, O_\oE, O_\oF, O_\oG$, and $O_\oH$ now coincide modulo 13 and the points
$O_\oA$ and $O_\oG$ even coincide modulo $13^2$.
\end{example}

The following technical  propositions is used in the next section in the analysis of the $\PGL$ equivalence classes of the different octad building blocks. In the following, $K$ and $R$ are as before, and $\pi$ is an uniformizer of $R$.

\begin{proposition}\label{prop:tci-invariant}
    Let $O, O'$ be two $\PGL$-equivalent Cayley octads. %
    Then %
    their twisted cubic indices %
    are equal.
\end{proposition}

\begin{proof}
  The twisted cubic index of $O$ is determined by the the 840 polynomials described
  in~\cite[Eq.~(7.4)]{psv11} evaluated at the
  Plücker coordinates of an octad as in Def. \ref{def:valuation-data}. Each of these evaluated polynomials consists of two terms,
  and the valuation of both terms can be immediately deduced from %
  the valuation data of the octad.
  If, for at least one of these 840 polynomials, both terms do not have equal valuation, then we can
  conclude that the twisted cubic index is 0. Otherwise, these evaluated polynomials can be written as
  $\{\,\pi^{s_i}\,T_i\,\}_i$ where $T_i$ is the difference of two products
  of units, and $s_i$ is the common valuation to  both terms of the polynomial.
  In this case, the twisted cubic index is $\min_i( \nu_K(T_i))$.

  Now, notice that the Plücker coordinates of
 an octad $O'$ in the same $\PGL$-orbit as $O$ are equal to those of $O$ up to
  some constant that comes from the determinant of the matrix that acts,
  and some
  other constants
  which depends
  on the projective normalisation on each
  point in $\PP^3$.
  Because each index $\oA$ through $\oH$ occurs the same number of times in both terms for each of the 840 polynomials, the octad $O'$ falls into the same case as $O$ does in the previous paragraph: that is, in at least one polynomial both terms do not have equal valuation, or in all polynomials both terms do have the same valuation. In the last case,  the powers $\pi^{s_i}$ might change, but the $T_i$ only change by a unit. %
  Hence, the twisted cubic index is preserved under the action of $\PGL$.
\end{proof}

\begin{proposition}\label{prop:coord-transform}
  Suppose $O = (O_\oA, \ldots, O_\oH)$ is an octad with valuation data
  $v$. Suppose we have a subset $S \subset \{\oA, \ldots, \oH\}$ such
  that the points $\{O_i\,|\, i\in S\}$ collide with the point $(1:0:0:0)$ modulo
  $\pi$, and suppose none of the points in
  $\{O_i\,|\, i\in T\}$ where $T = \{\oA, \ldots, \oH\} \backslash S$
  collide with $(1:0:0:0)$ modulo $\pi$. Consider the octad $O'$ obtained by
  applying to $O$ the $\PGL$ transformation
\begin{equation}\label{eq:Mtrans}
\arraycolsep=0.4\arraycolsep%
M = \mbox{\scriptsize$\begin{bmatrix}
  \pi &0 &0 &0\\ 0 &1 &0 &0\\ 0 &0 &1 &0\\ 0 &0 &0 &1
\end{bmatrix}$}\,.
\end{equation}

Then $O'$ has valuation data $v - v_\pt^S + v_\pl^T$
\end{proposition}

\begin{proof}
  For each four element subset $S \subset \{\oA, \ldots, \oH\}$, we have to
  count how many factors $\pi$ the determinant of the matrix
  $[ O_i ]_{i\in S}$ gains. We gain one factor of
  $\pi$ because all the $x$-coordinates are multiplied by $\pi$, and we lose
  one factor of $\pi$ for each point that has to be renormalised because of
  the $\PGL$ transformation.

For the points that reduce to $(1:0:0:0)$ modulo $\pi$, the affine coordinates are all divisible by $\pi$ after the transformation, and we have to divide them all by $\pi$ to renormalise. For points that do not reduce to $(1:0:0:0)$ the second, third, or fourth affine coordinate must be a unit, so at worst we have to renormalise by a unit, but this does not affect the number of factors $\pi$ in the determinant.

Carefully counting the numbers of factors $\pi$ gained and lost in each of the determinants and noting that the twisted cubic index does not change as of Prop.~\ref{prop:tci-invariant}, we indeed get that $O'$ has the expected valuation data.
\end{proof}

 The following lemma is probably known by the experts. Since we did not find an appropriate reference we include a brief proof.

\begin{lemma}\label{lem:PGL-generated}
Any $\PGL$ transformation is the composite  of elements of
  $\textup{PGL}_3(\mathcal{O}_{\overline{K}})$  and of diagonal matrices
  \begin{equation}\label{eq:Mq}
\arraycolsep=0.4\arraycolsep%
M_q = \mbox{\scriptsize$\begin{bmatrix}
  \pi^q &0 &0 &0\\ 0 &1 &0 &0\\ 0 &0 &1 &0\\ 0 &0 &0 &1
\end{bmatrix}$}\,.
\end{equation}
  for $q \in \Q_{>0}$,
  where $\pi^q \in \overline{K}$ is some element of valuation $q$.
\end{lemma}

\begin{proof}
  The group $\PGL$ is generated by the
  three types of elementary matrices describing row operations, those
  typically used in Gaussian elimination to perform row switching,
  multiplication and addition. Each of these can be written as a product of
  elements of $\textup{PGL}_3(\mathcal{O}_{\overline{K}})$ and the $M_{q}$. The details for this are omitted.

\end{proof}

\begin{proposition}\label{prop:equivalence-of-valuation-data}
Let $v$ be a valuation data. Then there exists a valuation data $v^\#$ with the following property. Suppose that $O$ is a regular octad with valuation data $v$, and $O'$ is the octad obtained from $O$ by applying the unique $\PGL$ transformation, putting the points $O'_\oA, \ldots, O'_\oE$ in standard position. Then $O'$ has valuation data $v^\#$.
\end{proposition}

\begin{proof}
By Lem.~\ref{lem:PGL-generated}, $\PGL$ is generated by elements of $\textup{PGL}_3(\mathcal{O}_{\overline{K}})$ and the diagonal matrices $M_{q}$ as defined in Eq.~\eqref{eq:Mq}.
The transformation in $\textup{PGL}_3(\mathcal{O}_{\overline{K}})$ do not affect the valuation of any of the Pl\"ucker coordinates.
Therefore, by applying Prop.~\ref{prop:coord-transform} (with $\pi^q$ in place of $\pi$), we see that the difference between the valuation data $v'$ of $O'$ and $v$ of $O$ is contained in the span $W_8$ of the valuation data of the shape $v_{\pt}^S - v_{\pl}^T$, where $S \sqcup T = \{\oA, \ldots, \oH\}$. This is an 8-dimensional $\Q$-vector space, which has basis $(v_{\pl}^T)_T$ where $T$ ranges over the sets containing exactly 7 points.

The valuation data $v'$ must have the property that $v'_{\oA\oB\oC\oD} = v'_{\oA\oB\oC\oE} = v'_{\oA\oB\oD\oE} = v'_{\oA\oC\oD\oE} = v'_{\oB\oC\oD\oE} = 0$. Moreover, it must also satisfy that $\min(\nu_K(O_\oF^{(1)})$, $\nu_K(O_\oF^{(2)})$, $ \nu_K(O_\oF^{(3)})$, $ \nu_K(O_\oF^{(4)})) = \min(v'_{\oB\oC\oD\oH}$, $ v'_{\oA\oC\oD\oH}$, $ v'_{\oA\oB\oD\oH}$, $ v'_{\oA\oB\oC\oH}) = 0$ as otherwise all the coordinates of $O_\oF$ would have positive valuation, contradicting the fact that $O_\oF$ must be properly normalised. Similarly, $\min(v'_{\oA\oB\oC\oG}, v'_{\oA\oB\oD\oG}, v'_{\oA\oC\oD\oG}, v'_{\oB\oC\oD\oG}) = \min(v'_{\oA\oB\oC\oH}, v'_{\oA\oB\oD\oH}, v'_{\oA\oC\oD\oH}, v'_{\oB\oC\oD\oH}) = 0$. As it turns out, there is a unique $w \in W_8$ such that $v + w$ satisfies these conditions.

The first five linear
conditions restrict the coset $v + w$ to a coset $\tilde{v}+\tilde{w}$ with
$\tilde{w}$ in a 3-dimensional subspace $W_3$ of $W_8$. It turns out that the
valuations of $\tilde{w}$ at ${\oA\oB\oC\oF}$, ${\oA\oB\oD\oF}$,
${\oA\oC\oD\oF}$ and ${\oB\oC\oD\oF}$ are all equal, let's say
$\tilde{w}_\oF$.  The minimum of $\tilde{v}+\tilde{w}$ on
these 4 valuations can thus be set to zero by solving the equation
$\tilde{w}_\oF = -\min(\tilde{v}_{\oA\oB\oC\oH},
\tilde{v}_{\oA\oB\oD\oH},\tilde{v}_{\oA\oC\oD\oH},\tilde{v}_{\oB\oC\oD\oH})$,
and similarly for the minima related to $O_\oG$ and $O_\oH$.  Now, the linear
map
$W_3 \to \Q^3:\ \tilde{w} \mapsto (\tilde{w}_\oF, \tilde{w}_\oG,
\tilde{w}_\oF)$ is a bijection. Therefore this $\tilde{w} \in W_3$, and thus
the corresponding $w \in W_8$, exists and is unique.

To conclude, the valuation data $v'$ must be equal to $v^\# = v + w$. Note
that $w$ does not depend at all on the octad $O$, but only on the valuation
data $v$.
\end{proof}

\section{Octad Building Blocks}
\label{sec:ocdiagrams}

In Sec.~\ref{sec:CombiningOctads} we introduce pictures which are
analogues of so-called ``cluster pictures''~\cite{m2d2, otherm2d2} in
the non-hyperelliptic setting. To do this, we first give fundamental instances
of valuation data. These are the building blocks from which our pictures are
built, and are the analogue of clusters of size $2$, $3$ and $4$ for genus 3 hyperelliptic curves. There is
also a new phenomenon in the non-hyperelliptic setting, whereby
non-hyperelliptic curves acquire, upon reduction, a hyperelliptic involution, \textit{i.e.}\ an involution for which the quotient is a curve of genus 0, upon
reduction. Accordingly, we introduce valuation data that recognises this. \medskip

These building blocks have already been introduced in Tab.~\ref{tab:introblocksdefinition} in terms of the elementary vectors of Def.~\ref{def:standard-valuation-data}.
To each equivalence class both a picture and a subspace of
the vector space $E_3 \simeq \F_2^{6}$ have been associated (see Sec.~\ref{subsec:introcombinatorialanalysis} regarding these subspaces, and, for more detail, Sec.~\ref{sec:sympl-f_2-vect}).

More generally, the action of $\textup{Sp}(6,2)$ on a building block is well defined and this leads to consider the action of Cremona transformations.
It follows almost by the definition of a regular Cayley octad that their
Plücker coordinates are non-zero, and that this remains true on orbits under
Cremona action.
We refine this result in terms of the local valuations. We show here that
Cremona transformations simply send an octad whose valuation data is one of these building blocks to an octad whose valuation data is a building block of the same type.

Throughout this entire section $K$ is still a local field, as in the previous section.

\subsection{Alpha Blocks}

\begin{definition}
    An $\bm{\alpha}$-block is any of the blocks given in rows 1-4 of Tab.~\ref{tab:introblocksdefinition}, up to $S_8$-permutation of the indices.
\end{definition}

Given the $\PGL$ equivalences stated in Prop.~\ref{thm:equiv-objects}, we use the labels {$\bm{\alpha}_1$} and {$\bm{\alpha}_2$} in situations where the exact points do not matter. We associate to these octads the respective pictures in Tab.~\ref{tab:introblocksdefinition}, noting in particular that these pictures are independent of $\PGL$ by Prop.~\ref{thm:equiv-objects}. In Tab.~\ref{tab:introblocksdefinition} we also associate to each equivalence class a one-dimensional subspace of $E_3$

\begin{remark}
  As a consequence of the definitions in Tab.~\ref{tab:introblocksdefinition}, the two pictures in
  Fig.~\ref{fig:plsym} are considered to be the same, and one can think of an ${\bm{\alpha}}_{2}$-block as a simply a partition of the octad into two sets of four. For the purposes of drawing diagrams, however, we break the symmetry of the partition, making it easier to read pictures with multiple such rectangles. We may also indicate one side of a partition with four identical symbols such as  \tikzsetnextfilename{GreenNode}\begin{tikzpicture}
\node[draw,circle,scale=0.5,LimeGreen,fill=LimeGreen] at (0,0) {};
\end{tikzpicture} or  \tikzsetnextfilename{GreenStar}\begin{tikzpicture}
\node[draw,star,star point ratio = 2.5, scale=0.25,LimeGreen,fill=LimeGreen] at (0,0) {};
\end{tikzpicture}, as in Tab.~~\ref{tab:smoctadsc02}--\ref{tab:smhoctadsc56}.

  Further, pictures with different labelling are considered different, hence the two
  pictures in Fig.~\ref{fig:twdiag} are different pictures.
  \begin{figure}[htbp]
    \centering
    \begin{subfigure}[c]{0.42\textwidth}
      \tikzsetnextfilename{TypePLa}
      \begin{tikzpicture}[scale=0.22]
        \ptslabel \plane{0}
      \end{tikzpicture}
      \qquad\qquad \tikzsetnextfilename{TypePLb}
      \begin{tikzpicture}[scale=0.22]
        \ptslabel \plane{4}
      \end{tikzpicture}
      \caption{Equivalent ${\bm{\alpha}}_2$-block pictures}
      \label{fig:plsym}
    \end{subfigure}
    \begin{subfigure}[c]{0.42\textwidth}
      \tikzsetnextfilename{TWa}
      \begin{tikzpicture}[scale=0.25]
        \ptslabel
        \twinbig{0}
      \end{tikzpicture}
      \qquad\qquad
      \tikzsetnextfilename{TWb}
      \begin{tikzpicture}[scale=0.25]
        \ptslabel
        \twinbig{1}
      \end{tikzpicture}
      \caption{Inequivalent ${\bm{\alpha}}_1$-block pictures}
      \label{fig:twdiag}
    \end{subfigure}
    \caption{Equivalence of pictures}
    \label{fig:eqdiag}
  \end{figure}
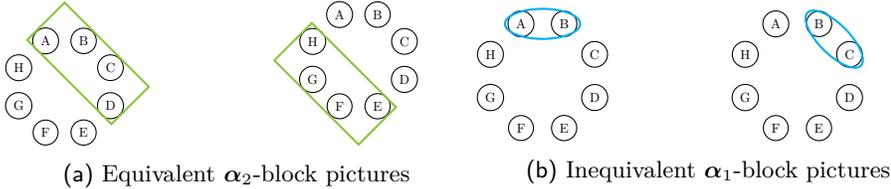
\end{remark}

We prove now that the valuation data for $\bm{\alpha}_{1}$-blocks (resp.\ $\bm{\alpha}_{2}$-blocks) are $\PGL$-equivalent.

\begin{proposition}\label{thm:equiv-objects}
\label{thm:obpgl3}We have the following $\PGL$-equivalences of valuation data.\\[-0.50cm]
\begin{itemize}
\item[(i)] The valuation data ${\bm{\alpha}}_{1{\mathrm{a}}}^{\oA\oB}$ and ${\bm{\alpha}}_{1{\mathrm{b}}}^{\oA\oB}$ are equivalent.
\item[(ii)] The valuation data ${\bm{\alpha}}_{2{\mathrm{a}}}^{\oA\oB\oC\oD}$, ${\bm{\alpha}}_{2{\mathrm{b}}}^{\oA\oB\oC\oD}$, and ${\bm{\alpha}}_{2{\mathrm{b}}}^{\oE\oF\oG\oH}$ are equivalent.
\end{itemize}
Moreover, an octad with these valuation data is $\textup{PGL}_3(\mathcal{O}_{\overline{K}})$-equivalent to an octad with five points in standard position, and restricted to such octads these valuation data exhaust the two equivalence classes.

\end{proposition}

\begin{proof}
  The equivalences between the valuation data and the fact that they are exhaustive, are obtained by changing the set of 5 points that are in standard position. The procedure described in the proof of Prop.~\ref{prop:equivalence-of-valuation-data} explains how the valuation data changes under these procedures. This procedure has been implemented, along with the verification of these equivalences and their completeness~\cite[file \texttt{EquivalenceOfValuationData.m}]{BGit23}.
\end{proof}

\subsubsection*{Cremona Transformations on Alpha Blocks}

Cremona transformations act on the $S_8$-equivalence classes of ${\bm{\alpha}}$-blocks in a well defined way.

\begin{table}[htbp]
  \centering
  \renewcommand{\arraystretch}{1.3}
  \newcommand{\DoScale}[1]{\scalebox{0.6}{#1}}%
  \let\MyScale\relax
  \begin{tabular}{
    >{\centering\arraybackslash}m{0.2cm}
    >{\centering\arraybackslash\collectcell\MyScale}m{1.5cm}<{\endcollectcell}
    >{\centering\arraybackslash}m{2cm}
    >{\centering\arraybackslash\collectcell\MyScale}m{1.5cm}<{\endcollectcell}
    p{0.3cm}
    >{\centering\arraybackslash}m{0.2cm}
    >{\centering\arraybackslash\collectcell\MyScale}m{1.5cm}<{\endcollectcell}
    >{\centering\arraybackslash}m{2cm}
    >{\centering\arraybackslash\collectcell\MyScale}m{1.5cm}<{\endcollectcell}
    }
    \#
    & Picture
    & Partition
    & Image
    &
    & \#
    & Picture
    & Partition
    & Image\\
    \noalign{\global\let\MyScale\DoScale}
    \cmidrule{1-4}\cmidrule{6-9}
    \morecmidrules
    \cmidrule{1-4}\cmidrule{6-9}
    $1$
    & \tikzsetnextfilename{CremonaLeft1}
      \begin{tikzpicture}[scale=0.3]
        \ptslabel
        \twinbig{0}
      \end{tikzpicture}
    & $\oA\oB\oC\oD \mid \oE\oF\oG\oH$
    &\tikzsetnextfilename{CremonaRight1v2}
      \begin{tikzpicture}[scale=0.3]
      \ptslabel
        \twinbig{0}
      \end{tikzpicture}
    &
    & $3$
    & \tikzsetnextfilename{CremonaLeft3}
      \begin{tikzpicture}[scale=0.3]
        \ptslabel
        \plane{0}
      \end{tikzpicture}
    & $\oA\oB\oC\oD \mid \oE\oF\oG\oH$
    & \tikzsetnextfilename{CremonaRight3b}
      \begin{tikzpicture}[scale=0.3]
        \ptscustomlabel{A}{B}{C}{D}{E}{F}{G}{H}
        \plane{0}
      \end{tikzpicture}
    \\[0.6cm]
    \cmidrule{1-4}\cmidrule{6-9}
    $2$
    & \tikzsetnextfilename{CremonaLeft2}
      \begin{tikzpicture}[scale=0.3]
        \ptslabel
        \twinbig{0}
      \end{tikzpicture}
    & $\oA\oC\oD\oE \mid \oB\oF\oG\oH$
    & \tikzsetnextfilename{CremonaRight2}
      \begin{tikzpicture}[scale=0.3]
        \ptslabel
        \plane{1}
      \end{tikzpicture}
    &
    & $4$
    & \tikzsetnextfilename{CremonaLeft4}
      \begin{tikzpicture}[scale=0.3]
        \ptslabel
        \plane{0}
      \end{tikzpicture}
    & $\oA\oB\oE\oF \mid \oC\oD\oG\oH$
    & \tikzsetnextfilename{CremonaRight4b}
      \begin{tikzpicture}[scale=0.3]
        \ptslabel
        \plane{0}
      \end{tikzpicture}
  \end{tabular}
  \caption{\label{table:Cremona-objects}Cremona transformations on
    $\bm{\alpha}$-blocks}
\end{table}

\begin{proposition}\label{CreTransTwins}
Suppose $O$ is an octad whose valuation data is an $\bm{\alpha}$-block. The action of any Cremona transformation on $O$ gives an octad whose valuation data is also an $\bm{\alpha}$-block, as described in Tab.~\ref{table:Cremona-objects}.
\end{proposition}

\begin{remark}
Since Cremona transformations are involutions, Tab.~\ref{table:Cremona-objects} can be read left-to-right and right-to-left.
\end{remark}

\begin{proof}

  We consider, one after the other, octads that realise valuation data of
  Prop.~\ref{thm:equiv-objects}, with the five points in standard
  position that correspond to the desired Cremona action. It is not difficult
  to see that the octad that results of this Cremona transformation has
  generically one of the valuation data of the table.  The subtle part is to
  prove that if a valuation entry has higher valuation, it contradicts the
  valuation data of the starting octad, \textit{i.e.}\ they imply that some
  Pl\"ucker coordinates outside the support of the starting valuation data
  must be zero modulo $\pi$.

  When the dependence of the octad points in $\pi$ results in a simple scaling
  of the Plücker coordinates by a power of $\pi$, this is immediate.  But in
  general, those verifications are relatively cumbersome. Nevertheless, they
  lend themselves well to automation in an algebraic calculation software (a
  \MAGMA verification script that does this is available at~\cite[file
  \texttt{PaperProof.m}]{BGit23}).

  In the first place, we can assume that in the ${\bm{\alpha}}_1$ case we
  actually are in the ${\bm{\alpha}}_{1{\mathrm{a}}}^{\oA\oB}$ case. Up to
  permutation the two Cremona transformations to consider are
  $\oA\oB\oC\oD \mid \oE\oF\oG\oH$ and $\oA\oC\oD\oE \mid \oB\oF\oG\oH$. In
  the first case, and again by Prop.~\ref{thm:equiv-objects} we can assume
  the points in the octad be given by the columns of the first matrix below
  and to be sent to the columns of the second one:
\begin{displaymath}
  \arraycolsep=0.4\arraycolsep%
  \mbox{\scriptsize$\begin{bmatrix}
    a_0& a_0+\pi b_0 & c_0& 1& 1 & 0 & 0 & 0  \\
    a_1& a_1+\pi b_1 & c_1 & 1 & 0 & 1 & 0 & 0  \\
    a_2&a_2+\pi b_2 & c_2 & 1 & 0 & 0 & 1 & 0  \\
    a_3& a_3+\pi b_3 & c_3 & 1 & 0 & 0 & 0 & 1
  \end{bmatrix}$}\mapsto
  \mbox{\scriptsize$\begin{bmatrix}
    a_0^{-1}& a_0^{-1}+\pi \frac{-b_0}{a_0(a_0+\pi b_0)} & c_0^{-1}& 1& 1 & 0 & 0 & 0  \\
    a_1^{-1}& a_1^{-1}+\pi \frac{-b_1}{a_1(a_1+\pi b_1)} & c_1^{-1} & 1 & 0 & 1 & 0 & 0  \\
    a_2^{-1}&a_2^{-1}+\pi \frac{-b_2}{a_2(a_2+\pi b_2)} & c_2^{-1} & 1 & 0 & 0 & 1 & 0  \\
    a_3^{-1}& a_3^{-1}+\pi \frac{-b_3}{a_3(a_3+\pi b_3)} & c_3^{-1} & 1 & 0 & 0 & 0 & 1
  \end{bmatrix}$}\,.
\end{displaymath}
We clearly see the produced ${\bm{\alpha}}_{1{\mathrm{a}}}^{\oA\oB}$, at least
generically.  It turns out that this is still true, with the notable exception
when the first and third octad points satisfy
\begin{displaymath}
  \renewcommand{\arraystretch}{0.8}
  \mbox{\scriptsize$\left\{\begin{array}{rcl}
  a_0\,a_2\,c_0\,c_1 - a_1\,a_2\,c_0\,c_1 - a_0\,a_1\,c_0\,c_2 + a_1\,a_2\,c_0\,c_2 + a_0\,a_1\,c_1\,c_2 - a_0\,a_2\,c_1\,c_2&=&0,\\
  a_0\,a_3\,c_0\,c_1 - a_1\,a_3\,c_0\,c_1 - a_0\,a_1\,c_0\,c_3 + a_1\,a_3\,c_0\,c_3 + a_0\,a_1\,c_1\,c_3 - a_0\,a_3\,c_1\,c_3&=&0,\\
  a_0\,a_3\,c_0\,c_2 - a_2\,a_3\,c_0\,c_2 - a_0\,a_2\,c_0\,c_3 + a_2\,a_3\,c_0\,c_3 + a_0\,a_2\,c_2\,c_3 - a_0\,a_3\,c_2\,c_3&=&0,\\
  a_1\,a_3\,c_1\,c_2 - a_2\,a_3\,c_1\,c_2 - a_1\,a_2\,c_1\,c_3 + a_2\,a_3\,c_1\,c_3 + a_1\,a_2\,c_2\,c_3 - a_1\,a_3\,c_2\,c_3&=&0\,.
  \end{array}\right.$}
\end{displaymath}
But these conditions impose that the octad points are on a twisted cubic
curve, (in this case the reduction is then hyperelliptic, see Sec.~\ref{SS:hyper}).

For the transformation associated to $\oA\oC\oD\oE \mid \oB\oF\oG\oH$ we have the following matrices description and we proceed in the same fashion,
\begin{displaymath}
  \arraycolsep=0.4\arraycolsep%
  \mbox{\scriptsize$\begin{bmatrix}
     1 & 1+\pi b_0 & 0 & 0 & 0 & 1 & g_0 & h_0  \\
     0 & \pi b_1 & 1 & 0 & 0 & 1 & g_1 & h_1\\
     0 & \pi b_2 & 0 & 1 & 0 & 1 & g_2 & h_2\\
     0 & \pi b_3 & 0 & 0 & 1 & 1 & g_3 & h_3
  \end{bmatrix}$}\mapsto
  \mbox{\scriptsize$\begin{bmatrix}
     1 & \pi \frac{1}{1+\pi b_0} & 0 & 0 & 0 & 1 & g_0^{-1} & h_0^{-1}  \\
     0 &  b_1^{-1} & 1 & 0 & 0 & 1 & g_1^{-1} & h_1^{-1}\\
     0 &  b_2^{-1} & 0 & 1 & 0 & 1 & g_2^{-1} & h_2^{-1}\\
     0 &  b_3^{-1} & 0 & 0 & 1 & 1 & g_3^{-1} & h_3^{-1}
  \end{bmatrix}$},
\end{displaymath}
and we see that an ${\bm{\alpha}}_{2{\mathrm{a}}}^{\oB\oC\oD\oE}$-block appears.  The
verification of this fact uses the relations~\eqref{eq:cayley} that must
satisfy both octads, in addition to the conditions from the valuation
data. There is one exception, given by the equations
\begin{displaymath}
  \renewcommand{\arraystretch}{0.8}
  \mbox{\scriptsize$
    \left\{\begin{array}{rcl}
      g_0\,g_2\,h_0\,h_1 - g_1\,g_2\,h_0\,h_1 - g_0\,g_1\,h_0\,h_2 + g_1\,g_2\,h_0\,h_2 + g_0\,g_1\,h_1\,h_2 - g_0\,g_2\,h_1\,h_2 &=& 0\,,\\
      g_0\,g_3\,h_0\,h_1 - g_1\,g_3\,h_0\,h_1 - g_0\,g_1\,h_0\,h_3 + g_1\,g_3\,h_0\,h_3 + g_0\,g_1\,h_1\,h_3 - g_0\,g_3\,h_1\,h_3 &=& 0\,,\\
      g_0\,g_3\,h_0\,h_2 - g_2\,g_3\,h_0\,h_2 - g_0\,g_2\,h_0\,h_3 + g_2\,g_3\,h_0\,h_3 + g_0\,g_2\,h_2\,h_3 - g_0\,g_3\,h_2\,h_3 &=& 0\,,\\
      g_1\,g_3\,h_1\,h_2 - g_2\,g_3\,h_1\,h_2 - g_1\,g_2\,h_1\,h_3 + g_2\,g_3\,h_1\,h_3 + g_1\,g_2\,h_2\,h_3 - g_1\,g_3\,h_2\,h_3 &=& 0\,,\\
      b_1\,b_3\,h_1\,h_2 - b_2\,b_3\,h_1\,h_2 - b_1\,b_2\,h_1\,h_3 + b_2\,b_3\,h_1\,h_3 + b_1\,b_2\,h_2\,h_3 - b_1\,b_3\,h_2\,h_3 &=& 0\,.
    \end{array}\right.$}
\end{displaymath}
But again, this corresponds to an octad whose points lie on a twisted cubic curve.\smallskip

The ${\bm{\alpha}}_2$ case follows after similar considerations, starting for
instance from an octad of valuation data
${\bm{\alpha}}_{2{\mathrm{b}}}^{\oA\oB\oC\oD}$ given as
\begin{displaymath}
  \arraycolsep=0.4\arraycolsep%
  \mbox{\scriptsize$\begin{bmatrix}
    1+\pi\,a_0& 1+\pi b_0 & 1+\pi\,(\lambda a_0+\lambda' b_0) & 1 & 1 & 0 & 0 & 0  \\
    1+\pi\,a_1& 1+\pi b_1 & 1+\pi\,(\lambda a_1+\lambda' b_1) & 1 & 0 & 1 & 0 & 0  \\
    1+\pi\,a_2& 1+\pi b_2 & 1+\pi\,(\lambda a_2+\lambda' b_2) & 1 & 0 & 0 & 1 & 0  \\
    1+\pi\,a_3& 1+\pi b_3 & 1+\pi\,(\lambda a_3+\lambda' b_3) & 1 & 0 & 0 & 0 & 1
  \end{bmatrix}$}\,.
\end{displaymath}
\end{proof}

\subsection{Chi Blocks}

In this section, we define new building blocks, ${\bm{\chi}}$-blocks.

\begin{definition}
     A $\bm{\chi}$-block is any of the blocks given in rows 5-10 of Tab.~\ref{tab:introblocksdefinition}, up to $S_8$-permutation of the indices.
\end{definition}

Just as we did for the ${\bm{\alpha}}$-blocks, we associate to ${\bm{\chi}}$-blocks both $\textup{PGL}_3$-independent
pictures and subspaces of $E_3$, as in Tab.~\ref{tab:introblocksdefinition}. Note from the definitions that there is no distinction between the left and right triple of the ``needle'', \textit{i.e.}\ they can be exchanged without changing the picture.

\begin{proposition}\label{PGLTypes}We have the following $\PGL$-equivalences of valuation data.\\[-0.20cm]
  \begin{itemize}
  \item[(i)] The valuation data
    ${\bm{\chi}}_{1{\mathrm{a}}}^{\oA\oB|\oC\oD\oE|\oF\oG\oH}$,
    ${\bm{\chi}}_{1{\mathrm{b}}}^{\oA\oB|\oC\oD\oE|\oF\oG\oH}$ and
    ${\bm{\chi}}_{1{\mathrm{c}}}^{\oA\oB|\oC\oD\oE|\oF\oG\oH}$ are equivalent. \\[-0.20cm]
  \item[(ii)] The valuation data ${\bm{\chi}}_{2{\mathrm{a}}}^{\oA\oB\oC}$,
    ${\bm{\chi}}_{2{\mathrm{b}}}^{\oA\oB\oC}$ and ${\bm{\chi}}_{2{\mathrm{c}}}^{\oA\oB\oC}$ are
    equivalent.
  \end{itemize}
  Moreover, an octad with these valuation data is
  $\mathrm{PGL}_3(\mathcal{O}_{\overline{K}})$-equivalent to an octad with five points in standard
  position, and restricted to such octads, up to switching $\oC\oD\oE$ and
  $\oF\oG\oH$ in the case of ${\bm{\chi}}_1$, these valuation data exhaust the
  two equivalence classes.
\end{proposition}
\begin{proof}
The proof is the same as for Prop.~\ref{thm:equiv-objects}.
\end{proof}

\subsubsection*{Cremona Transformation on Chi Blocks}

Similarly to $\bm{\alpha}$-blocks, Cremona transformations act on the
$S_8$-equivalence classes of octad pictures which are ${\bm{\chi}}$-blocks in a well
defined way.

\begin{table}[htbp]
  \centering
  \renewcommand{\arraystretch}{1.3}
  \newcommand{\DoScale}[1]{\scalebox{0.6}{#1}}%
  \let\MyScale\relax
  \begin{tabular}{
    >{\centering\arraybackslash}m{0.2cm}
    >{\centering\arraybackslash\collectcell\MyScale}m{1.5cm}<{\endcollectcell}
    >{\centering\arraybackslash}m{2cm}
    >{\centering\arraybackslash\collectcell\MyScale}m{1.5cm}<{\endcollectcell}
    p{0.3cm}
    >{\centering\arraybackslash}m{0.2cm}
    >{\centering\arraybackslash\collectcell\MyScale}m{1.5cm}<{\endcollectcell}
    >{\centering\arraybackslash}m{2cm}
    >{\centering\arraybackslash\collectcell\MyScale}m{1.5cm}<{\endcollectcell}
    }
    \#
    & Picture
    & Partition
    & Image
    &
    & \#
    & Picture
    & Partition
    & Image\\
    \noalign{\global\let\MyScale\DoScale}
    \cmidrule{1-4}\cmidrule{6-9}
    \morecmidrules
    \cmidrule{1-4}\cmidrule{6-9}
    $1$
    & \tikzsetnextfilename{CremonaLeft5}
      \begin{tikzpicture}[scale=0.3]
        \ptslabel
        \TB{0}
      \end{tikzpicture}
    & $\oA\oB\oC\oD \mid \oE\oF\oG\oH$
    & \tikzsetnextfilename{CremonaRight5c}
      \begin{tikzpicture}[scale=0.3]
        \ptscustomlabel{A}{B}{C}{D}{E}{F}{G}{H}
        \TB{0}
      \end{tikzpicture}
    &
    & $4$
    & \tikzsetnextfilename{CremonaLeft7}
      \begin{tikzpicture}[scale=0.3]
        \ptslabel
        \TA{0}
      \end{tikzpicture}
    & $\oB\oC\oD\oE \mid \oA\oF\oG\oH$
    & \tikzsetnextfilename{CremonaRight7}
      \begin{tikzpicture}[scale=0.3]
        \ptslabel
        \TA{0}
      \end{tikzpicture} \\[0.6cm]
    \cmidrule{1-4}\cmidrule{6-9}
    $2$
   & \tikzsetnextfilename{CremonaLeft5}
      \begin{tikzpicture}[scale=0.3]
        \ptslabel
        \TB{0}
      \end{tikzpicture}
    & $\oA\oB\oD\oE \mid \oC\oF\oG\oH$
    & \tikzsetnextfilename{CremonaRight5b}
      \begin{tikzpicture}[scale=0.3]
       \ptslabel
        \TA{0}
      \end{tikzpicture}
    &
    & $5$
    & \tikzsetnextfilename{CremonaLeft8}
      \begin{tikzpicture}[scale=0.3]
        \ptslabel
        \TA{0}
      \end{tikzpicture}
    & $\oA\oC\oD\oF \mid \oB\oE\oG\oH$
    & \tikzsetnextfilename{CremonaRight8b}
      \begin{tikzpicture}[scale=0.3]
        \ptslabel
        \TA{4}
      \end{tikzpicture}
      \\[0.6cm]
    \cmidrule{1-4}\cmidrule{6-9}
    $3$
    & \tikzsetnextfilename{CremonaLeft6}
      \begin{tikzpicture}[scale=0.3]
        \ptslabel
        \TA{0}
      \end{tikzpicture}
    & $\oA\oB\oC\oH \mid \oD\oE\oF\oG$
    & \tikzsetnextfilename{CremonaRight6b}
      \begin{tikzpicture}[scale=0.3]
        \ptslabel
        \TA{0}
      \end{tikzpicture}
    &
    &
    &
    &
    &
  \end{tabular}
  \caption{\label{table:Cremona-types}Cremona transformations on $\bm{\chi}$-blocks}
\end{table}

\begin{proposition}\label{CreTransTypes}
Suppose $O$ is an octad whose valuation data is a $\bm{\chi}$-block. The action of any Cremona transformation on $O$ gives an octad whose valuation data is also a $\bm{\chi}$-block, as described in Tab.~\ref{table:Cremona-types}.
\end{proposition}

\begin{proof} Given the symmetries between some labellings we only need to
  consider the cases listed in Tab.~\ref{table:Cremona-types}.  Now we follow
  the same strategy as in Prop.~\ref{CreTransTwins} to study each of
  them. Thanks to Prop.~\ref{PGLTypes}, we can start with
  ${\bm{\chi}}_{2\mathrm{a}}^{\oA\oB\oC}$ in case $\#$1 from
  Tab.~\ref{table:Cremona-types}, \textit{i.e.}\ %

  \begin{displaymath}
    \arraycolsep=0.4\arraycolsep%
    \mbox{\scriptsize$\begin{bmatrix}
      1 & 0 & 0 & 0 & 1 & f_0          & f_0             & f_0            \\
      0 & 1 & 0 & 0 & 1 & f_0 +\pi f_1 & f_0 +\pi \lambda f_1 & f_0 +\pi
      \lambda' f_1 \\
      0 & 0 & 1 & 0 & 1 & f_0 +\pi f_2 & f_0 +\pi \lambda f_2 & f_0 +\pi \lambda' f_2\\
      0 & 0 & 0 & 1 & 1 & f_3          & g_3             & h_3
    \end{bmatrix}$}\,.
  \end{displaymath}
Recall the conventions we have made regarding Cremona transformations in Rem.~\ref{rmk:cremonatransformations}. It is immediate to see that the Cremona transformation
$\oA\oB\oC\oD \mid \oE\oF\oG\oH$ maps it to another
${\bm{\chi}}_{2\mathrm{a}}^{\oA\oB\oC}$, and it remains to prove that no other
relations appears, except when this octad lies on a twisted cubic curve
(see~\cite[file \texttt{PaperProof.m}]{BGit23}).
In case $\#$2, the Cremona transformation is given by the permutation $\oA\oB\oD\oE \mid \oC\oF\oG\oH$, and it maps
${\bm{\chi}}_{2\mathrm{a}}^{\oA\oB\oC}$ to
${\bm{\chi}}_{1\mathrm{c}}^{\oA\oB|\oC\oD\oE|\oF\oG\oH}$,
  \begin{displaymath}
    \arraycolsep=0.4\arraycolsep%
    \mbox{\scriptsize$\begin{bmatrix}
      1 & 0 & 1 & 0 & 0 & f_0\,\pi & \lambda\, f_0\,\pi & \lambda'\,f_0\,\pi \\
      0 & 1 & 1 & 0 & 0 & f_1\,\pi & \lambda\,f_1\,\pi & \lambda'\,f_1\,\pi \\
      0 & 0 & 1 & 1 & 0 & f_2 + f_0\,\pi  & g_2 + \lambda\,f_0\,\pi & h_2 + \lambda'\,f_0\,\pi \\
      0 & 0 & 1 & 0 & 1 & 1+f_0\,\pi & 1+\lambda\,f_0\,\pi &
      1+\lambda'\,f_0\,\pi
    \end{bmatrix}$}
    \mapsto
    \mbox{\scriptsize$\begin{bmatrix}
      1 & 0 & 1 & 0 & 0 & f_0^{-1} & \lambda^{-1}\, f_0^{-1} & \lambda'^{-1}\,f_0^{-1} \\
      0 & 1 & 1 & 0 & 0 & f_1^{-1} & \lambda^{-1}\,f_1^{-1} & \lambda'^{-1}\,f_1^{-1} \\
      0 & 0 & 1 & 1 & 0 & \pi\,\frac{1}{\,f_2 + f_0\,\pi}  & \pi\,\frac{1}{g_2 + \lambda\,f_0\,\pi} & \pi\,\frac{1}{h_2 + \lambda'\,f_0\,\pi} \\
      0 & 0 & 1 & 0 & 1 & \pi\,\frac{1}{1+f_0\,\pi} &
      \pi\,\frac{1}{1+\lambda\,f_0\,\pi} & \pi\,\frac{1}{1+\lambda'\,f_0\,\pi}
    \end{bmatrix}$}\,.
  \end{displaymath}
We proceed in a similar fashion with the other cases, starting for instance from
${\bm{\chi}}_{1\mathrm{b}}^{\oA\oB|\oC\oD\oE|\oF\oG\oH}$,
\begin{displaymath}
  \arraycolsep=0.4\arraycolsep%
  \mbox{\scriptsize$\begin{bmatrix}
       1 & 1+\pi\,b_0   & 1 & 1 & 0 & 0 & 0 & 1  \\
      a_1& a_1+ \pi\,b_1& 1 & 0 & 1 & 0 & 0 & a_1\\
     a_2 & a_2+\pi\,b_2 & 1 & 0 & 0 & 1 & 0 & h_2\\
     a_2 & a_2+\pi\,b_3 & 1 & 0 & 0 & 0 & 1 & h_3
   \end{bmatrix}$}\,.
\end{displaymath}

\end{proof}

\subsection{Hyperelliptic Blocks}\label{SS:hyper}

We introduce the next two building blocks, the \textit{hyperelliptic blocks}.

\begin{definition}
A hyperelliptic block is any of the blocks given in rows 18-19 of Tab.~\ref{tab:introblocksdefinition}, up to $S_8$-permutation of the indices.
\end{definition}

Uniquely to hyperelliptic blocks, we associate a subset of $E_3$ (instead of a subspace), and a picture, as in Tab.~\ref{tab:introblocksdefinition}. The reason why hyperelliptic blocks have associated subsets instead of subspaces is to get the exact right combinations of compatible blocks, as in Tab.~\ref{tab:twoblockssubspaces}; this will become clear when we define the combinatorial structures in Sec.~\ref{sec:formalanalysisoctadpictures}.

\begin{proposition}\label{prop:TCuequiv}
  Let $O$ be a octad whose valuation data consists of a {\bf TCu} block together with one other building block. Then, the only other building blocks that can occur are
  $\bm{\alpha}_1$, $\bm{\chi}_2$ and $\bm{\phi}_3$-blocks.  Moreover, the {\bf
    TCu} block is preserved under $\PGL$-equivalence.
\end{proposition}

\begin{proof}
    The fact that the {\bf TCu} block is preserved under $\PGL$-equivalence immediately follows from Prop.~\ref{prop:tci-invariant}.
    Moreover, for the blocks $\bm{\alpha}_2$, $\bm{\chi}_1$ and $\bm{\phi}_1$, $\bm{\phi}_2$, we get into the first case of the proof of Prop.~\ref{prop:tci-invariant}, and the {\bf TCu} block cannot occur.
    Therefore, $\bm{\alpha}_1$, $\bm{\chi}_2$ and $\bm{\phi}_3$ are the only blocks compatible with {\bf TCu}.
\end{proof}

In the same way, hyperelliptic lines are preserved under $\PGL$-equivalence.

\begin{proposition}
The valuation data $\mathrm{{\bf Line}}^{\oA\oB\oC\oD}$ and $\mathrm{{\bf Line}}^{\oE\oF\oG\oH}$ are $\PGL$-equivalent. Moreover, an octad with these valuation data is $\mathrm{PGL}(\O_{\overline{K}})$-equivalent to an octad with five points in standard position, and restricted to such octads, these valuation data exhaust the equivalence class.
\end{proposition}

\begin{proof}
The proof is the same as for Prop.~\ref{thm:equiv-objects}.
\end{proof}

We end this section with the following result on the action of Cremona transformations on hyperelliptic blocks:

\begin{proposition}\label{prop:CreTransHE}
    Suppose $O$ is an octad whose valuation data is a hyperelliptic-block. The action of any Cremona transformation on $O$ gives an octad whose valuation data is also a hyperelliptic-block, as described in Tab.~\ref{table:Cremona-HE}.
\end{proposition}

    \begin{table}[htbp]
  \centering
  \renewcommand{\arraystretch}{1.3}
  \newcommand{\DoScale}[1]{\scalebox{0.6}{#1}}%
  \let\MyScale\relax
  \begin{tabular}{
    >{\centering\arraybackslash}m{0.2cm}
    >{\centering\arraybackslash\collectcell\MyScale}m{1.5cm}<{\endcollectcell}
    >{\centering\arraybackslash}m{2cm}
    >{\centering\arraybackslash\collectcell\MyScale}m{1.5cm}<{\endcollectcell}
    }
    \#
    & Picture
    & Partition
    & Image
    \\
    \noalign{\global\let\MyScale\DoScale}
    \cmidrule{1-4}
    \morecmidrules
    \cmidrule{1-4}
    $1$
    & \tikzsetnextfilename{CremonaLeftHE1}
      \begin{tikzpicture}[scale=0.3]
        \ptslabel
        \TCu
      \end{tikzpicture}
    & $\oA\oB\oC\oD \mid \oE\oF\oG\oH$
    &\tikzsetnextfilename{CremonaRightHE1}
      \begin{tikzpicture}[scale=0.3]
      \ptslabel
        \HE{0}
      \end{tikzpicture}
    \\[0.6cm]
    \cmidrule{1-4}
    $2$
    & \tikzsetnextfilename{CremonaLeftHE2}
      \begin{tikzpicture}[scale=0.3]
        \ptslabel
        \HE{0}
      \end{tikzpicture}
    & $\oA\oB\oE\oF \mid \oC\oD\oG\oH$
    & \tikzsetnextfilename{CremonaRightHE2}
      \begin{tikzpicture}[scale=0.3]
        \ptslabel
        \HE{0}
      \end{tikzpicture}
    \\[0.6cm]
    \cmidrule{1-4}
    $3$
    & \tikzsetnextfilename{CremonaLeftHE3}
      \begin{tikzpicture}[scale=0.3]
        \ptslabel
        \HE{0}
      \end{tikzpicture}
    & $\oA\oB\oC\oE \mid \oD\oF\oG\oH$
    & \tikzsetnextfilename{CremonaRightHE3}
      \begin{tikzpicture}[scale=0.3]
        \ptslabel
        \HE{2}
      \end{tikzpicture}
  \end{tabular}
  \caption{\label{table:Cremona-HE}Cremona transformations on
    hyperelliptic-blocks}
\end{table}

\begin{proof}
    Case 1 is in fact already known in the literature, see Sec.~IX.3 in \cite{DolgachevOrtland}. Cases 2 and 3 are then a consequence of applying case 1 twice.
\end{proof}

\subsection{Phi Blocks}
\label{sec:phi-blocks}
We introduce the last class of building blocks, a special type of hyperelliptic blocks, ${\bm \phi}$-blocks.

\begin{definition}
     A $\bm{\phi}$-block is any of the blocks given in rows 11-17 of Tab.~\ref{tab:introblocksdefinition}, up to $S_8$-permutation of the indices.
\end{definition}

As before, and under Prop.~\ref{prop:PGLphiblocks}, we use the $\PGL$-independent labels ${\bm{\phi}}_1$, ${\bm{\phi}}_2$ and ${\bm{\phi}}_3$. We associate to the above valuation data pictures and subgroups as in Tab.~\ref{tab:introblocksdefinition}. For technical reasons, we also attach to each ${\bm{\phi}}$-block an auxiliary  ${\bm{\alpha}}$-block; the technical necessity of these blocks has been given in Sec.~\ref{subsec:mainresults}, but see also Def.~\ref{def:blockdecomposition} and Conj.~\ref{conj:VDUnieuqOctadDiagram}.

\begin{remark}
  The subspaces associated to a ${\bm{\phi}}$-block are both 3-dimensional and equivalent
  under the action of $\textup{Sp}(6,2)$ (see Tab.~\ref{tab:introblocksdefinition}). Thus there is a choice to be made
  about which of the subspaces one associates to a given ${\bm{\phi}}$-block. The
  constructions in this section depend on this choice, but ultimately
  Conj.~\ref{conj:VDUnieuqOctadDiagram} is being independent of such a
  choice.
\end{remark}

\begin{remark}
It seems that a $\bm{\phi}$-block is literally the sum of its auxiliary blocks: one $\bm{\alpha}$-block and two hyperelliptic blocks. Indeed, ${\bm{\phi}}_{1\mathrm{a}}^{\CAApts}$ is the sum of $\textbf{Line}^{\oA\oB\oE\oF}$, $\textbf{Line}^{\oA\oB\oG\oH}$, and ${\bm{\alpha}}_{2\mathrm{a}}^{\oA\oB\oC\oD}$. Similarly, ${\bm{\phi}}_{2\mathrm{b}}^{\oA\oB\oC|\oD\oE\oF}$ is the sum of ${\bm{\alpha}}_{1\mathrm{a}}^{\oA\oB}$, $\textbf{Line}^{\oA\oB\oC\oD}$, and $\textbf{Line}^{\oA\oB\oC\oE}$. Finally, ${\bm{\phi}}_{3\mathrm{a}}^{\oA\oB\oC\oD}$ seems to be only the sum of $\textbf{Line}^{\oA\oB\oC\oD}$ and $\alpha_{2\mathrm{a}}^{\oA\oB\oC\oD}$, but the eight points also automatically lie on a twisted cubic for this valuation data. This seems to be in line with the fact that a \texttt{(1=1)} curve is automatically hyperelliptic.
\end{remark}

\begin{proposition}\label{prop:PGLphiblocks} We have the following $\PGL$-equivalences of valuation data.\\[-0.50cm]
  \begin{itemize}
  \item[(i)] ${\bm{\phi}}_{1\mathrm{a}}^{\CAApts}$ and ${\bm{\phi}}_{1\mathrm{b}}^{\CAApts}$ are equivalent.
  \item[(ii)] $\bm{\phi}_{2\mathrm{a}}^{\oA\oB\oC|\oF\oG\oH}$,
    $\bm{\phi}_{2\mathrm{b}}^{\oA\oB\oC|\oF\oG\oH}$ and $\bm{\phi}_{2\mathrm{c}}^{\oA\oB\oC|\oF\oG\oH}$
    are equivalent.
  \item[(iii)] $\bm{\phi}_{3\mathrm{a}}^{\oA\oB\oC\oD}$ and $\bm{\phi}_{3\mathrm{b}}^{\oA\oB\oC\oD}$ are
    equivalent.
  \end{itemize}
  Moreover, an octad with these valuation data is
  $\mathrm{PGL}_3(\mathcal{O}_{\overline{K}})$-equivalent to an octad with five points in standard
  position, and restricted to such octads, up to symmetry of the partitions of
  the indices, these valuation data exhaust the two equivalence classes.
\end{proposition}

\begin{proof}
The proof goes in the same lines as for Prop.~\ref{thm:equiv-objects}.
\end{proof}

\subsubsection*{Cremona Transformation on Phi Blocks}

Similarly to $\bm{\alpha}$- and $\bm{\chi}$-blocks,  Cremona transformations
act on the $S_8$-equivalence classes of octad pictures of $\bm{\phi}$-blocks in a well defined way.

\begin{table}[htbp]
  \centering
  \renewcommand{\arraystretch}{1.3}
  \newcommand{\DoScale}[1]{\scalebox{0.6}{#1}}%
  \let\MyScale\relax
  \begin{tabular}{
    >{\centering\arraybackslash}m{0.2cm}
    >{\centering\arraybackslash\collectcell\MyScale}m{1.5cm}<{\endcollectcell}
    >{\centering\arraybackslash}m{2cm}
    >{\centering\arraybackslash\collectcell\MyScale}m{1.5cm}<{\endcollectcell}
    p{0.3cm}
    >{\centering\arraybackslash}m{0.2cm}
    >{\centering\arraybackslash\collectcell\MyScale}m{1.5cm}<{\endcollectcell}
    >{\centering\arraybackslash}m{2cm}
    >{\centering\arraybackslash\collectcell\MyScale}m{1.5cm}<{\endcollectcell}
    }
    \#
    & Picture
    & Partition
    & Image
    &
    & \#
    & Picture
    & Partition
    & Image\\
    \noalign{\global\let\MyScale\DoScale}
    \cmidrule{1-4}\cmidrule{6-9}
    \morecmidrules
    \cmidrule{1-4}\cmidrule{6-9}
    $1$
    & \tikzsetnextfilename{CremonaLeft9}
    \begin{tikzpicture}[scale=0.3]
    \ptslabelcross
    \CA
  \end{tikzpicture}
    & $\oA\oB\oE\oF \mid \oC\oD\oG\oH$
    & \tikzsetnextfilename{CremonaRight9b}
      \begin{tikzpicture}[scale=0.3]
      \ptslabel
        \CC{0}
      \end{tikzpicture}
    &
    & $6$
    & \tikzsetnextfilename{CremonaLeftC}
      \begin{tikzpicture}[scale=0.3]
        \ptslabel
        \CB{0}
      \end{tikzpicture}
    & $\oA\oC\oD\oE \mid \oB\oF\oG\oH$
    & \tikzsetnextfilename{CremonaRightD}
      \begin{tikzpicture}[scale=0.3]
        \ptslabel
        \CC{1}
      \end{tikzpicture}
      \\[0.6cm]
    \cmidrule{1-4}\cmidrule{6-9}
    $2$
    & \tikzsetnextfilename{CremonaLeft10}
      \begin{tikzpicture}[scale=0.3]
        \ptslabelcross
        \CA
      \end{tikzpicture}
    & $\oA\oB\oC\oD \mid \oE\oF\oG\oH$
    & \tikzsetnextfilename{CremonaRight10b}
    \begin{tikzpicture}[scale=0.3]
      \ptslabelcross
      \CA
    \end{tikzpicture}
    &
    & $7$
    & \tikzsetnextfilename{CremonaLeftE}
      \begin{tikzpicture}[scale=0.3]
        \ptslabel
        \CB{0}
      \end{tikzpicture}
    & $\oA\oB\oC\oD \mid \oE\oF\oG\oH$
    & \tikzsetnextfilename{CremonaRightE2}
      \begin{tikzpicture}[scale=0.3]
        \ptscustomlabel{A}{B}{C}{D}{E}{F}{G}{H}
        \CB{0}
      \end{tikzpicture}
      \\[0.6cm]
    \cmidrule{1-4}\cmidrule{6-9}
    $3$
    & \tikzsetnextfilename{CremonaLeftA}
      \begin{tikzpicture}[scale=0.3]
        \ptslabelcross
        \CA
      \end{tikzpicture}
    & $\oA\oB\oE\oH \mid \oC\oD\oF\oG$
    & \tikzsetnextfilename{CremonaRightA3}
      \begin{tikzpicture}[scale=0.3]
        \ptslabelcross
        \CA
      \end{tikzpicture}
    &
    & $8$
    & \tikzsetnextfilename{CremonaLeftF}
      \begin{tikzpicture}[scale=0.3]
        \ptslabel
        \CB{0}
      \end{tikzpicture}
    & $\oA\oB\oC\oG \mid \oD\oE\oF\oH$
    & \tikzsetnextfilename{CremonaRightF2}
      \begin{tikzpicture}[scale=0.3]
        \ptscustomlabel{A}{B}{D}{E}{G}{C}{H}{F}
        \CB{0}
      \end{tikzpicture}
      \\[0.6cm]
    \cmidrule{1-4}\cmidrule{6-9}
    $4$
    & \tikzsetnextfilename{CremonaLeftB}
      \begin{tikzpicture}[scale=0.3]
        \ptslabelcross
        \CA
      \end{tikzpicture}
    & $\oA\oC\oE\oG \mid \oB\oD\oF\oH$
    & \tikzsetnextfilename{CremonaRightB2}
      \begin{tikzpicture}[scale=0.3]
        \ptscustomlabel{D}{E}{H}{B}{C}{F}{G}{A}
        \CA
      \end{tikzpicture}
    &
    & $9$
    & \tikzsetnextfilename{CremonaLeftG}
      \begin{tikzpicture}[scale=0.3]
        \ptslabel
        \CC{0}
      \end{tikzpicture}
    & $\oA\oB\oC\oD \mid \oE\oF\oG\oH$
    & \tikzsetnextfilename{CremonaRightG}
      \begin{tikzpicture}[scale=0.3]
        \ptslabel
        \CC{0}
      \end{tikzpicture}
      \\ [0.6cm]
    \cmidrule{1-4}\cmidrule{6-9}
    $5$
    &  \tikzsetnextfilename{CremonaLeftH}
    \begin{tikzpicture}[scale=0.3]
        \ptslabelcross
        \CA
      \end{tikzpicture}
    & $\oA\oB\oC\oE \mid \oD\oF\oG\oH$
    & \tikzsetnextfilename{CremonaRightH}
    \begin{tikzpicture}[scale=0.3]
       \ptscustomlabel{D}{E}{A}{B}{F}{C}{G}{H}
        \CB{0}
      \end{tikzpicture}
    &
    &
    &
    &
    &
  \end{tabular}
  \caption{\label{table:Cremona-candies}Cremona transformations on ${\bm{\phi}}$-blocks
  }
\end{table}

\begin{proposition}\label{prop:CreTransCans}
 Suppose $O$ is an octad whose valuation data is a $\bm{\phi}$-block. The action of any Cremona transformation on $O$ gives an octad whose valuation data is also a $\bm{\phi}$-block, as described in Tab.~\ref{table:Cremona-candies}.
\end{proposition}

\begin{proof}

  Again, we only need to consider the cases listed in
  Tab.~\ref{table:Cremona-candies} and we follow the same strategy as in
  Prop.~\ref{CreTransTwins} to study each of them.
  We leave aside the details (see~\cite[file \texttt{PaperProof.m}]{BGit23}),
  except that we respectively instantiate
  ${\bm{\phi}}_{1\mathrm{a}}^{\CAApts}$,
  $\bm{\phi}_{2\mathrm{a}}^{\oA\oB\oC|\oF\oG\oH}$ and
  $\bm{\phi}_{3\mathrm{a}}^{\oA\oB\oC\oD}$ as
  \begin{displaymath}
    \arraycolsep=0.4\arraycolsep%
    \mbox{\scriptsize$\begin{bmatrix}
        1 & 0 & 0 & 0 & 1 & 1          & \pi g_1 & \pi(g_1 + h_1)\\
        0 & 1 & 0 & 0 & 1 & 1+\pi f_1 & \pi g_2 & \pi(g_2 + h_1)\\
        0 & 0 & 1 & 0 & 1 & 1+f_2      & 1 & 1 \\
        0 & 0 & 0 & 1 & 1 & 1+f_2 g_3 & g_3 & g_3+\pi h_3
      \end{bmatrix}$} ,
    \
    \mbox{\scriptsize$\begin{bmatrix}
        0 & 0 & \pi c_1 & 1 & e_1        & 0 & 1 & 1          \\
        1 & 0 & 1        & 0 & 1+\pi e_1 & 0 & 1 & 1+\pi h_1 \\
        0 & 0 & \pi c_2 & 0 & \pi e_2   & 1 & 1 & h_2        \\
        0 & 1 & c_3      & 0 & 1+\pi e_3 & 0 & 1 & 1+\pi h_3
      \end{bmatrix}$}
    \text{\footnotesize and }
    \mbox{\scriptsize$\begin{bmatrix}
        1 &    0 &    c_0 &    d_0 &    0 &    0 &    1 &    1\\
        0 &    1 &    c_1 &    d_1 &    0 &    0 &    1 &    1+\pi h_1\\
        0 &    0 &    \pi c_2 &    \pi \lambda c_2 &    1 &    0 &    1 &    h_2\\
        0 & 0 & \pi c_3 & \pi \lambda c_3 & 0 & 1 & 1 & h_3
      \end{bmatrix}$}.
  \end{displaymath}

\end{proof}

\section{Correspondence between Octad Pictures and Reduction Types}
\label{sec:CombiningOctads}

Let $C$ be a plane quartic over the ring of integers $R$ of a local field $K$, and let $O$ be a Cayley octad of $C$. The purpose of this section is to attach an \textit{octad picture}, $d_K(O)$, to $O/K$, which we conjecture to recover the dual graph of the special fibre of the stable model of $C$.

There are two challenges here. First, we must explain how to pass from the valuation data, $v$, of $O$ to the picture $d_K(O)$. This is achieved via decomposing $v$ into the building blocks of Sec.~\ref{sec:CayleyOctadsSmoothCurves}. Second, we must ensure that such a passage specifies a unique decomposition. This is the purpose of Conj.~\ref{conj:VDUnieuqOctadDiagram}.\medskip

Let us recall that a building block is (up to relabelling) one of the $19$
valuation data introduced in Sec.~\ref{sec:ocdiagrams}.

\begin{definition}
\label{def:blockdecomposition}
Given a valuation data $v$, a \textit{block decomposition} of $v$ is defined as a set $\mathcal{B} = \{B_1, \ldots, B_n, b_1, \ldots, b_k \}$, where each $B_i$ is a building block, each $b_i$ is an auxiliary block, and, for some positive $m_1, \ldots, m_n \in \Q$ and non-negative $l_1, \ldots, l_k \in \Q$,
$$
v = m_1 B_1 + \ldots + m_n B_n \, + \,  (l_1 b_{1} + \ldots + l_k b_{k})
$$
The auxiliary blocks $b_i$ all have the same octad picture.
\end{definition}

A priori, there could be several ways to decompose a valuation data as a sum of blocks.

\begin{definition}
\label{def:compatibleblocks}
Let $\mathcal{B}$ be a set of building blocks. We say that $\mathcal{B}$ is \textit{compatible} if, excluding auxiliary blocks, for all pairs $B,B' \in \mathcal{B}$, the picture induced by $B + B'$ belongs to Tab.~\ref{tab:twoblockssubspaces} (up to choice of labelling).
Two blocks inducing the exact same picture are considered compatible.
\end{definition}

Def.~\ref{def:compatibleblocks} may appear \textit{ad hoc}, but has a very natural interpretation in the setting of Sec.~\ref{sec:formalanalysisoctadpictures}.

\begin{definition}
\label{def:blockdecompositiondiagram}
Let $\mathcal{B} = \{ B_1, \ldots, B_n, b_1, \ldots, b_k \}$ be a set of compatible building blocks, with the $b_i$ auxiliary blocks. Each block induces a picture according to Tab.~\ref{tab:introblocksdefinition}. We define $d_K(\mathcal{B})$ as the picture achieved by overlaying the picture of each $B_i$. (Recall that by convention we omit auxiliary blocks from our pictures).

Such pictures are called \textit{octad pictures}. If $\mathcal{B}$ contains a hyperelliptic block or a $\bm{\phi}$-block, then $d_K(\mathcal{B})$ is called a \textit{hyperelliptic octad picture}, and if not, then it is a \textit{non-hyperelliptic octad picture}.
\end{definition}

\begin{conjecture}
  \label{conj:VDUnieuqOctadDiagram}
  Let $O$ be a Cayley octad with valuation data $v$ over a non-archimedean local
  field $K$ of residue characteristic $p \neq 2$. Then $v$ admits a
  unique compatible block decomposition up to $\PGL$-equivalence.
\end{conjecture}

In order to prove that such a decomposition always exists, one needs to prove that the cone of all possible valuation data for Cayley octads is covered by the different cones of valuation data for the different compatible sets $\mathcal{B}$. We can actually prove the uniqueness part of this conjecture.

\begin{theorem}
  \label{thm:UniqueDecomposition}
  Let $O$ be a Cayley octad with valuation data $v$ over a non-archimedean local field $K$ of residue characteristic $p \neq 2$. If $v$ admits a block decomposition, then this block decomposition is unique up to $\PGL$-equivalence.
\end{theorem}

\begin{proof}
We prove that the block decomposition into equivalence classes of building blocks is unique, if it exists. If one considers a {\bf $\phi$} block and its auxiliary block to be compatible and we ignore the {\bf TCu} building block completely (as it is completely independent of the rest anyway), then there are 2\,626\,648 compatible sets of equivalence classes of building blocks. If $R_1$ and $R_2$ are two such compatible sets, our goal is prove that the only valuation data which can be independently decomposed into blocks in $R_1$, and into block in $R_2$, are the ones consisting of blocks in $R_1 \cap R_2$.

To prove this, we look at the vector spaces generated by these blocks inside
$\Q^P / W_8$, where $P$ is the set consisting of the 70 subsets of
$\{\oA, \ldots, \oH\}$ of cardinality 4,
and $W_8$ is the 8-dimension subspace generated by the 8 vectors $v_{\pl}^S$,
where $\# S = 7$. Each equivalence class of a building block corresponds to a
unique vector in $\Q^P / W_8$. As noted in the proof of
Prop.~\ref{prop:coord-transform}, the difference between valuation data
of equivalent octads lies in $W_8$, and it suffices to prove unique
decomposition inside $\Q^P / W_8$.

For each pair $(R_1,R_2)$, we checked that there is no vector that can be
decomposed in two different ways, one with the blocks in $R_1$, and one with
the blocks in $R_2$.  More precisely, we checked that there is no linear
relation between the vectors associated to the blocks in $R_{1} \cup R_{2}$ with
only non-negative and at least one positive
sign for the vectors associated to $R_1 \backslash R_2$, or with
only non-positive and at last one negative sign for the vectors associated to
$R_2 \backslash R_1$.

This has been checked by the means of a large computation~\cite[file
\texttt{Unique\-Decomposi\-tion.m}]{BGit23} (it took a little bit under 16 CPU
hours and about 2 GB of memory). In the process, we leveraged the action of
$S_8$ to reduce our computation from an enumeration of $2\,626\,648^2$ pairs of possible $(R_1,R_2)$ to an enumeration of only $507 \times 2\,626\,648$ such pairs.

As a consequence of the non-existence of such a linear relation, we see that there does not exist a valuation data with a non-unique decomposition.
\end{proof}

\begin{remark}
In the proof of Thm.~\ref{thm:UniqueDecomposition}, the requirement on the signs in the linear relation is necessary, as there are some unexpected linear relations between some compatible sets of building blocks. For example, $\textbf{Line}^{\oA\oB\oC\oD} = \bm{\phi}_3^{\oA\oB\oC\oD} - \bm{\alpha}_2^{\oA\oB\oC\oD}$ and $\bm{\chi}_2^{\oA\oB\oC} + \bm{\chi}_2^{\oD\oE\oF} + \bm{\alpha}_1^{\oG\oH} - \bm{\phi}_2^{\oA\oB\oC|\oD\oE\oF} = \bm{\chi}_1^{\oG\oH|\oA\oB\oC|\oD\oE\oF}$. In the first example $\bm{\phi}_3^{\oA\oB\oC\oD}$ and $\bm{\alpha}_2^{\oA\oB\oC\oD}$ are compatible, but the latter occurs with a negative coefficient, so this is not a valid block decomposition. Note that $\textbf{Line}^{\oA\oB\oC\oD}$ and $\bm{\alpha}_2^{\oA\oB\oC\oD}$ are not compatible. Similarly, in the second example $\bm{\phi}_2^{\oA\oB\oC|\oD\oE\oF}$ and $\bm{\chi}_1^{\oG\oH|\oA\oB\oC|\oD\oE\oF}$ are not compatible.
\end{remark}
\medskip

\begin{definition}
\label{def:octadpictures}
Under Conj.~\ref{conj:VDUnieuqOctadDiagram}, given a Cayley octad $O$ over a non-archimedean local field with residue characteristic $p \neq 2$, define the octad picture of $O/K$ as $d_K(O) = d_K(\mathcal{B})$, where $\mathcal{B}$ is the unique compatible block decomposition of the valuation data of $O$.
\end{definition}

\begin{theorem}
\label{thm:pgl-action-on-block-decomposition}
Let $O$ and $O'$ be two normalised Cayley octads over a non-archimedean
local field $K$ of residue characteristic $p \ne 2$ which coincide under the
action of $\PGL$. Then $d_K(O) = d_K(O')$.
\end{theorem}

\begin{proof}
In Prop.~\ref{thm:equiv-objects},~\ref{PGLTypes},~\ref{prop:PGLphiblocks} and~\ref{prop:TCuequiv}, we saw that this statement is true for single blocks. For simplicity, we assume that $O'_\oA$, $O'_\oB$, $O'_\oC$, $O'_\oD$, $O'_\oE$ are in standard position. Suppose that the valuation data of $O$ is the sum of some compatible blocks, then we can apply Prop.~\ref{thm:equiv-objects},~\ref{PGLTypes},~\ref{prop:PGLphiblocks} and~\ref{prop:TCuequiv} to these individual blocks to get equivalent blocks, now normalised so that $O'_\oA$, $O'_\oB$, $O'_\oC$, $O'_\oD$, $O'_\oE$ are in standard position.
Taking the sum of these equivalent blocks and following the proof of Prop.~\ref{prop:equivalence-of-valuation-data}, it follows that the valuation data of $O'$ is the sum $v$ of these equivalent blocks, as long as not all of the three minima $\min(v_{\oA\oB\oC\oF}, v_{\oA\oB\oD\oF}, v_{\oA\oC\oD\oF}, v_{\oB\oC\oD\oF})$, $\min(v_{\oA\oB\oC\oG}, v_{\oA\oB\oD\oG}, v_{\oA\oC\oD\oG}, v_{\oB\oC\oD\oG})$ and $\min(v_{\oA\oB\oC\oH}, v_{\oA\oB\oD\oH}, v_{\oA\oC\oD\oH}, v_{\oB\oC\oD\oH})$ have become positive. We can check computationally that this never happens, by checking all 2\,626\,648 compatible sets of building blocks (not including $\mathbf{TCu}$). This is done in \cite[file \texttt{CheckMinimaCompatibleSums.m}]{BGit23}.
\end{proof}

\begin{remark}
In general, the procedure in the proof of Prop.~\ref{prop:equivalence-of-valuation-data} describing the change of valuation data under $\PGL$-transformations changing the set of points in standard position, is not additive with respect to addition of blocks. For example, ${\bm{\phi}}_{1\textrm{a}}^{\oA\oB|\oC\oD||\oE\oF|\oG\oH}$ is the sum of ${\bm{\alpha}}_{2\textrm{a}}^{\oA\oB\oC\oD}$, $\textbf{Line}^{\oA\oB\oE\oF}$, and $\textbf{Line}^{\oA\oB\oG\oH}$, and, after a $\PGL$-transformation putting $O_{\oD}, O_{\oE}, O_{\oF}, O_{\oG}$, and $O_{\oH}$ in standard position, we obtain the valuation data ${\bm{\phi}}_{1\textrm{b}}^{\oA\oB|\oC\oD||\oE\oF|\oG\oH}$.
This latter building block, however, is $\PGL$-equivalent, but not equal, to the sum ${\bm{\alpha}}_{2\textrm{b}}^{\oA\oB\oC\oD} +  \textbf{Line}^{\oA\oB\oE\oF} + \textbf{Line}^{\oA\oB\oG\oH}$.
The reason is that for this sum, all three minima have become positive. Geometrically, this can be seen to not be realisable as there is no point lying on both non-intersection lines $O_{\oE}O_{\oF}$ and $O_{\oG}O_{\oH}$.
\end{remark}

\begin{remark}
\label{rmk:exceptionalpictures}
     There is an exceptional set of octad pictures that come from compatible block decompositions containing seven $\bm{\alpha}$-blocks. Geometrically, these blocks can only come from the valuation data of a Cayley octad over a field of characteristic $2$, \textit{e.g.}\ taking the 8 vertices of a cube, and so we do not consider them further in the present work. In this case, we expect that the stable reduction might not be determined from the valuation data alone, similar to the situation of cluster pictures in characteristic 2.
\end{remark}

\begin{definition}
     Write $\mathcal{P}^\textup{NH}$ for the set of non-hyperelliptic octad pictures not amongst the exceptional pictures of Rem.~\ref{rmk:exceptionalpictures}, $\mathcal{P}^{\textup{HE}}$ for the set of octad pictures containing a $\mathbf{Line}$, $\mathbf{TCu}$, or $\bm{\phi}$-block, and $\mathcal{P}^\phi$ for the set of octad pictures containing a $\bm{\phi}$-block.
\end{definition}

There is an obvious action of the permutation group $S_8$ on octad pictures which simply permutes the points $\{ \oA, \ldots, \oH\}$. One can ignore the ordering of the points of a Cayley octad, in which case $d_K$ descends to a map (which we denote $\bar{d}_K$) whose range is $(\mathcal{P}^\textup{NH} \cup \mathcal{P}^\textup{HE}) / S_8$ and we say the image of $\bar{d}_K$ consists of \textit{unlabelled} octad pictures.

Conj.~\ref{conj:SpecialFibreOfTheStableModel} is stated in this setting,
since in practice it is easier.

\begin{conjecture}[cf.\ Conj.~\ref{conj:SpecialFibreOfTheStableModelDetailed}]\label{conj:SpecialFibreOfTheStableModel}
Let $C/K$ be a plane quartic over a non-archimedean local field $K$ of residue characteristic $p \neq 2$,  with some Cayley octad $O$. Then $\bar{d}_K(O)$ determines the dual graph of the special fibre of the stable model of $C$. Moreover, the special fibre is hyperelliptic precisely when $\bar{d}_K(O)$ is a hyperelliptic octad picture.
\end{conjecture}

    The full correspondence of Conj.~\ref{conj:SpecialFibreOfTheStableModel} is given in Tab.~\ref{tab:smoctadsc02},~\ref{tab:smoctadsc34},, and~\ref{tab:smoctadsc56} in App.~\ref{sec:special-fibres-octad-1} in the non-hyperelliptic case (resp. Tab.~\ref{tab:smhoctadsc02},~\ref{tab:smhoctadsc3},~\ref{tab:smhoctadsc4} and~\ref{tab:smhoctadsc56} in App.~\ref{sec:special-fibres-octad-2} in the hyperelliptic case). That is, we give a list of all $42$ types of stable reduction ($32$ in the hyperelliptic case) and their corresponding octad pictures.

\section{Combinatorial Analysis of Octad Pictures}
\label{sec:formalanalysisoctadpictures}

The aim of this section is to analyse the octad pictures of Sec.~\ref{sec:CombiningOctads} in a purely combinatorial manner. We give a framework whereby an octad picture corresponds to a collection of subspaces of $E_3$ (Def.~\ref{def:introE3}). In this setting the notion of compatibility (Def.~\ref{def:compatibleblocks}) is much more natural (see Def.~\ref{def:admissible}), and we can also give the space of octad pictures a structure which recovers the graph $\textup{SM}_3$. Crucially, this framework exhibits the correspondence between octad pictures and stable models more transparently than Conj.~\ref{conj:SpecialFibreOfTheStableModel}; Def.~\ref{def:stablegraph} gives a one-size-fits-all recipe for constructing the stable reduction type.

\begin{remark} \label{rem:general} For simplicity, we restrict ourselves to genus $g=3$. But all
  the results in this section %
  extend %
  to $g=2$, reobtaining known results,
  and possibly $g \geqslant 4$, see Rem.~\ref{rmk:stablegraphgeneralgenus}.
\end{remark}

\subsection{Symplectic \texorpdfstring{$\F_2$}{F2}-Vector Spaces}
\label{sec:sympl-f_2-vect}

The construction of the vector space $E_3$ (Def.~\ref{def:introE3}) generalises to the vector space $E_g$ by replacing $\Sigma_3$ with $\Sigma_g$, an alphabet of size $2g+2$.

Recall that a subspace $X$ is called isotropic if the symplectic pairing vanishes on $X \times X$. The subspaces
of $E_3$ that we associate in Sec.~\ref{sec:ocdiagrams} with fundamental
instances of valuation data are either one-dimensional or not
isotropic. This leads to the following definition.

\begin{definition}
A subspace of $E_g$ is \subsnamechoice{} if it is not isotropic of dimension greater than~1.
\end{definition}

\begin{proposition}
\label{prop:principalsubspacesE3}
Up to both the action of $S_8$ and taking orthogonal complements, there are $8$ \subsnamechoice{} subspaces of $E_3$: the subspaces $0$, $\langle \oA\oB\rangle$, $\langle \oA\oB\oC\oD \rangle$, $\langle \oA\oB, \oA\oC \rangle$, $\langle \oA\oB, \oA\oC\oD\oE \rangle$, $\langle \oA\oB, \oA\oC, \oA\oD \rangle$, $\langle \oA\oB, \oA\oC, \oD\oE \rangle$, and $\langle \oA\oB, \oC\oD, \oA\oC\oE\oF \rangle$.
\end{proposition}

\begin{proof}
This can be verified computationally.%
\end{proof}
\noindent%
For future use we include here the orthogonal complements of the previous subspaces as well as their intersections:
\begin{table}[h]
\begin{center}
\begin{tabular}{|c|c|c|}\hline
$X$ &$X^\bot$   &$X \cap X^\bot$    \\ \hline
$\mathbf{0}$    &$E_3$  &$\mathbf{0}$ \\
$\langle \oA\oB\rangle$ &$\langle \oA\oB, \oC\oD, \oC\oE,\oC\oF,\oC\oG\rangle$  &$\langle \oA\oB\rangle$  \\
$\langle \oA\oB\oC\oD \rangle$  &$\langle \oA\oB, \oA\oC, \oA\oD,\oE\oF,\oE\oG\rangle$  &$\langle \oA\oB\oC\oD\rangle$   \\
$\langle\oA\oB, \oA\oC\rangle$  &$\langle \oA\oB\oC\oD, \oD\oE,\oD\oF,\oD\oG\rangle$    &$\mathbf{0}$ \\
$\langle\oA\oB, \oA\oC\oD\oE\rangle$    &$\langle \oC\oD, \oC\oE,\oF\oG,\oF\oH\rangle$  &$\mathbf{0}$   \\
$\langle\oA\oB, \oA\oC, \oA\oD\rangle$  &$\langle\oE\oF, \oE\oG, \oE\oH\rangle$ &$\langle \oA\oB\oC\oD\rangle$   \\
$\langle\oA\oB, \oA\oC, \oD\oE\rangle$  &$\langle\oD\oE, \oF\oG, \oF\oH\rangle$ &$\langle \oD\oE\rangle$  \\
$\langle\oA\oB, \oC\oD, \oA\oC\oE\oF\rangle$       &$\langle\oE\oF, \oG\oH, \oA\oB\oE\oG\rangle$    &$\langle \oA\oB\oC\oD\rangle$  \\  \hline
\end{tabular}
\end{center}
\caption{Merotropic subspaces up to the action of $S_8$. }\label{tab:orthocomp}
\end{table}

These $7$ non-trivial subspaces are precisely the subspaces associated to $\bm{\alpha}$-blocks, $\bm{\chi}$-blocks, and $\bm{\phi}$-blocks defined in Tab.~\ref{tab:introblocksdefinition}.

\begin{definition}\label{def:admissible}
Suppose $\mathcal{L}$ is a set whose elements are pairs $(X, X^\bot)$ of orthogonally complement subspaces of $E_3$ such that $X$ is merotropic, ordered so that $\dim(X) \leqslant \dim(X^\bot)$. Then, $\mathcal{L}$ is said to be \textit{compatible} when
\begin{itemize}
    \item $(\mathbf{0}, E_3) \in \mathcal{L}$,
    \item for all distinct pairs $(X, X^\bot), (Y, Y^\bot) \in \mathcal{L}$, we have $X \subset Y, \, Y \subset X$ or $X \subset Y^\bot$, and
    \item for all distinct pairs $(X, X^\bot), (Y, Y^\bot) \in \mathcal{L} \backslash \{ (\mathbf{0}, E_3) \}$, neither $X \subset Y \cap Y^\bot$, nor $Y \subset X \cap X^\bot$.
\end{itemize}

Write $\mathfrak{L}_3$ for the set of all compatible collections of subspaces of $E_3$ not containing a pair of 3-dimensional subspaces. Write $\mathfrak{L}_3^{\bm{\phi}}$ for the set of all compatible collections of subspaces of $E_3$ containing a pair of 3-dimension subspaces.
\end{definition}

Recall that $\textup{Sp}(6,2)$ has been introduced as the group of linear transformations of $E_3$ respecting the symplectic pairing. Consequently, this action respects inclusion of subspaces and taking orthogonal complements. Hence there is an action of $\textup{Sp}(6,2)$ on $\mathfrak{L}_3$ and $\mathfrak{L}_3^{\bm{\phi}}$.

\begin{definition}
\label{def:picturesubspacecorrespondence}
Any octad picture not containing a $\mathbf{Line}$ or $\mathbf{TCu}$, induces a collection of orthogonally complement subspaces of $E_3$ as follows.

Let $P$ be an octad picture, coming from the block decomposition $\mathcal{B} = \{ B_1, \ldots, B_n\}$. Each $B_i$ induces a subspace $X_i \subset E_3$ according to Tab.~\ref{tab:introblocksdefinition}. The picture $P$ induces the collection of subspaces
$$
\varsigma(P) = \{ (X_1, X_1^\bot) , \ldots, (X_n, X_n^\bot), (\mathbf{0}, E_3) \}
$$
\end{definition}

\begin{proposition}
\label{prop:picturesubspacecorrespondence}
There is a one-to-one correspondence between octad pictures not containing a $\mathbf{Line}$ or $\mathbf{TCu}$ and compatible collections of subspaces of $E_3$, given by $\varsigma$.
\end{proposition}

\begin{proof}
    That each induced collection of subspaces is compatible amounts to checking that each picture of Tab.~\ref{tab:twoblockssubspaces} induces a compatible pair. To show that the correspondence is one-to-one, it suffices to enumerate all collections of compatible subspaces, and verify each is achieved by some octad picture. This can be done computationally.
\end{proof}

\begin{remark}
    The exceptional orbit of pictures in the statement of Prop.~\ref{prop:principalsubspacesE3} can be characterised as any compatible collection of subspaces $\mathcal{L}$ consisting of exactly $7$ pairs $(X_1,X_1^\bot)$, $\ldots$, $(X_7,X_7^\bot)$ of subspaces with $\dim(X_i) = 1$. On account of Prop.~\ref{prop:picturesubspacecorrespondence} and Conj.~\ref{conj:SpecialFibreOfTheStableModel}, we expect collections of compatible subspaces to contain information about stable models in residue characteristic $\neq 2$. We anticipate that this exceptional orbit of subspaces has meaning in the characteristic $2$ setting. Note also that this exceptional orbit is exactly the image of the exceptional octad diagrams of Rem.~\ref{rmk:exceptionalpictures} under $\varsigma$.
\end{remark}

\subsubsection*{Hyperelliptic Octad Pictures}

We can extend the map $\varsigma$ to octad pictures containing a $\mathbf{Line}$ or $\mathbf{TCu}$.

\begin{definition}
A \textit{hyperelliptic subset} $X \subset E_3$ is any subset, up to permutation of the indices by $S_8$, appearing in the third column of the bottom two rows of Tab.~\ref{tab:introblocksdefinition}.
\end{definition}

\begin{proposition}
For any $\sigma \in \textup{Sp}(6,2)$ and any hyperelliptic subset $X$, the subset $\sigma(X)$ is hyperelliptic. Moreover, up to $S_8$, there are two hyperelliptic subsets of $E_3$.
\end{proposition}

\begin{proof}
This can be verified computationally.
\end{proof}

\begin{definition}
\label{def:hyperellipticcompatiblesubspaces}
A pair $(\mathcal{L},H)$ is said to be a \textit{hyperelliptic-compatible} collection of subspaces if $\mathcal{L}$ is a compatible collection of subspaces, $H$ is a hyperelliptic subset of $E_3$, and for all pairs $(X,X^\bot) \in \mathcal{L}$ we have $X \subset H$.

We write $\mathcal{L}_3^\textup{HC}$ for the set of all hyperelliptic-compatible collections of subspaces of $E_3$.
\end{definition}

\begin{definition}
\label{def:picturesubspacescorrespondenceHE}
    Let $P$ be an octad picture containing a $\mathbf{Line}$ or $\mathbf{TCu}$. Then $P$ induces a pair $(\mathcal{L},H)$, where $\mathcal{L}$ is a compatible collection of subspaces of $E_3$ and $H$ is a hyperelliptic subset of $E_3$, as follows.

    Let $\mathcal{L}$ be the collection of subspaces induced by $\{ B_1, \ldots, B_n \}$, as in Def.~\ref{def:picturesubspacecorrespondence}. Let $H$ be the hyperelliptic subset corresponding to $B_\textup{HE}$, as in Tab.~\ref{tab:introblocksdefinition}.
    We write $\varsigma(P) = (\mathcal{L},H)$ (extending the domain of definition of $\varsigma$).
\end{definition}

\noindent%

\begin{proposition}
\label{prop:picturesubspacecorrespondenceHE}
There is a one-to-one correspondence between octad pictures containing a $\mathbf{Line}$ or $\mathbf{TCu}$ and hyperelliptic-compatible collections of subspaces of $E_3$. That is, $\varsigma$ induces a bijection
$$
\varsigma \, \colon \, \mathcal{P}^\textup{HE} \backslash \mathcal{P}^\phi \xrightarrow{\sim} \mathcal{L}_3^\textup{HC}
$$
\end{proposition}

\begin{proof}
    Let $P \in \mathcal{P}^\textup{HE}  \backslash \mathcal{P}^\phi$, and let $\varsigma(P) = (\mathcal{L},H)$. By Prop.~\ref{prop:picturesubspacecorrespondence}, $\mathcal{L}$ is a compatible collection of subspaces. If $(X, X^\bot) \in \mathcal{L}$, then verifying $(X, X^\bot)$ and $H$ are mutually compatible amounts to verifying that each of the pictures in the final three rows of Tab.~\ref{tab:twoblockssubspaces} induce a pair satisfying Def.~\ref{def:hyperellipticcompatiblesubspaces}.

    To ensure that $\varsigma$ is a bijection, we need only enumerate $\mathfrak{L}_3^\textup{HC}$ and verify that all its elements are the image of some hyperelliptic octad picture under $\varsigma$. This is achieved computationally.
\end{proof}

\subsubsection*{Action of the Symplectic Group}

By virtue of Prop.~\ref{prop:picturesubspacecorrespondence} and~\ref{prop:picturesubspacecorrespondenceHE}, one can define an action of $\textup{Sp}(6,2)$ on $\mathcal{P}^\textup{NH} \cup \mathcal{P}^\textup{HE}$, pulling back by $\varsigma$. Note that this extends the obvious action of $S_8$.

\begin{conjecture}
\label{conj:symplecticgroupactioncommutespictures}
Let $K$ be a non-archimedean local field with residue characteristic $\neq 2$. The map $d_K$ sending an octad to its octad picture commutes with the action of the group $\textup{Sp}(6,2)$.
\end{conjecture}

\begin{remark}
    By Prop.~\ref{CreTransTwins},~\ref{CreTransTypes},~\ref{prop:CreTransCans} and~\ref{prop:CreTransHE}, Conj.~\ref{conj:symplecticgroupactioncommutespictures} is true when the octad picture comes from a single block.
\end{remark}

\begin{remark}
\label{rmk:octadpictureindices}
We remark that there is no action of $\textup{Sp}(6,2)$ on
unlabelled octad pictures, as $S_8$ is not a normal subgroup of
$\textup{Sp}(6,2)$. However, we can identify the $S_8$-cosets with Cremona
transformations. The $36$ Cremona transformations applied to an unlabelled octad picture $D$ gives an `orbit' of octad pictures, but it may not consist of $36$ distinct pictures. We define the \textit{multiplicity} of an unlabelled octad picture $P$ to be the number of Cremona transformations that fix it. See the Tables in  App.~\ref{sec:special-fibres-octad-1}, and in App.~\ref{sec:special-fibres-octad-2} for the multiplicities of the different octad pictures.
\end{remark}

\subsection{Stable Graphs of Octad Pictures}
\label{sec:stable-graphs-octad}

The aim of this section is to give a general-purpose algorithm that makes the correspondence of Conj.~\ref{conj:SpecialFibreOfTheStableModel} explicit. That is, given a compatible collection of subspaces, $\mathcal{L}$, we construct a graph (the \textit{stable graph} of $\mathcal{L}$) in a purely combinatorial fashion. Conj.~\ref{conj:SpecialFibreOfTheStableModel} relates this stable graph to any plane quartic which has $P$ as one of its octad pictures, where $\varsigma(P) = \mathcal{L}$ or $(\mathcal{L}, \star)$.

This construction, given in Def.~\ref{def:stablegraph}, is somewhat intricate, and so we start with some preliminary definitions. In many cases, these definitions mimic the nomenclature surrounding cluster pictures; we aim to make these connections evident.

\begin{definition}
\label{def:inclusiongraph}
Let $\mathcal{L}$ be a collection of compatible subspaces of $E_3$. The \textit{inclusion graph} of $\mathcal{L}$ is the following directed graph:
\begin{itemize}
    \item For each $(X, X^\bot) \in \mathcal{L}$, a vertex labelled with the subspace $X$, with the exception of the pair $(\mathbf{0},E_3)$ for which we use the subspace $E_3$.
    \item An edge $X \to Y$ when $Y$ is the smallest subspace strictly containing $X$. $Y$ is said to be the \textit{parent} of $X$, and $X$ a \textit{child} of $Y$.
\end{itemize}
Note that parents must be unique. For, if $Y, Z$ are both parents of $X$, then either $Y \subset Z, Z \subset Y$ (so $Y = Z$ by minimality), or $Y \subset Z^\bot$. The third case implies $Y \subset Z \cap Z^\bot$, contradicting Def.~\ref{def:admissible}.
\end{definition}

\begin{remark}
    The purpose of the inclusion graph is to give a partial order to a compatible collection of subspaces. Note that in the case of cluster pictures there is an obvious partial order given by inclusion of clusters.
\end{remark}

 In the setting of cluster pictures, each cluster typically contributes $1$ vertex to the dual graph of the special fibre of the stable model. The exceptions to this are clusters of size two (which contribute $0$ vertices), and so-called \textit{\"{u}bereven} clusters, which contribute $2$ vertices. We define now a relation index $\delta$ that appears to be the appropriate generalisation of this phenomenon (with the non-\"{u}bereven and \"{u}bereven cases corresponding to this index being $0$ or $1$ respectively) for the octad picture.

\begin{definition}\label{def:orthogonallydependent}
    Let $\mathcal{L}$ be an admissible collection of subspaces of $E_3$, $X$ a vertex of the inclusion graph of $\mathcal{L}$, and $S = \{ Y_1 \cap {Y_1}^\bot, \ldots, Y_n \cap {Y_n}^\bot \}$, where each $Y_i$ is a child of $X$. We say that $S$ is \textit{orthogonally dependent under $X$} if:
    \begin{itemize}
        \item For each $i$, $Y_i \cap {Y_i}^\bot = \langle y_i \rangle \neq \{ 0 \}$, and
        \item $\sum_{i} y_i \in X \cap X^\bot$.
    \end{itemize}
    We define $\textup{rel}(X)$ to be the $\F_2$-vector space of all orthogonally dependent sets under $X$, with addition given by symmetric difference.  The \textit{relation index} of $X$ is $\delta_X = \dim\textup{rel}(X)$.
 \end{definition}

 \begin{remark}
    In the construction of the stable graph (Def.~\ref{def:stablegraph} and~\ref{def:stable-graph}), the integer $\delta_X$ dictates the number of vertices the subspace pair $(X,X^\bot)$ contributes.
\end{remark}

\begin{definition}
\label{def:subspacegraph}
    Let $\mathcal{L}$ be a collection of compatible subspaces of $E_3$. The \textit{subspace graph} of $\mathcal{L}$ is the graph achieved by augmenting the inclusion graph via the following process.
    For each vertex $X$, if $\delta_X > 0$, we draw a blue box around all children of $X$ which belong to some orthogonally dependent subset of $X$.
    Then, for each $X$ in the inclusion graph of $\mathcal{L}$ we replace the label $X$ with $\dim(X)$. If $(\mathcal{L},H)$ is a hyperelliptic-compatible collection of subspaces of $E_3$, then its subspace graph is the subspace graph of $\mathcal{L}$, where, instead of replacing the vertex labelled $E_3$ with $6$, we replace it with the label $6^{*}$.
\end{definition}

\begin{example}\label{ex:stable-graph}
  Let
  $\mathcal{L} = \{ (0,E_3), (X_1, X_1^\bot), (X_2, X_2^\bot), (X_3, X_3^\bot)
  \}$, where $X_1 = \langle \oA\oB \rangle, X_2 = \langle \oC\oD \rangle$ and
  $X_3 = \langle \oA\oB\oC\oD \rangle$. Then $\mathcal{L}$ is a compatible
  collection of subspaces, and its inclusion graph is given in
  Fig.~\ref{fig:inclusiongraphexample} and its subspace graph is given in
  Fig.~\ref{fig:subspacegraphexample}.

\begin{figure}[htbp]
  \begin{subfigure}[b]{0.35\linewidth}
    \centering
    \tikzsetnextfilename{inclusiongraphexample}
     \begin{tikzpicture}[scale=1]
        \node at (0,0.2) {};
        \node[shape=circle,draw=black,scale=0.5] (6) at (0,0) {$E_3$};
        \node[shape=circle,draw=black,scale=0.5] (1A) at (-0.5,-1) {$X_1$};
        \node[shape=circle,draw=black,scale=0.5] (1B) at (0,-1) {$X_2$};
        \node[shape=circle,draw=black,scale=0.5] (1C) at (0.5,-1) {$X_3$};
        \draw (6) -- (1A);
        \draw (6) -- (1B);
        \draw (6) -- (1C);
      \end{tikzpicture}
      \caption{Inclusion graph of $\mathcal{L}$}
      \label{fig:inclusiongraphexample}
  \end{subfigure}
  \vspace{\floatsep}
  \begin{subfigure}[b]{0.35\linewidth}
    \centering
    \tikzsetnextfilename{subspacegraphexample}
    \begin{tikzpicture}[scale=1]
      \node at (0,0.2) {};
      \node[shape=circle,draw=black,scale=0.5] (6) at (0,0) {6};
      \node[shape=circle,draw=black,scale=0.5] (1A) at (-0.5,-1) {1};
      \node[shape=circle,draw=black,scale=0.5] (1B) at (0,-1) {1};
      \node[shape=circle,draw=black,scale=0.5] (1C) at (0.5,-1) {1};
      \draw (6) -- (1A);
      \draw (6) -- (1B);
      \draw (6) -- (1C);
      \draw [draw = blue] (0.7,-0.8) rectangle (-0.7,-1.2);
    \end{tikzpicture}
    \caption{Subspace graph of $\mathcal{L}$}
    \label{fig:subspacegraphexample}
  \end{subfigure}
  \vspace*{-0.5cm}
  \caption{Inclusion/subspace graphs defined by Ex.~\ref{ex:stable-graph}}
\end{figure}
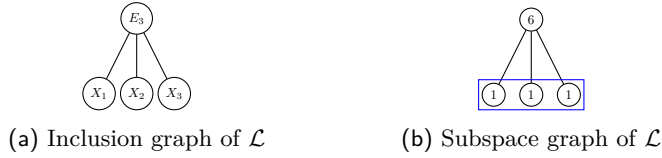

\end{example}

\begin{remark}
\label{rmk:subspacegraphdeterminesSM}
    The subspace graph appears to be weaker information than the inclusion graph. Nonetheless, the subspace graph determines the inclusion graph (up to the action of $\textup{Sp}(6,2)$). This is a coincidence as the number of possible cases for the subspace graphs is relatively small in genus 3. If one were to generalise the method to higher genus curves, it is not true anymore that the subspace graph uniquely determines the inclusion graph, see \cite[Rem.~5.1.20]{JordanPHD}.

    On account of Prop.~\ref{prop:picturesubspacecorrespondence} and~\ref{prop:picturesubspacecorrespondenceHE}, there is a correspondence between subspace graphs and octad pictures. These are given in full in Tab.~\ref{tab:smoctadsc02} -~\ref{tab:smhoctadsc56}.
\end{remark}

Now we introduce useful notation to give the recipe to obtain the dual graph of the special fibre of the stable model from the subspace graph. This complicated recipe is the equivalent to the one in \cite[Sec. 4.6]{otherm2d2} in the context of cluster pictures for hyperelliptic curves.

In the next definition, by a \textit{directed, weighted multi-graph} we mean a $4$-tuple $G =
(V,E,\gamma,\lambda)$, where $(V,E)$ is a directed multi-graph (\textit{i.e.}\ %
we allow multiple edges between the same pair of vertices), $\gamma \colon V
\to I$ is a weighting of the vertices, and $\lambda \colon E \to I'$ is a
weighting of the edges (where $I,I'$ are some label sets).
For $v \in V$ we write $\varepsilon^{-}(v)$ for the set of edges into $v$, and $\varepsilon(v)$ for the set of edges which are either into or out of $v$ (or both).

\begin{definition}
\label{def:stablegraph}
Let $\mathcal{L}$ be an octad picture. Then the \textit{inclusion multi-graph} of $\mathcal{L}$ is a directed, labelled multi-graph $G = (V_\mathcal{L}, E_\mathcal{L}, \gamma, \lambda)$, where:\medskip
\begin{description}

    \item[Vertices] For each $X$ in the inclusion graph of $\mathcal{L}$,
      there are $\delta_X + 1$ vertices $\Gamma^1_{X}, \ldots \Gamma^{\delta_X
        +1}_{X}$, \textit{i.e.}\ %
    $$
    V_\mathcal{L} = \{ \Gamma^i_X \, \colon \, X \textup{ is in the inclusion graph of $\mathcal{L}$}, \textup{ and } i = 1, \ldots, \delta_X + 1 \}\,;\medskip
    $$

    \item[Edges and their Weights] For each vertex $X$ ($\neq E_3$) in the
      inclusion graph of $\mathcal{L}$, a total of $d_X +1 = \dim(X \cap
      X^\bot)+1$ edges $e_X^{1}, \ldots, e_X^{d_X+1}$, such that each $e_X^i$
      is from some vertex $\Gamma_X^j$to some vertex $\Gamma_Y^k$, where $Y$
      is the parent of $X$, \textit{i.e.}\ %
      \begin{displaymath}
        E_\mathcal{L} = \{ e_X^{i} \, \colon \,  X \textup{ is in the inclusion
          graph of $\mathcal{L}$}, \textup{ and } i = 1, \ldots, d_X + 1 \}\,;
      \end{displaymath}\smallskip
      The function $\lambda$ is such that it has codomain $E_3$,
    \begin{itemize}
        \item  $\lambda(e^i_X)$ is the non-trivial element of $X \cap X^\bot$ (if it exists), or $0$ otherwise,
        \item for each $v \in V_\mathcal{L}$,
          \begin{math}\displaystyle
            \sum_{e \in \varepsilon(v)} \lambda(e) = 0\,,
          \end{math}
        \item for each vertex $v$, we define the collection of subspaces
        $$
          \lambda^{-}(v) = \{ \langle n_e \lambda(e) \rangle \, \colon \, e \in \varepsilon^{-}(v) \textup{, $n_e$ the number of edges in $\varepsilon^{-}(v)$ with weight $\lambda(e)$} \: \} \backslash \{ \bf{0} \}.
        $$

        Then for each $X$ in the inclusion graph and $1 \leqslant i \leqslant \delta_X + 1$, we have $\lambda^{-}(\Gamma_X^i) \in \textup{rel}(X)$
        \item For each $X$ in the inclusion graph, $\{ \lambda^{-}(\Gamma_X^1), \ldots, \lambda^{-}(\Gamma_X^{\delta_X + 1}) \}$ spans $\textup{rel}(X)$\,;
        \end{itemize}\medskip

    \item[Genera of Vertices] For each $X$ in the inclusion graph, $G$ can be partitioned into two sub-graphs, $G_X$ and $G_X^\bot$, where
    \begin{itemize}
    \item the vertices of $G_X$ are the $\Gamma_Z^{i}$ where $Z$ is below $X$ in the inclusion graph,
    \item the vertices of $G_X^{\bot}$ are the $\Gamma_Z^{i}$ where $Z$ is not below $X$ in the inclusion graph,
    \item $G_X$ and $G_X^{\bot}$ are joined by $d_X + 1$ edges\,;
    \end{itemize}\smallskip
    The function $\gamma \colon V_\mathcal{L} \to \Z^{\geqslant 0}$ satisfies
    $$
    \sum_{v \in G_X} \gamma(v) + \dim ( H_1(G_X) ) = \frac{1}{2}\dim(X / X \cap X^\bot)\,,
    $$
    $$
    \sum_{v \in G_X^{\bot}} \gamma(v) + \dim ( H_1(G_X^{\bot}) ) = \frac{1}{2}\dim(X^\bot / X \cap X^\bot)\,.
    $$
\end{description}

\end{definition}
\medskip\noindent%
In this definition, $H_1(G_X)$, resp. $H_1(G_X^{\bot})$, denotes the first homology
group of $G_X$, resp. $G_X^{\bot}$.
The idea in defining $\gamma$ is to assign each vertex a genus. For every pair
$(X, X^\bot)$ you should be able to split the graph into two bits: one
bit has total genus $1/2 \dim(X / X \cap X^\bot)$, the other one has genus
$1/2 \dim(X^\bot / X \cap X^\bot)$, and they are linked by $d_X + 1$ edges. But
to give each vertex a genus you have to subtract the genus of all its children
and all the ``extra'' genus you get from the links.

It is not \textit{a priori} clear that there is always a unique graph satisfying the conditions of Def.~\ref{def:stablegraph}. Nonetheless, it can be verified that, up to isomorphism, such a unique graph exists for every octad picture. This boils down to checking each individual case; note that Def.~\ref{def:stablegraph} is independent of the action of $\textup{Sp}(6,2)$ so in fact there are only $42$ cases to check, one from each orbit. \medskip

Given an octad picture $\mathcal{L}$ , we can recover from its
inclusion multi-graph (the dual graph of) a stable type by contracting any
vertex of genus 0 with exactly two edges into a single edge.

\begin{definition}\label{def:stable-graph}
  The \textit{stable graph} of $\mathcal{L}$ is the graph formed by removing
  all vertices $v$ from the inclusion multi-graph for which $\gamma(v) = 0$ and $|\epsilon(v)| = 2$, and replacing them with a single edge the two neighbours of $v$. Additionally we forget the direction of each edge (that is, the stable graph is not a directed graph).
\end{definition}

\begin{example}[Stable reduction type \texttt{(1=1)}]
Let $\mathcal{L} = \{ (\mathbf{0},E_3), (X,X^\bot) \}$, where $X = \langle
\oA\oB,$ $\oA\oC,$ $\oA\oD \rangle$ (\textit{i.e.}\ coming from a $\bm{\phi}_3$-block, see Tab.~\ref{tab:introblocksdefinition}). Then $\mathcal{L}$ is an octad picture, and we find its stable graph.

As both $\textup{rel}(E_3)$ and $\textup{rel}(X)$ are trivial, each contributes one vertex to the graph: $\Gamma_{E_3}^{1}$ and $\Gamma_{X}^1$. Note $X \cap X^\bot = \langle \oA\oB\oC\oD \rangle$, so there are $2$ edges in total in the graph. As $\delta_{E_3} = \delta_X = 0$, the only option is that both edges are from $\Gamma_X^{1}$ to $\Gamma_{E_3}^1$. Both edges are labelled by $\oA\oB\oC\oD$, and so $\sum_{e \in \varepsilon(\Gamma_X^{1})} \lambda(e) = 2 \cdot \oA\oB\oC\oD = 0$ (and similarly for $\Gamma_{E_3}^1$). Moreover,
\begin{displaymath}
  \lambda^{-}(\Gamma_X^1) = \lambda^{-}(\Gamma_{E_3}^1) = \{ \langle 2 \cdot \oA\oB\oC\oD \rangle \} \backslash \{ \bf{0} \} = \emptyset.
\end{displaymath}
Finally, each vertex must be given a genus. Partitioning according to $X$, we
see that $G_X$ consists of one vertex and no edges, so we must have
$\gamma(\Gamma_X^1) = \frac{1}{2}\dim(X/X\cap X^\bot) = 1$. Similarly, $G_X^\bot$
consists of one vertex, and so $\gamma(\Gamma_{E_3}^1) =
\frac{1}{2}\dim(X^\bot/X\cap X^\bot) = 1$. No vertices are required to be
contracted for the stable graph, and so we have described it completely. It is
given in Fig.~\ref{fig:stablegraphexample}.

\begin{figure}[htbp]
    \centering
    \tikzsetnextfilename{stablegraphexample}
    \begin{tikzpicture}[scale=0.7]
        \node[shape=circle,draw=black,scale=0.7] (A) at (-1,0) {1};
        \node[shape=circle,draw=black,scale=0.7] (B) at (1,0) {1};
        \draw (A) to[out=45,in=135] (B)  node [midway, above = 13pt] {$\oA\oB\oC\oD$};
        \draw (A) to[out=-45,in=-135] (B) node [midway, below = 13pt] {$\oA\oB\oC\oD$};
        \end{tikzpicture}
    \caption{Stable graph of $\mathcal{L}$}
    \label{fig:stablegraphexample}
\end{figure}
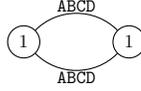
\end{example}

\begin{example}[Stable reduction type \texttt{(CAVE)}]
\label{ex:cavesubspaces}
  Let now $\mathcal{L}$ be the compatible collection
  $\{ (\mathbf{0},E_3),$ $(X_1,X_1^\bot),$ $\ldots, (X_5,X_5^\bot)\}$, where
  $X_1 = \langle \oA\oB \rangle$, $X_2 = \langle \oA\oB\oC\oD \rangle$,
  $X_3 = \langle \oC\oD \rangle$, $X_4 = \langle \oA\oB\oE\oF \rangle$,
  $X_5 = \langle \oE\oF \rangle$ (\textit{i.e.}\ coming from five ${\bm {\alpha}}$-blocks, see Tab.~\ref{tab:introblocksdefinition}).  %
The inclusion graph of $\mathcal{L}$ is given in Fig.~\ref{fig:inclusiongraphex2}.

We must have that $\textup{rel}( X_1 )$ is trivial, as $X_1$ has no children. The same goes for $X_2, X_3, X_4$ and $X_5$, and therefore each contributes exactly one vertex, $\Gamma^1_{X_i}$
However $\textup{rel}(E_3)$ is non-trivial. Indeed, it is $2$-dimensional and generated by $\{ X_1, X_2, X_3 \}$ and $\{ X_1, X_4, X_5 \}$. Thus $\delta_{E_3} = 2$, and $E_3$ contributes three vertices $\Gamma^1_{E_3}, \Gamma^2_{E_3}, \Gamma^3_{E_3}$.

Observe that $\dim( X_i \cap X_i^\bot) + 1 = 2$, for $i = 1, \ldots, 5$, so each $X_i$-vertex has two edges, $e^1_{X_i}, e^2_{X_i}$ (both labelled with the non-trivial element of $X_i \cap X_i^\bot$), with $e^j_{X_i}$ between $\Gamma^1_{X_i}$ and $\Gamma^k_{E_3}$ for some $k$ (depending on $i$).
Without loss of generality $e^1_{X_1}$ joins $\Gamma^1_{X_1}$ to $\Gamma^1_{E_3}$. But then $e^2_{X_1}$ cannot join  $\Gamma^1_{X_1}$ to $\Gamma^1_{E_3}$, as then no $\lambda^-(\Gamma^{i}_{E_3})$ can contain $\oA\oB$, and hence $\{ \lambda^-(\Gamma^{1}_{E_3}),\, \lambda^-(\Gamma^{2}_{E_3}),\, \lambda^-(\Gamma^{3}_{E_3}) \}$ cannot span $\textup{rel}(E_3)$. So without loss of generality, $e^2_{X_1}$ joins $\Gamma^1_{X_1}$ to $\Gamma^2_{E_3}$.

We must ensure that the sum of the labels into $\Gamma_{E_3}^1$ and $\Gamma_{E_3}^2$ is $0$. However given $\oA\oB$ is one such label, and there are only two relations in which $\oA\oB$ appears, our hand is forced (up to swapping $\Gamma_{E_3}^1$ and $\Gamma_{E_3}^2$). We must have that $e^1_{X_2}$ (resp. $e^1_{X_3}$) joins $\Gamma_{X_2}^1$ (resp. $\Gamma_{X_3}^1$) to $\Gamma^1_{E_3}$, and that $e^1_{X_4}$ (resp. $e^1_{X_5}$) joins $\Gamma_{X_4}^1$ (resp. $\Gamma_{X_5}^1$) to $\Gamma_{E_3}^2$. The remaining four edges $e^2_{X_i}$, $i = 2,3,4,5$ must then join $\Gamma^1_{X_i}$ to $\Gamma^3_{E_3}$. Note that already we have a graph with $3$-dimensional homology (see Fig.~\ref{fig:stablegmultiraphex2}).

To find the genera of each label, we note that $\dim(X_i/ X_i \cap X_i^\bot) = 0$, $i = 1, \ldots, 5$, so by the partitioning condition at each $X_i$ we must have $\gamma(X_i) = 0$. As we already have a graph with $3$-dimensional homology, $\gamma(\Gamma^i_{E_3}) = 0$ for $i = 1, 2, 3$.

Lastly, we contract the vertices $\Gamma^1_{X_i}$, $i = 1, \ldots, 5$. The stable graph of $\mathcal{L}$ (minus edge labels) is given in Fig.~\ref{fig:stablegraphex2}.

\begin{figure}[htbp]
  \begin{subfigure}[b]{0.3\linewidth}
    \centering\tikzsetnextfilename{inclusiongraphex2}
    \begin{tikzpicture}[scale=1.4]
      \node[shape=circle,draw=black,scale=0.6] (6) at (0,0) {$E_3$};
      \node[shape=circle,draw=black,scale=0.6] (1A) at (-1,-0.73) {$X_1$};
      \node[shape=circle,draw=black,scale=0.6] (1B) at (-0.5,-0.73) {$X_2$};
      \node[shape=circle,draw=black,scale=0.6] (1C) at (0,-0.73) {$X_3$};
      \node[shape=circle,draw=black,scale=0.6] (1D) at (0.5,-0.73) {$X_4$};
      \node[shape=circle,draw=black,scale=0.6] (1E) at (1,-0.73) {$X_5$};%
      \draw (6) -- (1A); \draw (6) -- (1B); \draw (6) -- (1C);%
      \draw (6) -- (1D); \draw (6) -- (1E);
    \end{tikzpicture}
    \caption{Inclusion graph}
    \label{fig:inclusiongraphex2}
  \end{subfigure}
  \begin{subfigure}[c]{0.3\linewidth}
    \centering\tikzsetnextfilename{stablemultigraphex2}
    \begin{tikzpicture}[scale=2.8]
      \node (A) at (-0.5,0) {$\Gamma^1_{E_3}$};
      \node (B) at (0,0.73) {$\Gamma^3_{E_3}$};
      \node (C) at (0.5,0) {$\Gamma^2_{E_3}$};
      \draw (A) to[out=80,in=220] node[fill=white,midway] {\scriptsize$\Gamma^1_{X_2}$} (B) ;%
      \draw (A) to[out=40,in=-100] node[fill=white,midway] {\scriptsize$\ \Gamma^1_{X_3}$} (B) ;%
      \draw (C) to[out=110,in=-40] node[fill=white,midway] {\ \ \scriptsize$\Gamma^1_{X_5}$} (B) ;%
      \draw (C) to[out=140,in=-80] node[fill=white,midway] {\scriptsize$\Gamma^1_{X_4}$} (B) ;%
      \draw (A) -- (C) node[fill=white,midway] {\scriptsize$\Gamma^1_{X_1}$};
    \end{tikzpicture}
    \caption{Inclusion multi-graph}
    \label{fig:stablegmultiraphex2}
  \end{subfigure}
  \begin{subfigure}[b]{0.3\linewidth}
    \centering\tikzsetnextfilename{stablegraphex2}
    \begin{tikzpicture}[scale=1.5]

      \node[shape=circle,draw=black,scale=0.7] (A) at (-0.5,0) {0};
      \node[shape=circle,draw=black,scale=0.7] (B) at (0,0.73) {0};
      \node[shape=circle,draw=black,scale=0.7] (C) at (0.5,0) {0};
      \draw (A) to[out=80,in=220] (B);
      \draw (A) to[out=40,in=-100] (B);
      \draw (C) to[out=110,in=-40] (B);
      \draw (C) to[out=140,in=-80] (B);
      \draw (A) -- (C);
    \end{tikzpicture}
    \caption{Stable graph}
    \label{fig:stablegraphex2}
  \end{subfigure}

  \caption{Graphs attached to
    $\mathcal{L}$ }
  \label{fig:cavegraph}
\end{figure}
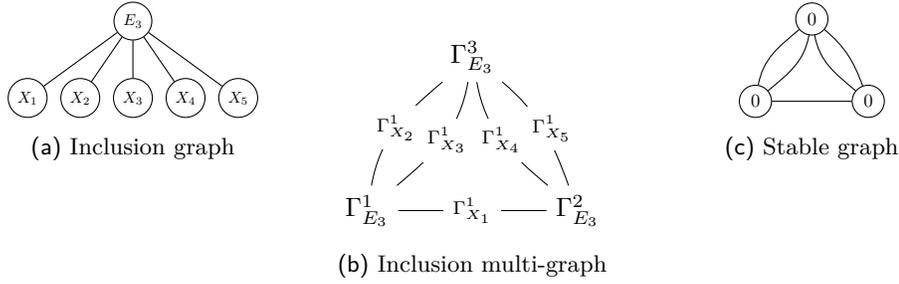

\end{example}

\begin{conjecture}\label{conj:SpecialFibreOfTheStableModelDetailed}
Let $C/K$ be a plane quartic over a non-archimedean local field $K$ of residue characteristic $p \neq 2$, with some Cayley octad $O$. Let $\bar{\mathcal{L}} = \varsigma(\bar{d}_K(O))$.

Then the stable graph of $\bar{\mathcal{L}}$ is the dual graph of the special fibre of the stable model of $C/K$, and $C/K$ has hyperelliptic reduction precisely when $\bar{d}_K(O) \in \mathcal{P}^{\textup{HE}} / S_8$.

Moreover, if $\textup{Ocs}$ is the set of normalized Cayley octads of $C$ (considered as unordered sets), then $\bar{d}_K(\textup{Ocs})$ is the image of an $\textup{Sp}(6,2)$-orbit of octad pictures under quotient by $S_8$, and the multiplicity of each octad picture $P$ in this orbit is equal to $\# \bar{d}_K^{-1}(P)$.
\end{conjecture}

Conj.~\ref{conj:SpecialFibreOfTheStableModelDetailed} is equivalent to Conj.~\ref{conj:SpecialFibreOfTheStableModel}. The methods in this section give an algorithmic way of constructing the correspondences in Tab.~~\ref{tab:smoctadsc02},~\ref{tab:smoctadsc34},~\ref{tab:smoctadsc56},~\ref{tab:smhoctadsc02},~\ref{tab:smhoctadsc3},~\ref{tab:smhoctadsc4} and~\ref{tab:smhoctadsc56} in the appendices.

\begin{remark}
\label{rmk:stablegraphgeneralgenus}
We expand briefly upon Rem.~\ref{rem:general}, in particular how Def.~\ref{def:admissible} and~\ref{def:stablegraph} may be tailored to the case $g \neq 3$.

First consider the case $g = 2$. We write $\mathfrak{L}_2$ for the set of compatible subspaces of $E_2$. In this case, there is a one-to-one correspondence between $\mathfrak{L}_2$ and cluster pictures of hyperelliptic genus $2$ curves (up to change of variable). Each $\mathcal{L} \in \mathfrak{L}_2$ produces a unique graph satisfying Def.~\ref{def:stablegraph}, and moreover this graph is precisely the dual graph of the special fibre of the stable model specified by the corresponding cluster picture.

In fact, the preceding paragraph applies \textit{mutatis mutandis} for any
cluster picture in arbitrary genus. That is, given a cluster picture $P$ for a
hyperelliptic curve $C$ of genus $g > 1$, there is a natural collection of
compatible subspaces of $E_g$ such that the graph satisfying
Def.~\ref{def:stablegraph} is the dual graph of the special fibre of the
stable model of $C$. See~\cite{JordanPHD} for details (although we caution the reader against slight changes in notation).
\end{remark}

\subsection{The Structure of Octad Pictures}
\label{subsec:structureoctadpictures}

Similarly to stable model types (see Sec.~\ref{sec:quart-stable-reduct}), it is
possible to obtain all (non-hyperelliptic) octad pictures recursively by degeneration. It
yields a directed graph $\mathcal{G}_3$  whose vertex set is $\mathcal{P}^\textup{NH} \cup \mathcal{P}^\phi$ and whose edges come in two flavours:
\begin{itemize}
\item There is an edge $D_1 \to D_2$ whenever $\varsigma(D_2) = \varsigma(D_1) \cup \{ (X, X^\bot) \}$, with $\dim(X) \leqslant 2$ and $(X, X^\bot) \notin \varsigma(D_1)$.
\item There is an edge $D_1 \to D_2$ whenever $\varsigma(D_2)$ is the result of swapping $(X, X) \in \varsigma(D_1)$ with $(Y,Y^\bot)$, where $Y \cap Y^\bot = X$.
\end{itemize}

We may define an analogous graph, $\mathcal{G}_3^\textup{HE}$ for hyperelliptic octad pictures. The vertex set of $\mathcal{G}_3^\textup{HE}$ is $\mathcal{P}^\textup{HE}$, and there is an edge $D_1 \to D_2$ whenever $\varsigma_1(D_2) = \varsigma_1(D_1) \cup \{ (X, X^\bot) \}$ with $ (X, X^\bot ) \notin \varsigma(D_1)$. Here $\varsigma_1(D)$ denotes the first component of $\varsigma(D)$ if $D \in \mathcal{P}^\textup{HE} \backslash \mathcal{P}^\phi$, and just $\varsigma(D)$ otherwise (as there is no second component in that case).
\medskip

From $\mathcal{G}_3$ (resp. $\mathcal{G}_3^{\textup{HE}}$) one can recover the structure of the Deligne--Mumford compactification of $M_3$ (resp, $H_3$).

\begin{theorem}
\label{thm:OctadDiagramGraphIsomorphisms}
Let $D$ be an octad picture (resp. hyperelliptic octad picture). Then the stable
graph of $\varsigma(D)$ is the dual graph of a stable model
(resp. hyperelliptic stable model) in genus $3$. Moreover, this gives an isomorphism of directed graphs
\begin{displaymath}
\mathcal{G}_3\, /\, \textup{Sp}(6,2) \longrightarrow \textup{SM}_3\hspace*{1cm} (\textit{resp.}\ \ \mathcal{G}^{\textup{HE}}_3 \,/\, \textup{Sp}(6,2) \longrightarrow \textup{SM}^{\textup{HE}}_3)\,.
\end{displaymath}
\end{theorem}

\begin{proof}
  This can be verified computationally.
\end{proof}

\subsection{Worked Example}
\label{sec:worked-example}

Let $C/\Q_3$ be the quartic
$$
x^3y + x^2y^2 + xy^3 + x^3z - y^3z - 2x^2z^2 + xyz^2 - 2y^2z^2 +
    2xz^3 - 2yz^3 - z^4 = 0\,.
$$
We compute a Cayley octad $O$ for $C$ via the algorithm described
in~\cite[\S\,2]{psv11}. The exact coordinates of $O$ are ungainly, and not
insightful, so we omit them here. Computation of Pl\"ucker coordinates yields
valuations
$v_{\oA\oF\oG\oH}=2$, $v_{\oE\oF\oG\oH}=6$, $v_{\oB\oE\oG\oH}=2$,
$v_{\oD\oF\oG\oH}=2$, $\ldots$ %
and the twisted cubic index is $0$. The block decomposition of this valuation data is
$$
2 {\bm \alpha}_1^{\oA\oB} + 2 {\bm \alpha}_1^{\oE\oF} + {\bm \phi}_1^{\oA\oB | \oE\oF || \oG \oH | \oC \oD}
$$
and so the octad picture $P = d_{\Q_3}(O)$ is given by Fig.~\ref{fig:octadpictureworkedexample1}.
According to Tab.~\ref{tab:smhoctadsc3} (row~5) the stable reduction type of $C/\Q_3$ is \texttt{(Z=1)$_{\mathtt{H}}$}.

\begin{figure}[htbp]
    \centering
    \tikzsetnextfilename{CandyTwinTwin}
    \begin{tikzpicture}[scale=0.25]
        \ptscustomlabel{B}{G}{H}{E}{F}{C}{D}{A}
        \CA
        \twinbig{7}
        \twinbig{3}
    \end{tikzpicture}
    \caption{Octad picture of $O/\bar{\Q}_3$}
    \label{fig:octadpictureworkedexample1}
\end{figure}
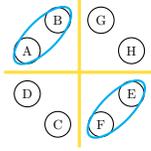

We now show how to obtain this prediction via Sec.~\ref{sec:formalanalysisoctadpictures}. The collection of compatible subspaces $\mathcal{L} = \varsigma(P)$ is given by $\mathcal{L} = \{ (\mathbf{0}, E_3), (X_1, X_1^\bot), (X_2, X_2^\bot), (X_3,X_3^\bot) \}$, where $X_1 = \langle \oA\oB \rangle, \, X_2 = \langle \oE\oF \rangle$ and $X_3 = \langle \oC\oD, \oG\oH, \oA\oB\oC\oG \rangle$. The inclusion graph of $\mathcal{L}$ is given in Fig.~\ref{fig:inclusiongraphworkedexample1}.

To find the subspace graph we must first find the relation index of each vertex in the inclusion graph. Note that $X_1, X_2, X_3$ have no children and so their relation indices are all $0$. However, we note the relation $\oA\oB + \oE\oF + \oC\oD\oG\oH = 0$, where $\oA\oB \in X_1 \cap X_1^\bot, \oE\oF \in X_2 \cap X_2^\bot$ and $\oC\oD\oG\oH \in X_3 \cap X_3^\bot$. Thus $\textup{rel}(E_3) = \{ \emptyset, \{ X_1, X_2, X_3 \} \}$. This is a $1$-dimensional $\F_2$-vector space, \textit{i.e.}\ $\delta_{E_3} = 1$. The subspace graph is thus as given in Fig.~\ref{fig:subspacegraphworkedexample1}.

We now find the stable graph $G$ of $\mathcal{L}$. The vertices $X_1, X_2$ and $X_3$ all contribute one vertex to $G$ (denoted $\Gamma_{X_1},\Gamma_{X_3},\Gamma_{X_3}$ respectively), whereas $E_3$ contributes $2 (= \delta_{E_3} + 1)$ vertices (denoted $\Gamma^1_{E_3}, \Gamma^2_{E_3}$).

Observe that $\dim(X_i \cap X_i^\bot) = 1, i = 1,2,3$, so each $X_i$ has two edges total (weighted $\oA\oB, \oE\oF$ and $\oC\oD\oG\oH$ respectively), and each joins to either $\Gamma^1_{E_3}$ or $\Gamma^2_{E_3}$.
Consider the edges into  $\Gamma^1_{E_3}$. If two $\oA\oB$ edges come into this vertex, then $\langle \oA\oB \rangle \notin \lambda^-(\Gamma^1_{E_3})$ or $\lambda^-(\Gamma^2_{E_3})$, and so in particular $\textup{rel}(E_3)$ is not spanned. Thus one $\oA\oB$ edge joins to $\Gamma^1_{E_3}$, the other to $\Gamma^2_{E_3}$. The same argument shows that the $\oE\oF$ and $\oC\oD\oG\oH$ edges are one into $\Gamma^1_{E_3}$ and one into $\Gamma^2_{E_3}$.

Now consider the genus assigned to each vertex. The $X_1$ and $X_2$ vertices have no child vertices, and so have genus $\gamma(X_i) = \frac{1}{2}\dim(X_i / X_i \cap X_i^\bot) = 0$. The $X_3$ vertex similarly satisfies $\gamma(X_3) = \frac{1}{2}\dim(X_3 / X_3 \cap X_3^\bot) = 1$. The homology of $G$ is already two-dimensional, and so the edges $\Gamma^i_{E_3}$ are assigned genus $0$. After contracting the the two vertices with genus $0$ and two edges, we are left with $G$ as in Fig.~\ref{fig:stablegraphworkedexample1}. Note that this is the dual graph of the stable reduction type \texttt{(Z=1)$_{\mathtt{H}}$}, in agreement with Tab.~\ref{tab:smhoctadsc3}.

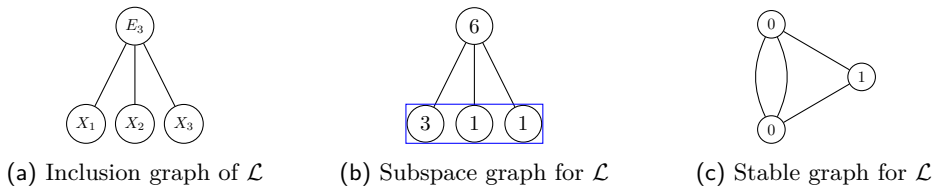
\begin{figure}[htbp]
  \begin{subfigure}[b]{0.3\linewidth}
    \centering
    \tikzsetnextfilename{CandyTwinTwinInclusionGraph}
    \begin{tikzpicture}[scale=1.3]
     \node at (0,0.2) {};
        \node[shape=circle,draw=black,scale=0.6] (6) at (0,0) {$E_3$};
        \node[shape=circle,draw=black,scale=0.6] (1A) at (-0.5,-1) {$X_1$};
        \node[shape=circle,draw=black,scale=0.6] (1B) at (0,-1) {$X_2$};
        \node[shape=circle,draw=black,scale=0.6] (1C) at (0.5,-1) {$X_3$};
        \draw (6) -- (1A);
        \draw (6) -- (1B);
        \draw (6) -- (1C);
    \end{tikzpicture}
    \caption{Inclusion graph of $\mathcal{L}$}
    \label{fig:inclusiongraphworkedexample1}
  \end{subfigure}
  \begin{subfigure}[b]{0.3\linewidth}
    \centering
    \tikzsetnextfilename{CandyTwinTwinSubspaceGraph}
    \begin{tikzpicture}[scale=1.3]
        \node at (0,0.2) {};
        \node[shape=circle,draw=black,scale=0.8] (6) at (0,0) {6};
        \node[shape=circle,draw=black,scale=0.8] (3) at (-0.5,-1) {3};
        \node[shape=circle,draw=black,scale=0.8] (1B) at (0,-1) {1};
        \node[shape=circle,draw=black,scale=0.8] (1C) at (0.5,-1) {1};
        \draw (6) -- (3);
        \draw (6) -- (1B);
        \draw (6) -- (1C);
        \draw [draw = blue] (0.7,-0.8) rectangle (-0.7,-1.2);
      \end{tikzpicture}
    \caption{Subspace graph for $\mathcal{L}$}
    \label{fig:subspacegraphworkedexample1}
  \end{subfigure}
  \begin{subfigure}[b]{0.3\linewidth}
    \centering
    \tikzsetnextfilename{StabGraphWE1}
    \begin{tikzpicture}[scale=1]
      \node[shape=circle,draw=black,scale=0.6] (A) at (-1,-0.7) {0};
      \node[shape=circle,draw=black,scale=0.6] (B) at (-1,0.7) {0};
      \node[shape=circle,draw=black,scale=0.6] (C) at (0.2,0) {1};
      \draw (A) to[out=115,in=-115] (B);
      \draw (A) to[out=65,in=-65] (B);
      \draw (A) -- (C);
      \draw (B) -- (C);
    \end{tikzpicture}
    \caption{Stable graph for $\mathcal{L}$}
    \label{fig:stablegraphworkedexample1}
  \end{subfigure}
  \caption{Graphs attached to
    $\mathcal{L}$ }
  \label{fig:graphworkedexample}
\end{figure}

\section{Computational Aspects}
\label{sec:exper-supp-conj}

Although the calculations are of a higher order of technical difficulty than the ones for determining cluster pictures of hyperelliptic curves, it turns out that, similarly to the hyperelliptic case,
the bottleneck for determining an octad picture is having to work over a big field extension: the field over which we can define the octad picture.
\begin{itemize}
\item Determining an octad of a plane quartic curve comes down to the
  computation of bitangents, which can done as in~\cite{psv11}. This requires
  especially to compute a splitting field of a polynomial of degree at most
  28, but the degree of this extension remains reasonable (about a hundred)
  when one performs the computations on a $p$-adic completion of $\Q$.
  \smallskip
\item The second step is to look at the building blocks whose valuation data
  agree with the one given by an octad, \textit{i.e.}\ we have a list of 4949 elementary
  valuation data (one for each of the $\bm{\alpha}$-blocks, $\bm{\chi}$-blocks
  and $\bm{\phi}$-blocks), and we restrict this list to vectors whose
  support is contained in the support of the valuation data of the octad.
  It remains then to determine a linear combination with non-negative coefficients
  of vectors in this list that is equal to the octad valuation data, which
  amounts to solving a linear system of small dimension over the integers.
\end{itemize}\smallskip

In these computations, we have found it convenient to normalise an octad by
some matrix in $\PGL$ so that the first five points are in the standard
position $(1:0:0:0),\ldots, (0:0:0:1), (1:1:1:1)$\,.
The orbit of an octad under $\PGL\times S_8$ restricted to this
representation becomes then finite, only 56 octads corresponding to the choice
of 5 points in the 8 points $O_\oA$, $O_\oB$, \ldots, $O_\oH$.
Moreover, the Cremona transformation of a normalised octad leads to 35
possible images, the number of ways to cut
$\{\oA,\oB,\oC,\oD,\oE,\oF,\oG,\oH\}$ into two sets of size 4.
Thus, we are reduced to consider 36 $\PGL\times S_8$-orbits, each consisting
of 56 normalised octads (see Fig.~\ref{fig:crepgl}).\smallskip

The calculation of octad pictures for different Cremona orbits allows to
easily distinguish related types of reduction, \textit{e.g.} types
\texttt{(1ne)} and \texttt{(2m)} which both have octad pictures composed of
one $\bm{\alpha}$-block and one $\bm{\chi}$-block, or types \texttt{(0nnn)}
and \texttt{(1\text{-}\text{-}\text{-}0)} which both have octad pictures
composed of three $\bm{\alpha}$-blocks.
Similarly, the calculation of octad pictures in the same $\PGL$-orbit allows
to speed up in certain cases the calculations, \textit{e.g.} it is immediate
to detect that a valuation data is of type
${\bm{\alpha}}_{2{\mathrm{a}}}^{\oA\oB\oC\oD}$ since it is the only one to
have only two positive valuations, whereas that of
${\bm{\alpha}}_{2{\mathrm{b}}}^{\oA\oB\oC\oD}$ has 53 positive valuations, and
therefore potentially corresponds to a greater number of block
decompositions.\smallskip

\begin{figure}[htbp]
  \resizebox{0.70\linewidth}{!}{%
    \tikzsetnextfilename{OctadOrbits}
    \begin{tikzpicture}[xscale=0.8]
      \node[ellipse, draw = black, minimum width = 10cm, minimum height = 1.2cm] (e) at (0,0) {};

      \node[ellipse,fill=white,text height=0.2cm, text width=1.1cm, align=center] at (e.south west) (O2) {$O_{\oA\oB\oC\oD\oF\,\oE\oG\oH}$}; %
      \node[ellipse,fill=white] at (e.south) (O3) {$O_{\oA\oB\oD\oF\oG\,\oC\oE\oH}$}; %
      \node[ellipse,fill=white] at (e.east) (O4) {$O_{\oA\oC\oE\oF\oH\,\oB\oD\oG }$}; %
      \node[ellipse,fill=white] at (e.north) (O5) {$O_{\oB\oC\oE\oF\oG\,\oA\oD\oH}$}; %

      \node[ellipse,fill=white,text height=0.2cm, text width=0.2cm] at (e.west) (O1) {$O_{\oA\oB\oC\oD\oE\,\oF\oG\oH}$}; %

      \node[dashed, ellipse, draw = black, minimum width = 10cm, minimum height = 1.2cm] (e1) at (0,-2cm) {};%
      \draw[lin,-to, thick] (O3) to (e1.center); \draw[lin, thick] (e.center) to (O3);

      \node[dashed, ellipse, draw = black, minimum width = 10cm, minimum height = 1.2cm] (e35) at (0,-5cm) {};%
      \draw[shorten <=0.2cm,shorten >=0.2cm,-to, dashed] (e1.south) to (e35.center);%
      \draw[shorten <=0.2cm,shorten >=0.2cm, dashed] (e1.center) to (e1.south);

      \node at (-8cm,0) (pgl1)     {\footnotesize  1$^{\text{rst}} \textup{PGL}$-orbit}; %
      \node at (-8cm,-2cm) (pgl2)  {\footnotesize  2$^{\text{nd}} \textup{PGL}$-orbit}; %
      \node at (-8cm,-4cm) (pgl35) {}; %
      \node at (-8cm,-5cm) (pgl36) {\footnotesize 36$^{\text{th}} \textup{PGL}$-orbit};

      \draw[-to,shorten <=0.2cm,shorten >=0.2cm] (pgl1) to node[midway,right]{\small \textsc{Cremona}$_{\oA\oB\oC\oD\,\mid\,\oE\oF\oG\oH}$} (pgl2);%
      \draw[-to] (pgl35) to node[midway,right]{\small \textsc{Cremona}$_{\oA\oF\oG\oH\,\mid\,\oB\oC\oD\oE}$} (pgl36);%
      \draw[dashed] (pgl2) to (pgl35.south);
    \end{tikzpicture}
  }
  \caption{$\textup{PGL}$ and Cremona orbits of normalised octads}
  \label{fig:crepgl}
\end{figure}
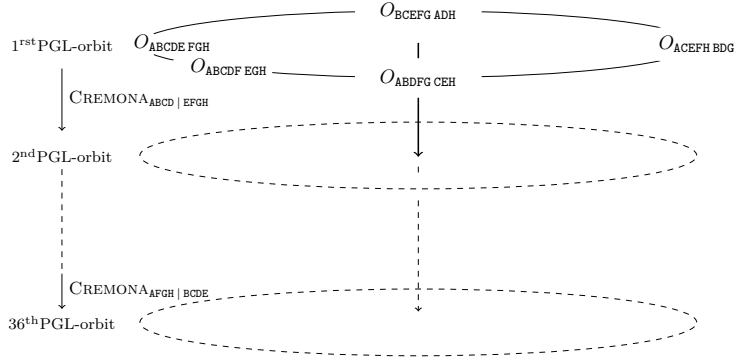

Our \MAGMA implementation is publicly available~\cite{BGit23}.
We have performed many tests, starting with quartics whose stable reduction type is
already known by other argument, particularly quartics with automorphisms such
as Picard curves~\cite{BKSW20} or Ciani curves~\cite{BCKLS20}.
But they are far from covering all possible cases. Thus, we first give in
Sec.~\ref{sec:expl-constr} constructions for all types of stable reduction.
In Sec.~\ref{sec:some-applications}, we then applied
Conj.~\ref{conj:SpecialFibreOfTheStableModel} to other quartics,
\textit{i.e.}\ quartics with large automorphism groups or complex
multiplication.
In Sec.~\ref{sec:worked-example}, we detail in a worked example the
combinatorial analysis of a non-trivial octad picture.

\subsection{Explicit Families of Plane Quartics for each Reduction Type}
\label{sec:expl-constr}

To give substance to the combinatorial viewpoint of the special fibre of the
stable models given in Fig.~\ref{fig:stablemodelgraph}, and be able to
corroborate by experiments some of our results, we now set out to give
examples of quartics that realise each of the 42 possible reduction types.

Here, we adopt Tim Dokchitser's machinery~\cite{Dokchitser21}, which in many cases
provides an easy way to obtain a regular model with normal crossings.
A central object in this approach is the 2-dimensional Newton polytope formed
by the monomials of a quartic, possibly labelled with the valuation of their
respective coefficients, the so-called $\Delta$ and $\Delta_v$ polytopes in
Dokchitser terminology.

\begin{figure}[htbp]
  \centering
  \begin{subfigure}{.4\textwidth}
    \centering
\tikzsetnextfilename{polydelta}
    \begin{tikzpicture}[xscale=0.6,yscale=0.5]
      \node[lrg] at (0,0) (1) {$1$};
      \node[lrg] at (1/2,1) (2) {$y$};
      \node[lrg] at (1,2) (3) {$y^2$};
      \node[lrg] at (3/2,3) (4) {$y^3$};
      \node[lrg] at (2,4) (5) {$y^4$};

      \node[lrg] at (1,0) (6) {$x$};
      \node[lrg] at (3/2,1) (7) {$xy$};
      \node[lrg] at (2,2) (8) {$xy^2$};
      \node[lrg] at (5/2,3) (9) {$xy^3$};
      \node[lrg] at (2,0) (10) {$x^2$};
      \node[lrg] at (5/2,1) (11) {$x^2y$};
      \node[lrg] at (3,2) (12) {$x^2y^2$};
      \node[lrg] at (3,0) (13) {$x^3$};
      \node[lrg] at (7/2,1) (14) {$x^3y$};
      \node[lrg] at (4,0) (15) {$x^4$};
    \end{tikzpicture}
    \caption{Two-dimensional polytope $\Delta$}\label{fig:polydelta}
  \end{subfigure}
  \quad\quad
  \begin{subfigure}{.4\textwidth}
    \centering
\tikzsetnextfilename{polydeltav}
    \begin{tikzpicture}[xscale=0.6,yscale=0.5]
      \node[fname] at (1.8,1.5) {$F_1$};
      \node[fname] at (2.8,0.50) {$F_2$};
      \node[lrg] at (2,4) (top) {};
      \node[lrg] at (1/2,1) (1) {0};
      \node[sml] at (3/2,1) (2) {0};
      \node[sml] at (2,2) (3) {0};
      \node[lrg] at (5/2,3) (4) {0};
      \node[lrg] at (2,0) (5) {0};
      \node[sml] at (5/2,1) (6) {1};
      \node[sml] at (3,2) (7) {2};
      \node[sml] at (3,0) (8) {3};
      \node[sml] at (7/2,1) (9) {4};
      \node[lrg] at (4,0) (10) {6};
      \draw[lin]
      (4) edge (5)
      (1) edge (4)
      (4) edge (7) (7) edge (9) (9) edge (10)
      (5) edge (8) (8) edge (10)
      (1) edge (5)
      ;
    \end{tikzpicture}
    \caption{Lower convex hull $\Delta_v$}\label{fig:polydeltav}
  \end{subfigure}%
  \caption{Newton polytopes}
  \label{fig:newton}
\end{figure}
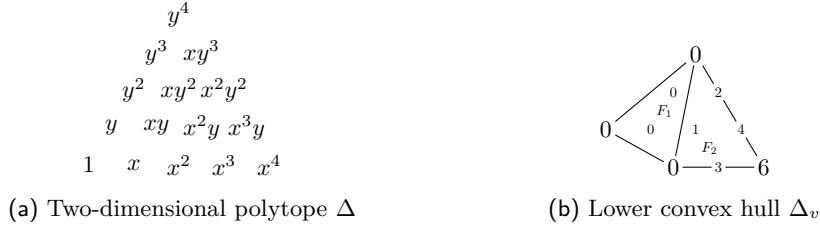

To be more precise, let $\pi$ be a uniformizer of $R$. Given a quartic $f=\sum a_{ij}x^iy^jz^{4-i-j}$, it is classical to
represent $\Delta$ as an equilateral triangle, with at each vertex the
monomials $x^4$, $y^4$ and $z^4$ (with possibly $z=1$, see
Fig.~\ref{fig:polydelta}).  We can
define $\Delta_v$, the lower convex hull that covers the three-dimensional
polytope $\{ (i, j, \nu_K(a_{ij}))\ |\ a_{ij} \neq 0\}$.
Dokchitser associates to the edges and faces of $\Delta_v$ zero and
one-dimensional schemes over $\mathcal{O}_K$ which, under certain conditions, almost completely define
the stable model (see~\cite[Th. 1.1]{Dokchitser21}). One says that $f$ is
$\Delta_v$-regular essentially if these zero- and one-dimensional schemes are smooth.
In turn, the subsets of monomials spanned by each face give equations for the
charts of the model. For instance, the  faces $F_1$ and $F_2$ defined by the
$\Delta_v$ polytope of the curve $x\,y^3+y+\pi^6\,x^4+x^2=0$ respectively
correspond to the genus 2 curve $x^2 + x\,y^3 + y = 0$ and the genus 1 curve
$x^3 + x + y^3 = 0$ (see Fig.~\ref{fig:polydeltav}), and as such the
reduction type of this curve is \texttt{(2e)}.

\begin{figure}[htbp]
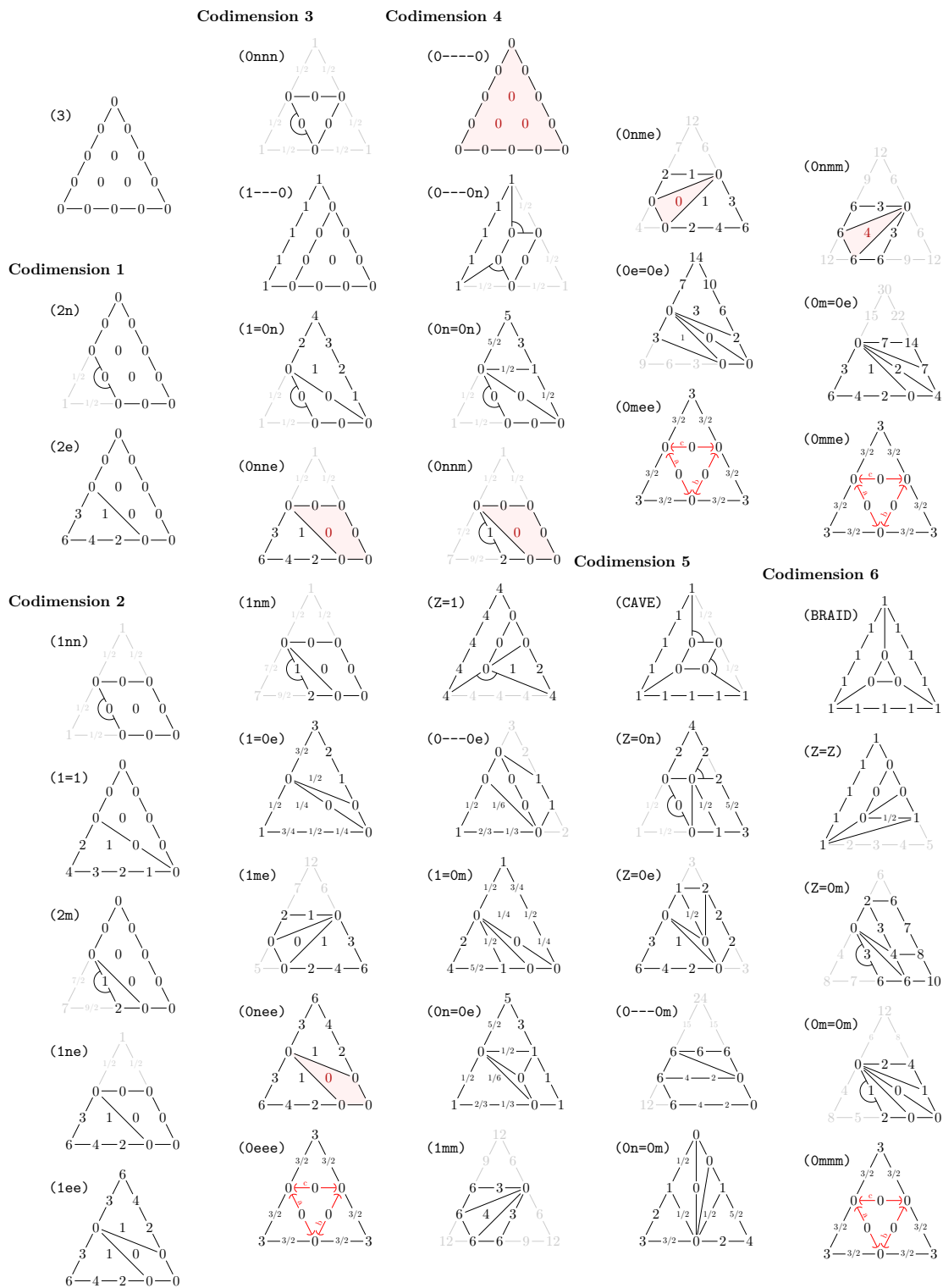

  \begin{subfigure}{0.17\linewidth}
    {\scalebox{0.7}{
\tikzsetnextfilename{stbtyp_3}
}
    \\
  \end{subfigure}
  \caption{$\Delta_v$ polytopes for each of the 42 stable reduction types (singular faces/edges are in red, see~\cite{Dokchitser21})
  }
  \label{fig:DeltaRedTypes}
\end{figure}

We apply the following \textit{modus operandi} in our situation. We try to
choose the valuations of the coefficients in such a way that the resulting
$\Delta_v$ polytope has non-contractible faces of some desired genus. Doing
so, we arrive at large families of quartics which generically realise most of
the 42 types of reduction. We show them on Fig.~\ref{fig:DeltaRedTypes}, via
their $\Delta_v$ polytopes (the interested reader can
consult~\cite[file \texttt{models.m}]{BGit23} for a version of these
polytopes in computational form).

The 32 polytopes which are generically $\Delta_v$ regular (\textit{i.e.}\ %
without any red printing)
allow to obtain quartics for which we only have to check the regularity
criteria of~\cite[Th. 1.1]{Dokchitser21}. The first ones are,
\begin{small}
  \begin{displaymath}
    \setlength{\arraycolsep}{0.5mm}
    \begin{array}{rl}
      \text{Type}&\mathtt{(2n)}:\ a_{00}\,\pi \,+\,
                   (a_{10}\,\pi\,x \,+\, a_{01}\,\pi\,y) \,+\,
                   (a_{20}\,x^2 \,+\, a_{11}\,x\,y \,+\, a_{02}\,y^2) \,+\,\\
                 &(a_{03}\,y^3 \,+\, a_{12}\,x\,y^2 \,+\, a_{21}\,x^2\,y \,+\, a_{30}\,x^3) \,+\,
                   a_{40}\,x^4 \,+\, a_{31}\,x^3\,y \,+\, a_{22}\,x^2\,y^2 \,+\, a_{13}\,x\,y^3 \,+\,
                   a_{04}\,y^4 = 0\,;\\[0.2cm]
      \text{Type}&\mathtt{(2e)}:\ a_{00}\,\pi^6 \,+\,
                   (a_{10}\,\pi^4\,x \,+\, a_{01}\,\pi^3\,y) \,+\,
                   (a_{20}\,\pi^2\,x^2 \,+\, a_{11}\,\pi\,x\,y \,+\, a_{02}\,y^2) \,+\,\\
                 &(a_{03}\,y^3 \,+\, a_{12}\,x\,y^2 \,+\, a_{21}\,x^2\,y \,+\, a_{30}\,x^3) \,+\,
                   a_{40}\,x^4 \,+\, a_{31}\,x^3\,y \,+\, a_{22}\,x^2\,y^2 \,+\, a_{13}\,x\,y^3 \,+\,
                   a_{04}\,y^4 = 0\,;\\[0.2cm]
      \text{Type}&\mathtt{(1nn)}:\ a_{00}\,\pi \,+\,
                   (a_{10}\,\pi\,x \,+\, a_{01}\,\pi\,y) \,+\,
                   (a_{20}\,x^2 \,+\, a_{11}\,x\,y \,+\, a_{02}\,y^2) \,+\,\\
                 &(a_{03}\,\pi\,y^3 \,+\, a_{12}\,x\,y^2 \,+\, a_{21}\,x^2\,y \,+\, a_{30}\,x^3) \,+\,
                   a_{40}\,x^4 \,+\, a_{31}\,x^3\,y \,+\, a_{22}\,x^2\,y^2 \,+\,
                   a_{13}\,\pi\,x\,y^3 \,+\, a_{04}\,\pi\,y^4\,;\\[0.2cm ]
      \text{Type}&\text{\texttt{(1=1)}}:\ a_{00}\,\pi^4 \,+\,
                   (a_{10}\,\pi^3\,x \,+\, a_{01}\,\pi^2\,y) \,+\,
                   (a_{20}\,\pi^2\,x^2 \,+\, \pi\,a_{11}\,x\,y \,+\, a_{02}\,y^2)\,+\,\\
                 &(a_{03}\,y^3 \,+\, a_{12}\,x\,y^2 \,+\, a_{21}\,x^2\,y \,+\, a_{30}\,\pi\,x^3)\,+\,
                   a_{40}\,x^4 \,+\, a_{31}\,x^3\,y \,+\, a_{22}\,x^2\,y^2 \,+\,
                   a_{13}\,x\,y^3 \,+\, a_{04}\,y^4.
    \end{array}
  \end{displaymath}
\end{small}

The remaining ten reduction types require special treatment.
The easiest one of these special cases is certainly the case of
\texttt{(0\text{-}\text{-}\text{-}\text{-}0)},
which can be obtained by arbitrarily lifting over $K$ the product modulo $\pi$ of
two non-degenerate conics.\smallskip

Five other cases can be obtained from a curve whose $\Delta_v$ model is
regular by making one of the components degenerate. For instance, quartics of
reduction type \texttt{(0nne)} can be obtained from quartics of type
\texttt{(1ne)} where the genus 1 component with a self-intersection becomes even more
singular. Developing the idea, we start from the $\Delta_v$ regular model
given for this type in Fig.~\ref{fig:DeltaRedTypes}, \textit{i.e.}\ %
\begin{small}
  \begin{multline*}
    a_{00}\,\pi^6 \,+\,
    (a_{10}\,\pi^4\,x \,+\, a_{01}\,\pi^3\,y) \,+\,
    (a_{20}\,\pi^2\,x^2 \,+\, a_{11}\,\pi\,x\,y \,+\, a_{02}\,y^2) \,+\,\\
    \ \ \ \ \ \ \ \ (a_{03}\,\pi\,y^3 \,+\, a_{12}\,x\,y^2 \,+\, a_{21}\,x^2\,y \,+\, a_{30}\,x^3) \,+\,
    a_{40}\,x^4 \,+\, a_{31}\,x^3\,y \,+\, a_{22}\,x^2\,y^2 \,+\, a_{13}\,\pi\,x\,y^3 \,+\, a_{04}\,\pi\,y^4 = 0\,.
  \end{multline*}
\end{small}
We have that the two genus 1 components of the special fibre are
\begin{small}
  \begin{eqnarray*}
    a_{30}\,x^3 + a_{20}\,x^2 + a_{11}\,x\,y + a_{10}\,x + a_{02}\,y^2 + a_{01}\,y + a_{00} &=& 0\,,\\
    a_{40}\,x^4 + a_{31}\,x^3\,y + a_{30}\,x^3 + a_{22}\,x^2\,y^2 + a_{21}\,x^2\,y + a_{12}\,x\,y^2 + a_{02}\,y^2 &=& 0\,.
  \end{eqnarray*}
\end{small}
The second is the one with a self-intersection, at point $(0:1:0)$. In order to make this
curve degenerate into a curve of genus 0 with 2 self-intersections and obtain finally a
curve of reduction type~\texttt{(0nne)}, it amounts generically to making its
discriminant vanish. This can be done, for example, by choosing $a_{02}$ modulo
$\pi$ as a root of a polynomial of degree 3 with coefficients in the other
$a_{ij}$'s.

Similarly, we can find curves of type \texttt{(0nee)} (resp. \texttt{(0nnm)},
\texttt{(0nme)} and \texttt{(0nmm)}) by degenerating curves of type
\texttt{(1ee)} (resp. \texttt{(1nm)}, \texttt{(1me)} and \texttt{(1mm)}).\smallskip

The most difficult case is the reduction type \texttt{(0eee)}, since  there is not enough space in an equilateral triangle to realise
it in a $\Delta_v$-regular way.
The key idea in this situation is to consider curves whose special fibre has
three cusps, without loss of generality at points $(0:0:1)$, $(0:1:0)$ and
$(1:0:0)$, and more precisely as in Fig.~\ref{fig:DeltaRedTypes} as quartics of
the form
\begin{small}
  \begin{multline*}
        a_{00}\,\pi^3 \,+\,
        (a_{10}\,\pi^2\,x \,+\, a_{01}\,\pi^2\,y) \,+\,\\
        (a_{20}\,x^2 \,+\, a_{11}\,x\,y \,+\, a_{02}\,y^2)  \,+\,
        (a_{30}\,\pi^2\,x^3 \,+\, a_{21}\,x^2\,y \,+\, a_{12}\,x\,y^2 \,+\, a_{03}\,\pi^2\,y^3)  \,+\,\\
        a_{40}\,\pi^3\,x^4 \,+\, a_{31}\,\pi^2\,x^3\,y \,+\, a_{22}\,x^2\,y^2 \,+\, a_{13}\,\pi^2\,x\,y^3 \,+\,
        a_{04}\,\pi^3\,y^4 = 0\,,
  \end{multline*}
\end{small}
with, modulo $\pi$, non-zero coefficients $a_{20}$, $a_{11}$, $a_{02}$,
$a_{12}$, $a_{22}$, $a_{21}$ and
\begin{displaymath}
  a_{11}^2-4\,a_{02}\,a_{20} = a_{21}^2-4\,a_{20}\,a_{22} =
  a_{12}^2-4\,a_{02}\,a_{22} = 0 \bmod \pi\,.
\end{displaymath}
After some Gr\"obner basis computations, we notice that the last condition
splits in two irreducible components,
\begin{multline}
  4\,a_{02}\,a_{20} - a_{11}^2 =
  2\,a_{02}\,a_{21} + a_{11}\,a_{12} =
  4\,a_{02}\,a_{22} - a_{12}^2 =\\
  a_{11}\,a_{21} + 2\,a_{12}\,a_{20} =
  2\,a_{11}\,a_{22} + a_{12}\,a_{21} =
  4\,a_{20}\,a_{22} - a_{21}^2 = 0\bmod\pi
 \label{eq:threecusps}
\end{multline}
and
\begin{multline*}
  4\,a_{02}\,a_{20} - a_{11}^2 =
  2\,a_{02}\,a_{21} - a_{11}\,a_{12} =
  4\,a_{02}\,a_{22} - a_{12}^2 =\\
  a_{11}\,a_{21} - 2\,a_{12}\,a_{20} =
  2\,a_{11}\,a_{22} - a_{12}\,a_{21} =
  4\,a_{20}\,a_{22} - a_{21}^2 = 0\bmod \pi\,.
\end{multline*}
Quartics whose coefficients satisfy the latter can be rewritten as
$(2\,a_{22}\,x\,y + a_{12}\,y\,z + a_{21}\,x\,z)^2$. Their special fibre turns
out to be the square of a conic, and has (generically) genus 3
hyperelliptic stable reduction.

It is thus the component given by Eq.~\eqref{eq:threecusps} that is
interesting. We skip details, but here, after using a method involving
blow-ups and base change~\cite[Chap. 3]{HM98}, we actually verified that we
generically have a curve of reduction type \texttt{(0eee)}.

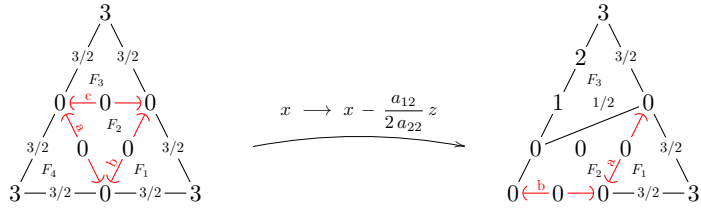
\begin{figure}[htbp]
  \centering
\tikzsetnextfilename{typ0eeea}
  \begin{tikzpicture}[xscale=0.6,yscale=0.6]
    \node[fname] at (2.8,0.50) {$F_1$};
    \node[fname] at (2.2,1.5) {$F_2$};
    \node[fname] at (1.8,2.5) {$F_3$};
    \node[fname] at (0.75,0.50) {$F_4$};
    \node[lrg] at (0,0) (1) {3};
    \node[sml] at (1/2,1) (2) {3/2};
    \node[lrg] at (1,2) (3) {0};
    \node[sml] at (3/2,3) (4) {3/2};
    \node[lrg] at (2,4) (5) {3};
    \node[sml] at (1,0) (6) {3/2};
    \node[lrg] at (3/2,1) (7) {0};
    \node[lrg] at (2,2) (8) {0};
    \node[sml] at (5/2,3) (9) {3/2};
    \node[lrg] at (2,0) (10) {0};
    \node[lrg] at (5/2,1) (11) {0};
    \node[lrg] at (3,2) (12) {0};
    \node[sml] at (3,0) (13) {3/2};
    \node[sml] at (7/2,1) (14) {3/2};
    \node[lrg] at (4,0) (15) {3};
    \draw[lin]
    (3) edge[red,(-] node[lname] {a} (7) (7) edge[red,-)] (10)
    (10) edge[red,(-] node[lname] {b} (11) (11) edge[red,-)] (12)
    (1) edge (6) (6) edge (10)
    (3) edge (4) (4) edge (5)
    (1) edge (2) (2) edge (3)
    (10) edge (13) (13) edge (15)
    (12) edge (14) (14) edge (15)
    (3) edge[red,(-] node[lname] {c} (8) (8) edge[red,-)] (12)
    (5) edge (9) (9) edge (12)
    ;
  \end{tikzpicture}
  \hspace*{-0.5cm}\scalebox{0.7}{\begin{tabular}{c}
    \xymatrix{ & \ar@/^/[rrrr]^{\displaystyle x\ \longrightarrow\ x\, -\, \frac{a_{12}}{2\,a_{22}}\,z} &&&&}\\[3cm]
  \end{tabular}}
\tikzsetnextfilename{typ0eeeb}
  \begin{tikzpicture}[xscale=0.6,yscale=0.6]
    \node[fname] at (2.8,0.50) {$F_1$};
    \node[fname] at (1.8,0.50) {$F_2$};
    \node[fname] at (1.8,2.5) {$F_3$};
    \node[lrg] at (0,0) (1) {0};
    \node[lrg] at (1/2,1) (2) {0};
    \node[lrg] at (1,2) (3) {1};
    \node[lrg] at (3/2,3) (4) {2};
    \node[lrg] at (2,4) (5) {3};
    \node[lrg] at (1,0) (6) {0};
    \node[lrg] at (3/2,1) (7) {0};
    \node[sml] at (2,2) (8) {1/2};
    \node[sml] at (5/2,3) (9) {3/2};
    \node[lrg] at (2,0) (10) {0};
    \node[lrg] at (5/2,1) (11) {0};
    \node[lrg] at (3,2) (12) {0};
    \node[sml] at (3,0) (13) {3/2};
    \node[sml] at (7/2,1) (14) {3/2};
    \node[lrg] at (4,0) (15) {3};
    \draw[lin]
    (10) edge[red,(-] node[lname] {a} (11) (11) edge[red,-)] (12)
    (1) edge[red,(-] node[lname] {b} (6) (6) edge[red,-)] (10)
    (2) edge (3) (3) edge (4) (4) edge (5)
    (1) edge (2)
    (10) edge (13) (13) edge (15)
    (12) edge (14) (14) edge (15)
    (2) edge (12)
    (5) edge (9) (9) edge (12)
    ;
  \end{tikzpicture}\vspace*{-1.5cm}
  \caption{$\Delta_v$-polytopes for type \texttt{(0eee)}}
  \label{fig:0eee}
\end{figure}

Additionally, we also remark that due to Relations~\eqref{eq:threecusps},
the change of variable $x\ \rightarrow\ x\, -\, ({a_{12}}/{2\,a_{22}})\,z$
cancels modulo $\pi$ the coefficients of $xy^2z$ and $y^2z^2$, which results
in a face of genus 1 in the $\Delta_v$ polytope (see deformation of the face
$F_3$ in Fig.~\ref{fig:0eee}). The monomials spanned by this face yields the
equation of one of the genus 1 component. Similarly, the changes of
variables $y\ \rightarrow\ y\, -\, ({a_{11}}/{2\,a_{02}})\,x$ and
$z\ \rightarrow\ z\, -\, ({a_{21}}/{2\,a_{20}})\,y$ give equations for the
other two genus 1 component. The genus 0 component is simply given by the face
at the ground (\textit{i.e.}\ face $F_2$ in the left hand side polytope in
Fig.~\ref{fig:0eee}).

Finally, it is relatively easy to choose the coefficients of the curve such
that the discriminant of these elliptic curves cancels. And this also makes it
possible to generate quartics of type \texttt{(0mee)}, \texttt{(0mme)} or
\texttt{(0mmm)}.

\subsection{Applications to Curves with Extra Automorphisms}
\label{sec:some-applications}

\subsubsection*{Large automorphism groups}

In characteristic 0, the possible automorphism groups of a quartic are well
known (see for instance~\cite{Dolgachev2012}).
The ``largest'' of them, meaning realised up to geometric isomorphism by a
single quartic, are $\CG_9$, $\Gg_{48}$, $\Gg_{96}$ and $\Gg_{168}$ where,
following~\cite{LRRS14} with the classical notation $\CG_n$ for the cyclic
group of order $n$ and $\AG_n$ for the alternating group of order $n!/2$,
$\Gg_{16}$ is a group of $16$ elements that is a direct product
$\CG_4 \times \CG_2 \times \CG_2$, $\Gg_{48}$ is a group of $48$ elements that
is a central extension of $\AG_4$ by $\CG_4$, $\Gg_{96}$ is a group of $96$
elements that is a semi-direct product $(\CG_4 \times \CG_4) \rtimes \SG_3$
and $\Gg_{168}$, is a group of $168$ elements isomorphic to
$\textup{PSL}_2 (\F_7)$.
More particularly,
\begin{itemize}
\item $ \CG_9$ can be represented by the quartic $X_{\CG_9}:~x^3 y + y^3 z + z^4 = 0$;
\item $ \Gg_{48}$ can be represented by the quartic $X_{\Gg_{48}}:~x^4 + (y^3 - z^3) z = 0$;
\item $ \Gg_{96}$ can be represented by the Fermat quartic $X_{\Gg_{96}}:~x^4 + y^4 + z^4 = 0$;
\item $ \Gg_{168}$ can be represented by the Klein quartic $X_{\Gg_{168}}:~x^3 y + y^3 z + z^3 x = 0$.
\end{itemize}

\begin{table}[htbp]
  \centering \begin{small}
    \renewcommand{\arraystretch}{0.8}
    \begin{tabular}{l|c|c|c|c|p{0.1\linewidth}}
      & $2$
      & $3$
      & $5$
      & $7$
      & \multicolumn{1}{c}{Other primes} \\
      \hline\hline
      $X_{\CG_9}$
      & \texttt{-} %
      & \texttt{(0eee)} %
      & \texttt{-} %
      & \texttt{-} %
      & \texttt{-} %
      \\
      $X_{\Gg_{48}}$
      & \texttt{(0eee)}$^\ddagger$ %
      & \texttt{(3)} %
      & \texttt{-} %
      & \texttt{-} %
      & \texttt{-} %
      \\
      $X_{\Gg_{96}}$
      & \texttt{(0eee)}$^\ddagger$ %
      & \texttt{-} %
      & \texttt{-} %
      & \texttt{-} %
      & \texttt{-} %
      \\
      $X_{\Gg_{168}}$
      & \texttt{-} %
      & \texttt{-} %
      & \texttt{-} %
      & \texttt{(3)}$_{_{\mathtt{H}}}$ %
      & \texttt{-} %
      \\
      \hline
    \end{tabular}
  \end{small}
  \caption{Reduction type of plane quartics over $\Q$ with many automorphisms}
  \label{tab:auto}
\end{table}

We give in Tab.~\ref{tab:auto} the results obtained with our reduction
criteria (where Type \texttt{(3)}$_{_{\mathtt{H}}}$ means good hyperelliptic reduction).
We indicate with a $\ddagger$
the characteristic 2 cases, which are outside the
scope of this paper, but whose reduction types can be found by other means.
In fact, our interest in these curves comes from the fact that they have
already been extensively studied.
The semistable reduction of the Klein curve is given in Elkies paper on
it~\cite{Elkies1999}. The one for the Fermat curve is given in Liu's
book~\cite[Ex. 10.4.39]{liu02}. And finally the one for the curve
$X_{\Gg_{48}}$ is a special case of Picard curves whose reduction is studied
in~\cite{BKSW20}, and can be computed with the \texttt{MCLF} package in
\texttt{SageMath}~\cite{RW2017}.
Note that all these reductions correspond to complex multiplication curves
(see Rem.~\ref{rem:cmred}).

\subsubsection*{Complex multiplication}

Let us now consider the 19 plane quartics defined over $\Q$ with complex
multiplication by a maximal order, defined and numbered as
in~\cite[Section~5]{KLLRSS18}.
At the time of writing, reduction types of these curves modulo prime divisors
of the minimal discriminant is only partially understood.
By using the criteria based on the invariants of~\cite{lllr21}, we know for
example which are the primes greater than 7 for which the reduction is still a
smooth quartic, Type \texttt{(3)}, or a smooth hyperelliptic curve, Type
\texttt{(3)}$_{_{\mathtt{H}}}$. Considerations on their CM order also allows in
some cases to determine if the reduction is a smooth hyperelliptic curve for
$p < 7$ (see~\cite[Prop.~4.1]{KLLRSS18} and~\cite[Prop. 6.1]{IKLLMV22}).

\begin{table}[H]
  \centering \begin{small}
    \renewcommand{\arraystretch}{0.8}
    \begin{tabular}{c|c|c|c|c|p{0.58\linewidth}}
      & $2$
      & $3$
      & $5$
      & $7$
      & \multicolumn{1}{c}{Other primes} \\
      \hline\hline
      $X_{1}$
      & \texttt{(3)}$_{_{\mathtt{H}}}^\ddagger$
      & \texttt{-}
      & \texttt{(2e)}
      & \texttt{(0eee)}
      & $13$~\texttt{(3)}, %
        $37$~\texttt{(3)}$_{_{\mathtt{H}}}$, %
        $15187$~\texttt{(3)}$_{_{\mathtt{H}}}$
      \\
      $X_{2}$
      & \texttt{(3)}$_{_{\mathtt{H}}}^\ddagger$
      & \texttt{(1ee)}
      & \texttt{-}
      & \texttt{(0eee)}
      & $701$~\texttt{(3)}$_{_{\mathtt{H}}}$ %
      \\
      $X_{3}$
      & \texttt{(3)}$_{_{\mathtt{H}}}^\ddagger$
      & \texttt{(0eee)}
      & \texttt{(1ee)}
      & \texttt{(3)}$_{_{\mathtt{H}}}$
      & $31$~\texttt{(3)}, %
        $233$~\texttt{(3)}$_{_{\mathtt{H}}}$, %
        $356399$~\texttt{(3)}$_{_{\mathtt{H}}}$\\
      $X_{5}$
      & \texttt{(3)}$_{_{\mathtt{H}}}^\ddagger$
      & \texttt{(1ee)}
      & \texttt{-}
      & \texttt{(3)}$_{_{\mathtt{H}}}$
      & $13$~\texttt{(3)}, %
        $37$~\texttt{(3)}$_{_{\mathtt{H}}}$, %
        $127$~\texttt{(3)}$_{_{\mathtt{H}}}$\\
      $X_{6}$
      & \texttt{(3)}$_{_{\mathtt{H}}}^\ddagger$
      & \texttt{(1ee)}
      & \texttt{-}
      & \texttt{(3)}$_{_{\mathtt{H}}}$
      & $17$~\texttt{(2e)}, %
        $19$~\texttt{(3)}, %
        $127$~\texttt{(3)}$_{_{\mathtt{H}}}$, %
        $211$~\texttt{(3)}$_{_{\mathtt{H}}}$, %
        $20707$~\texttt{(3)}$_{_{\mathtt{H}}}$\\
      $X_{7}$
      & \texttt{(3)}$_{_{\mathtt{H}}}^\ddagger$
      & \texttt{(0eee)}
      & \texttt{-}
      & \texttt{(3)}$_{_{\mathtt{H}}}$
      & $71$~\texttt{(3)}$_{_{\mathtt{H}}}$, %
        $73$~\texttt{(3)}, %
        $83$~\texttt{(2e)}, %
        $17665559$~\texttt{(3)}$_{_{\mathtt{H}}}$\\
      $X_{8}$
      & \texttt{(3)}$_{_{\mathtt{H}}}^\ddagger$
      & \texttt{-}
      & \texttt{-}
      &  \texttt{(0eee)}
      & $19$~\texttt{(3)}, %
        $499$~\texttt{(3)}$_{_{\mathtt{H}}}$\\
      $X_{9}$
      & \texttt{?}
      & \texttt{-}
      & \texttt{(2e)}
      & \texttt{(3)}$_{_{\mathtt{H}}}$
      & $13$~\texttt{(3)}, %
        $79$~\texttt{(3)}$_{_{\mathtt{H}}}$, %
        $233$~\texttt{(3)}$_{_{\mathtt{H}}}$, %
        $857$~\texttt{(3)}$_{_{\mathtt{H}}}$\\
      $X_{10}$
      & \texttt{?}
      & \texttt{-}
      & \texttt{-}
      & \texttt{\texttt{}(3)}$_{_{\mathtt{H}}}$
      & $41$~\texttt{\texttt{}(3)}$_{_{\mathtt{H}}}$, %
        $71$~\texttt{\texttt{}(3)}$_{_{\mathtt{H}}}$\\
      $X_{11}$
      & \texttt{?}
      & \texttt{-}
      & \texttt{-}
      & \texttt{(3)}$_{_{\mathtt{H}}}$
      & $23$~\texttt{(3)}$_{_{\mathtt{H}}}$, %
        $31$~\texttt{(3)}, %
        $47$~\texttt{(3)}$_{_{\mathtt{H}}}$, %
        $27527$~\texttt{(3)}$_{_{\mathtt{H}}}$\\
      $X_{12}$
      & \texttt{-}
      & \texttt{-}
      & \texttt{-}
      & \texttt{(3)}$_{_{\mathtt{H}}}$
      & $11$~\texttt{(0eee)}, %
        $5711$~\texttt{(3)}$_{_{\mathtt{H}}}$, %
        $73064203493$~\texttt{(3)}$_{_{\mathtt{H}}}$\\
      $X_{13}$
      & \texttt{?}
      & \texttt{-}
      & \texttt{-}
      & \texttt{-}
      & $11$~\texttt{(0eee)}, %
        $43$~\texttt{(3)}, %
        $547$~\texttt{(3)}$_{_{\mathtt{H}}}$, %
        $11827$~\texttt{(3)}$_{_{\mathtt{H}}}$, %
        $189169$~\texttt{(3)}$_{_{\mathtt{H}}}$\\
      $X_{14}$
      & \texttt{-}
      & \texttt{-}
      & \texttt{-}
      & \texttt{-}
      & $11$~\texttt{(2e)}$_{_{\mathtt{H}}}$, %
        $19$~\texttt{(3)}, %
        $101$~\texttt{(3)}$_{_{\mathtt{H}}}$, %
        $107$~\texttt{(3)}$_{_{\mathtt{H}}}$, %
        $8378707$~\texttt{(3)}$_{_{\mathtt{H}}}$\\
      $X_{15}$
      & \texttt{-}
      & \texttt{-}
      & \texttt{-}
      & \texttt{-}
      & $19$~\texttt{(3)}$_{_{\mathtt{H}}}$\\
      $X_{16}$
      & \texttt{-}
      & \texttt{(1ee)}
      & \texttt{-}
      & \texttt{-}
      & $19$~\texttt{(3)}$_{_{\mathtt{H}}}$, %
        $37$~\texttt{(3)}$_{_{\mathtt{H}}}$, %
        $79$~\texttt{(3)}$_{_{\mathtt{H}}}$, %
        $13373064392147$~\texttt{(3)}$_{_{\mathtt{H}}}$\\
      $X_{17}$
      & \texttt{?}
      & \texttt{(1ee)}$_{_{\mathtt{H}}}$
      & \texttt{-}
      & \texttt{-}
      & $19$~\texttt{(3)}$_{_{\mathtt{H}}}$, %
        $1229$~\texttt{(3)}$_{_{\mathtt{H}}}$, %
        $3913841117$~\texttt{(3)}$_{_{\mathtt{H}}}$\\
      $X_{18}$
      & \texttt{?}
      & \texttt{-}
      & \texttt{-}
      & \texttt{-}
      & $13$~\texttt{(3)}, %
        $19$~\texttt{(3)}$_{_{\mathtt{H}}}$, %
        $101$~\texttt{(3)}$_{_{\mathtt{H}}}$, %
        $251$~\texttt{(3)}$_{_{\mathtt{H}}}$, %
        $7468843725186901$~\texttt{(3)}$_{_{\mathtt{H}}}$\\
      $X_{19}$
      & \texttt{?}
      & \texttt{-}
      & \texttt{-}
      & \texttt{-}
      & $11$~\texttt{(3)}$_{_{\mathtt{H}}}$, %
        $43$~\texttt{(3)}$_{_{\mathtt{H}}}$\\
      $X_{20}$
      & \texttt{-}
      & \texttt{-}
      & \texttt{-}
      & \texttt{-}
      & $67$~\texttt{(3)}$_{_{\mathtt{H}}}$, %
        $1439$~\texttt{(3)}$_{_{\mathtt{H}}}$, %
        $2739021126001$~\texttt{(3)}$_{_{\mathtt{H}}}$\\
      \hline
    \end{tabular}
  \end{small}
  \caption{Reduction type of quartics over $\Q$ with CM by a maximal order}
  \label{tab:genus3cm}
\end{table}

We give in Tab.~\ref{tab:genus3cm} the results of our reduction criteria. We
notice that they are consistent with what we already knew (we indicate with a
$\ddagger$ where CM order considerations can be used).

\begin{remark}\label{rem:cmred}
  The expected type is very
  special, \textit{i.e.}\ \texttt{(3)} or \texttt{(3)}$_{_{\mathtt{H}}}$, \texttt{(2e)} or \texttt{(2e)}$_{_{\mathtt{H}}}$,
  \texttt{(1ee)} or \texttt{(1ee)}$_{_{\mathtt{H}}}$ and \texttt{(0eee)}.
  This is consistent with what is expected for curves with CM
  (see~\cite{BCLLMNO2015}).
\end{remark}

\vspace*{2cm}

\appendix

\section{Special fibres and octad pictures correspondence}

This appendix gives a complete account of several constructions given in the paper, along with the correspondences between them. In particular, catalogued are:
\begin{itemize}
    \item All stable reduction types (Thm.~\ref{thm:M3nonH}, \ref{thm:M3H}), under the heading \textit{stable model}, split into the cases non-hyperelliptic (App.~\ref{sec:special-fibres-octad-1}) and hyperelliptic (App.~\ref{sec:special-fibres-octad-2}), and ordered by codimension.
    \item All octad pictures (Def.~\ref{def:octadpictures}) outside the exceptional orbit (see Rmk.~\ref{rmk:exceptionalpictures}).
    \item All subspace graphs (Def.~\ref{def:subspacegraph}) besides that of the exceptional orbit.
\end{itemize}
The rows of Tab.~\ref{tab:smoctadsc02}--\ref{tab:smhoctadsc56} give the correspondences between:
\begin{itemize}
    \item Octad pictures and subspace graphs (see Rmk.~\ref{rmk:subspacegraphdeterminesSM}).
    \item Octad pictures and stable reduction types (and their dual graphs), as detailed in Def.~\ref{def:stable-graph} (see also Def.~\ref{def:inclusiongraph}--\ref{def:stablegraph}, and Prop.~\ref{prop:picturesubspacecorrespondence},~\ref{prop:picturesubspacecorrespondenceHE}).
\end{itemize}
In Tab.~\ref{tab:smoctadsc02}--\ref{tab:smhoctadsc56}, the column \textit{dual graph} is understood to be the dual graph of the special fibre of the corresponding stable model, and the number under each octad picture is its multiplicity in the $\textup{Sp}(6,2)$-orbit, up to $S_8$ (Rmk.~\ref{rmk:octadpictureindices}). 

Assuming Conj.~\ref{conj:VDUnieuqOctadDiagram} and \ref{conj:SpecialFibreOfTheStableModelDetailed}, these tables can be used in the following manner. Let $C/K$ be a plane quartic curve over a non-archimedean local field with residue characteristic $p \neq 2$, and let $O$ be a Cayley octad of $C$. Then $\bar{d}_K(O)$ belongs to one of the tables, and the corresponding stable model is the stable reduction type of $C/K$. Moreover, the $36$ octad pictures of $C/K$ are given by the orbit to which $\bar{d}_K(O)$ belongs.

\newpage

\subsection{Non-hyperelliptic reduction}
\label{sec:special-fibres-octad-1}
\mbox{}

\newcommand\Codimension[2]{
  \multicolumn{4}{c}{
    \begin{tabular}{p{0.38\hsize}p{0.22\hsize}p{0.46\hsize}}
      \hrule width \hsize \kern 2pt \hrule width \hsize
      & Codimension~$#1$\ ~\ ~#2
      &\hrule width \hsize \kern 2pt \hrule width \hsize
    \end{tabular}
  }
}

\begin{table}[H]
  \hspace*{-1cm} \small
  \renewcommand{\arraystretch}{1.0}

  \caption{Stable models (non-hyperelliptic), their dual graphs and
    corresponding subspace graphs and octad pictures (codim. 0, 1, 2, and 3 part 1)}
  \label{tab:smoctadsc02}
\end{table}

\begin{table}[H]
  \hspace*{-1cm} \small
  \renewcommand{\arraystretch}{1.2}
  \begin{tabular}{
    >{\centering\arraybackslash}m{1cm}
    >{\centering\arraybackslash}m{2cm}
    >{\centering\arraybackslash}m{1.5cm}
    >{\centering\arraybackslash}m{9cm}
    }

    Stable Model
    & \parbox{1cm}{\centering Dual Graph}
    & Subspace Graph
    & \parbox{4.5cm}{\centering  Octad Pictures \ \ \ \ \ \ \ \ (unlabelled $\textup{Sp}(6,2)$-orbit)} \\

    \Codimension{3}{(cont'd)} \\
\end{tabular}
  \caption{Stable models (non-hyperelliptic), their dual graphs and
    corresponding subspace graphs and octad pictures (codim. 3 part 2 and 4 part 1)}
  \label{tab:smoctadsc34}
\end{table}

\begin{table}[H]
  \hspace*{-1cm} \small
  \renewcommand{\arraystretch}{1.0}
  %
  \caption{Stable models (non-hyperelliptic), their dual graphs and
    corresponding subspace graphs and octad pictures (codim. 4 part 2, 5 and
    6)}
  \label{tab:smoctadsc56}
\end{table}

\subsection{Hyperelliptic reduction}
\label{sec:special-fibres-octad-2}
\mbox{}
\begin{table}[H]
  \hspace*{-1cm} \small
  \renewcommand{\arraystretch}{1.2}
  \begin{tabular} {
    >{\centering\arraybackslash}m{1cm}
    >{\centering\arraybackslash}m{2cm}
    >{\centering\arraybackslash}m{1.5cm}
    >{\centering\arraybackslash}m{9cm}
    }

    Stable Model
    & \parbox{1cm}{\centering Dual Graph}
    & Subspace Graph
    & \parbox{4.5cm}{\centering  Octad Pictures \ \ \ \ \ \ \ \ (unlabelled $\textup{Sp}(6,2)$-orbit)} \\[-0.1cm]

    \Codimension{0}{} \\[-0.4cm]
    \tikzsetnextfilename{TABHOctadT}
    \StbTypeHyp{\T}{3}
    &\tikzsetnextfilename{TABHOctadTAGraph}
      \begin{tikzpicture}[scale=0.4]
        \node[shape=circle,draw=black,scale=0.6] (3) at (0,0) {3};
      \end{tikzpicture}
    &\tikzsetnextfilename{TABHOctadTASpace}
      \begin{tikzpicture}[scale=0.8]
        \node at (0,0.2) {};
        \node[shape=circle,draw=black,scale=0.6] (6) at (0,0) {6*};
      \end{tikzpicture}
    &\tikzsetnextfilename{TABHOctadTA}
      \begin{tikzpicture}[scale=0.12]
        \pts
        \TCu
        \oindex{1}
      \end{tikzpicture}\quad
      \tikzsetnextfilename{TABHOctadTB}
      \begin{tikzpicture}[scale=0.12]
        \pts
        \HE{0}
        \oindex{35}
      \end{tikzpicture}\\

    \Codimension{1}{} \\
\end{tabular}
  \caption{Stable models (hyperelliptic), their dual graphs and corresponding
    subspace graphs and octad pictures (codim. 0, 1 and 2 part 1)
  }
  \label{tab:smhoctadsc02}
\end{table}

\begin{table}[H]
  \hspace*{-1cm} \small
  \renewcommand{\arraystretch}{1.2}
  %
  \caption{Stable models (hyperelliptic), their dual graphs and corresponding
    subspace graphs and octad pictures (codim. 2 part 2 and 3 part 1)}
  \label{tab:smhoctadsc3}
\end{table}

\begin{table}[H]
  \hspace*{-1cm} \small
  \renewcommand{\arraystretch}{1.2}
  %
  \caption{Stable models (hyperelliptic), their dual graphs and corresponding
    subspace graphs and octad pictures (codim. 3 part 2 and 4 part 1)
  }
  \label{tab:smhoctadsc4}
\end{table}

\begin{table}[H]
  \hspace*{-1cm} \small
  \renewcommand{\arraystretch}{1.2}
  %
  \end{tabular}
  \caption{Stable models (hyperelliptic), their dual graphs and corresponding
    subspace graphs and octad pictures (codim. 4 part 2 and 5)
    }
  \label{tab:smhoctadsc56}
\end{table}

\printbibliography

\end{document}